\documentclass{amsart}
\usepackage[english]{babel}
\usepackage[utf8]{inputenc}
\usepackage[T1]{fontenc}    
\usepackage{hyperref}
\usepackage{amsthm}
\usepackage{amsmath}
\usepackage{amssymb}
\usepackage{mathrsfs}
\usepackage{graphicx}
\usepackage{subcaption}
\usepackage{stmaryrd}
\usepackage{dsfont}
\usepackage{yfonts}
\usepackage[margin = 2.5cm]{geometry}
\usepackage{caption}
\usepackage{ulem}
\usepackage{caption}
\usepackage{blkarray}
\usepackage{amsmath,amsfonts}
\usepackage[all]{xy}
\usepackage{tikz}
\usepackage{enumerate}
\usepackage[Algorithme]{algorithm}
\usepackage{tikz-cd}
\usepackage{varioref}
\usepackage{nicematrix}
\usepackage{bm}
\usepackage[toc,page]{appendix}
\usepackage{calligra}

\newcommand{\Tr}{\mathrm{Tr\ }}

\numberwithin{equation}{subsection}

\newtheorem{theorem}{Theorem}[section]
\newtheorem{definition}{Definition}[section]
\newtheorem{prop}{Proposition}[section]
\newtheorem{corollaire}{Corollary}[section]
\newtheorem{lemma}{Lemma}[section]

\newtheorem*{remark}{Remark}
\newtheorem*{assumption}{Assumption}

\newtheorem{mainthm}{Theorem}

\newtheorem{mainconj}{Conjecture}

\labelformat{mainthm}{Theorem~#1}
\labelformat{maincor}{Corollary~#1}
\labelformat{mainconj}{Conjecture~#1}

\labelformat{lemma}{Lemma~#1}
\labelformat{theorem}{Theorem~#1}
\labelformat{prop}{Proposition~#1}
\labelformat{corollaire}{Corollary~#1}
\labelformat{definition}{Definition~#1}

\newcommand{\N}{\mathbb{N}}
\newcommand{\Z}{\mathbb{Z}}
\newcommand{\R}{\mathbb{R}}
\newcommand{\Q}{\mathbb{Q}}

\newcommand{\C}{\mathbb{C}}
\newcommand{\A}{\mathbb{A}}
\def\FF{\mathbb{F}}
\def\TT{\mathrm{\mathbf{T}}}

\def\W{\mathcal{W}}
\def\O{\mathcal{O}}

\def\L{\mathcal{L}}
\def\H{\mathcal{H}}
\def\P{\mathcal{P}}
\def\K{\mathcal{K}}

\newcommand{\toeq}{\buildrel\sim\over\rightarrow}
\newcommand{\bs}{\backslash}
\newcommand{\surj}{\twoheadrightarrow}
\newcommand{\inj}{\hookrightarrow}

\renewcommand{\a}{\alpha}

\renewcommand{\d}{\delta}
\newcommand{\g}{\mathfrak{g}}
\newcommand{\e}{\varepsilon}
\newcommand{\w}{\omega}
\newcommand{\ph}{\varphi}
\newcommand{\s}{\sigma}
\newcommand{\vs}{\varsigma}
\newcommand{\ga}{\gamma}

\def\Om{\Omega}
\def\Si{\Sigma}
\def\t{\tau}
\def\l{\lambda}

\def\i{\iota}

\def\G{\Gamma}

\def\k{\mathfrak{k}}
\def\p{\mathfrak{p}}

\def\n{\mathfrak{n}}
\def\m{\mathfrak{m}}

\def\inf{\infty}
\def\x{\times}

\def\Ker{\ensuremath{\mathrm{Ker }}}

\def\Tr{\ensuremath{\mathrm{Tr}}}
\def\vol{\ensuremath{\mathrm{vol}}}
\def\diag{\ensuremath{\mathrm{diag}}}

\def\Ad{\ensuremath{\mathrm{Ad}}}
\def\Hom{\ensuremath{\mathrm{Hom}}}
\def\End{\ensuremath{\mathrm{End}}}

\def\GL{\ensuremath{\mathrm{GL}_3}}

\def\GLn{\ensuremath{\mathrm{GL}_n}}

\def\Gal{\ensuremath{\mathrm{Gal}}}
\def\Frob{\ensuremath{\mathrm{Frob}}}

\def\glr{\ensuremath{\mathfrak{gl}_3(\R)}}

\def\SO{\ensuremath{\mathrm{SO}(3)}}

\def\Ld{\ensuremath{\mathrm{L}}}

\def\As{\mathrm{As}}
\newcommand{\U}{\mathrm{U}}

\title{Stable base change from unitary groups in three variables and integral relation of automorphic periods}

\author{Tristan Ricoul}
\date{}

\begin{document}

\maketitle

\begin{abstract}
Let $E$ be a real quadratic field and let $U_E$ be the quasi-split unitary group in three variables associated with $E$. We prove a  $p$-adic divisibility between the automorphic periods of a cuspidal automorphic representation $\pi_U$ of $U_E$ and the periods of its Rogawski stable base change to $\mathrm{GL}_3(E)$. This generalizes earlier works of Tilouine--Urban and Hida in the case of $\mathrm{GL}_2$. The divisibility we prove involves a new kind of automorphic periods for self-conjugate cuspidal automorphic representations of $\mathrm{GL}_3(E)$, defined within the middle degree of the cuspidal cohomology rather than the top or bottom degrees. Moreover, we prove an \textit{à la Hida} adjoint $L$-value formula for $\mathrm{GL}_3(E)$, relating these newly defined middle-degree periods to the usual top and bottom automorphic periods. Finally, we also prove a similar formula for $U_E$, which is the first instance of such a result for quasi-split unitary groups.
\end{abstract}

\tableofcontents

\section{Introduction}

A natural problem in the arithmetic theory of the Langlands program is to understand how periods of automorphic forms behave under instances of Langlands functoriality. For a given Langlands transfer, an integral relation between the periods of an automorphic representation and the periods of its transfer can be deduced from the Bloch--Kato conjectures for various motives attached to these representations. Such a period relation was first established by Tilouine--Urban \cite{TU22} for the quadratic base change of a classical modular form. However, beyond rank two, no analogous results were previously known.

The goal of this paper\footnote{This paper, together with~\cite{TR_BC}, originates from a preprint~\cite{R24} posted on arXiv in November~2024 and later split into two papers. This preprint was withdrawn and subsequently expanded into the author’s PhD thesis~\cite{thesis}.} is to study such period relations in the setting of the stable base change from unitary groups in three variables.
In doing so, we introduce a new type of automorphic periods that  arise in this context and prove a divisibility in the
expected period relation, which constitutes the first result of this kind in higher rank.

\subsection{A conjectural integral period relation for the stable base change} We begin by formulating a conjectural integral period relation for the stable base change. Let $E$ be a real quadratic field and let $U_E$ be the quasi-split unitary group in $3$ variables associated to $E$, i.e. the $\Q$-algebraic group whose $A$-points are given, for any $\Q$-algebra $A$, by:
$$
U_E(A) = \{ g \in \mathrm{GL}_3(E\otimes_\Q A), \,\, {}^t\s(g) J g= J \}
$$
where $J = \mathrm{antidiag}(1,-1,1) \in \GL(E)$, and $\s$ denotes the Galois conjugation of $E/\Q$ acting on the first coordinate of $E\otimes_\Q A$. Let $\pi$ be a cohomological cuspidal automorphic representation of $U_E$, which is self-dual, and assume that $\pi$ is non-endoscopic and globally generic. Let $p$ be an odd prime number and let $\O$ be the integer ring of some sufficiently large $p$-adic field $\K$. Then $\pi_U$ is associated with two $p$-integral automorphic periods $\Om_2(\pi_U)$ and $\Om_3(\pi_U)$, respectively called the bottom and top degree periods of $\pi_U$. These two periods are non-zero complex numbers, which compare the $\O$-integral structure of the cohomology of $U_E$ given by algebraic topology to the $\O$-integral structure on the space of automorphic forms given by Whittaker models.

Let $\Pi = \mathrm{SBC}(\pi)$ be the strong stable base change of $\pi$ to $\GL(E)$, whose existence, predicted by the Langlands principle of functoriality, has been established by Rogawski \cite{ARU3}. Then $\Pi$ is a cohomological automorphic representation of $\GL(E)$, which is cuspidal, as $\pi$ is non-endoscopic. The cuspidal cohomology of $\GL(E)$ is concentrated in degrees $q=4,5,6$. Thus, $\Pi$ is also associated with two $p$-integral automorphic periods $\Om_4(\Pi)$ (the bottom degree period) and $\Om_6(\Pi)$ (the top degree period), defined within the bottom $b=4$ and top $t=6$ degrees of the cuspidal cohomology of $\GL(E)$.

In general, as for any cuspidal automorphic representation of $\GLn$, these are the only two periods one can attach to $\Pi$. The reason for that is that $\Pi$ appears with multiplicity $1$ in extremal (top $t=6$ and bottom $b=4$) degree cuspidal cohomology groups, but not in middle degree $m=5$. The first contribution of the present paper is the introduction of a new kind of automorphic periods, $\Om_5(\Pi,\s,+)$ and $\Om_5(\Pi,\s,-)$, associated with any cohomological cuspidal representation $\Pi$ of $\GL(E)$ which is self-conjugate, i.e. such that $\Pi^\s = \Pi$ (note that the stable base change $\Pi = \mathrm{SBC}(\pi_U)$ of a self-dual representation $\pi_U$ is self-conjugate). These two periods are called the $\s$-periods of $\Pi$  and are defined within the middle degree of the cuspidal cohomology module of $\GL(E)$. The reason we introduce these middle-degree periods is that we conjecture that they are related to the periods of $\pi_U$, whereas, for dimensional reasons which we will explain later, it is not clear that the extremal periods are. \\

According to various famous conjectures (see Deligne \cite{De79}, Beilinson \cite{beilinson}, and Bloch-Kato \cite{BK07}), motivic periods are expected to be closely related to special values of $L$-functions. Thus, by the automorphic/motivic analogy, automorphic periods should also be related to special $L$-values, and in turn decompositions of $L$-functions should correspond to relations of automorphic periods. This leads us to formulate the following conjectural integral relation between the periods of $\pi_U$ and the newly defined middle-degree periods of $\Pi$ (see \cite{thesis} for a precise motivation) :

\begin{mainconj}
\label{main_conj}
Let $\n \subset \O_E$ be the mirahoric level of $\Pi$. Assume that $p$ is prime to $6N_{E/\Q}(\mathfrak{n}) h_E(\mathfrak{n})D_E$, then:
$$
\Om_5(\Pi,\s,-) \sim \Om_2(\pi_U) \cdot \Om_3(\pi_U) \cdot \nu_{\pi_U}
$$
where $ \nu_{\pi_U} := \eta_{\pi_U} \cdot \eta_{\pi_U}(H^5_\TT)[+]^{-1} \in \O$ and we write $z_1 \sim z_2$, for $z_1,z_2 \in \C^\x$ if there exists $a \in \O^\x$ such that $z_1 = a \cdot z_2$.
\end{mainconj}

In the above formula, the term $\nu_{\pi_U} \in \O$ measures the defect between the Hecke congruence number  of ${\pi_U}$ and the cohomological congruence number of ${\pi_U}$ on the degree-5 cuspidal cohomology of $\GL(E)$ (see \S\ref{congruence_modules} and \S\ref{congruence_modules_with_an_involution} for their precise definition). The congruence number $\eta_{\pi_U}$ has a deep number-theoretic meaning, as it can sometimes be related to the adjoint Selmer group of (the Galois representation associated with) $\pi_U$, and thus be can though as an automorphic analog of the latter (see \cite{thesis}[\S1.1.5]). Moreover, in some precise sense, this conjecture is equivalent to (an automorphic analog of) the Bloch-Kato Tamagawa number conjecture for the adjoint motive of $\pi_U$ twisted by the quadratic character $\chi_E$ (see the discussion below \ref{thmB}). For $\mathrm{GL}_2$, Tilouine-Urban \cite{TU22} have proven a similar period relation for the base change of a classical modular form $f$ to a quadratic field, building on some previous work by Hida \cite{Hi99} (note that for $\mathrm{GL}_2$ there is no difference between the stable base change and the classical base change). Their method ultimately relies on some numerical coincidence between the top degree of the cuspidal cohomology of $\mathrm{GL}_2(E)$ (for $E$ real or imaginary quadratic) and the real dimension of some locally symmetric manifold (in their case, the modular curve). However this numerical coincidence no longer holds in general for the stable base change to $\GLn(E)$ when $n\geq 3$, and the situation is more delicate,  as the periods expected to appear in such a period relation are therefore no longer the extremal periods, but rather some hypothetical intermediate degree periods which are not even defined. This is the reason why, in order to formulate \ref{main_conj}  for $n=3$, we needed to introduce the new middle-degree automorphic periods $\Om_5(\Pi,\s,\pm)$ attached to self-conjugate automorphic representations of $\GL(E)$. As already said, these periods are defined within the degree-5 cuspidal cohomology group, and $5$ is precisely the real dimension of the $\GL(\Q)$-locally symmetric manifold inside the $10$-dimensional $\GL(E)$-locally symmetric manifold.  \\

In this paper, we prove one divisibility of \ref{main_conj}. This result relies on three main ingredients. First, we establish one divisibility of (an automorphic version of) the Bloch-Kato conjecture for the adjoint motive of $\pi_U$, twisted by the even quadratic character $\chi_E$. Second, we prove an \textit{à la Hida} adjoint $L$-value formulas for $\GL(E)$, which relates the adjoint $L$-value $L(\Pi,\Ad,1)$ to the newly defined middle-degree periods $\Om_5(\Pi,\s,\pm)$ of $\Pi$. Finally, we prove  a similar \textit{à la Hida} adjoint $L$-value formulas for $U_E$ ; this provides the first instance of such a formula for automorphic representations of unitary groups.

\subsection{Main result} Let us now present in more details the main result of this paper. 
Let $E$ is a real quadratic field. Let $\pi_U$ and $\Pi= \mathrm{SBC}(\pi_U)$ be as above.  Assume moreover that $\pi_U$ is only ramified at primes which are split in $E$. As mentioned in the previous pararagraph, as $\Pi \simeq \Pi^\s$ is self-conjugate, it is associated with a new kind of middle-degree automorphic periods $\Om_5(\Pi,\s,+)$ and $\Om_5(\Pi,\s,-)$, called the $\s$-periods of $\Pi$ (see \S\ref{s_periods} for a precise definition). Finally, we assume that the Galois representation $\rho_\Pi$ attached to $\Pi$ is residually irreducible. Our main result is the following theorem:

\begin{mainthm}[\ref{divisibilite_unitaire}]
\label{thmA}
Let $\pi_U$ and $\Pi = \mathrm{SBC}(\pi_U)$ be as above. Assume that the cohomological weight $\mu_U$ of $\pi_U$ is $p$-small and that $p \nmid 6N_{E/\Q}(\mathfrak{n}) h_E(\mathfrak{n})D_E$ and that $p\notin S_\partial$. Then, under the assumption $(H = \TT_U)$, we have the following divisibility:
$$
\nu_{\pi_U} \cdot \Om_2(\pi_U) \cdot \Om_3(\pi_U) \;  \mid \; \Om_5(\Pi,\s,-)
$$
where $\nu_{\pi_U} := \eta_{\pi_U} \cdot \eta_{\pi_U}(H^5_{\TT})[+]^{-1} \in \O$, and we write $z_1 \;  \mid \; z_2$  for $z_1,z_2 \in \C^\x$ if there exists $a \in \O$ such that $z_2 = a \cdot z_1$.
\end{mainthm}

As already explained, $\nu_{\pi_U} \in \O$ measures the defect between the Hecke congruence number $\eta_{\pi_U}$ of ${\pi_U}$ and the congruence number $\eta_{\pi_U}(H^5_{\TT})[+]$ on some submodule $H^5_{\TT}$ of the degree-5 cuspidal cohomology of $\GL(E)$ (see \S\ref{congruence_modules} and \S\ref{congruence_modules_with_an_involution} for the definition of these congruence numbers). Note that its presence in the above divisibility makes the divisibility stronger. $D_E$ is the discriminant of $E$, $\n \subset \O_E$ is the (mirahoric) level of $\Pi$ and $h_E(\mathfrak{n})$ is the cardinal of the ray class group of level $\mathfrak{n}$. The condition of $p$-smallness excludes a finite number of small primes with respect to $\mu$. The set $S_\partial$ is a finite set of primes introduced in \ref{thmD} below. Under some conditions on $\pi_U$, this set is expected to be empty. In any case, the conditions on $p$ exclude only finitely many prime numbers overall. Finally, assumption  $(H = \TT_U)$ is a condition concerning the freeness of the (localized) top degree cuspidal cohomology of $U_E$ over the localized Hecke algebra. By adapting Calegari-Geraghty theory to quasi-split unitary groups, one might be able to give some sufficient conditions on $\pi$  (such as minimality, Fontaine-Laffaille at $p$ and residual enormous image) so that $(H = \TT_U)$ holds.

We also prove a similar divisibility of periods  when the representation $\pi_U$ is no longer assumed to be self-dual. However, in this situation, rather restrictive conditions must be imposed on the ramification of the representation $\pi_U$ in order to obtain the same result. Moreover, the periods appearing in the formula are no longer the $\s$-periods of $\Pi$ (as $\Pi$ is no longer self-conjugate in this case), but rather another type of middle-degree periods, called the $\e$-periods, which are defined for conjugate self-dual representations (as is $\Pi$, being a stable base change from $U_E$). For the sake of clarity, we do not include this result in the  introduction, and we refer the reader to \S4.1 for its presentation. \\

As mentioned above, the integral period relation in \ref{main_conj} is  related to the Bloch-Kato Tamagawa number conjectures for the three motives $\Ad(\pi_U)$, $\Ad (\pi_U) \otimes \chi_E$ and $\Ad(\Pi)$. Let us now briefly explain why, in order to show how \ref{thmA} follows from the other main results in this article. The automorphic stable base change corresponds, at the level of $L$-functions, to the following decomposition of adjoint $L$-functions:

\begin{equation}
\label{L_func_dec}
L(\Pi,\Ad,1) = L(\pi_U,\Ad,1) \cdot L(\pi_U,\Ad \otimes \chi_E,1) 
\end{equation}

Generalizing a formula first established by Hida \cite{Hi81a,Hi81b} in the case of modular forms, Balasubramanyam-Raghuram \cite{BR17} and Chen \cite{Che22} have established an adjoint $L$-value formula for $\Pi$, which relates the adjoint $L$-value $L(\Pi,\Ad,1)$ to the cohomological congruence number of $\Pi$, and the product $\Om_4(\Pi) \cdot \Om_6(\Pi)$ of the periods of $\Pi$. In this paper (see \ref{thmD} below) we prove a similar formula for the representation $\pi_U$ of the quasi-split unitary group $U_E$, which relates the adjoint $L$-value $L(\pi_U,\Ad,1)$ to the cohomological congruence number of $\pi_U$, and the product $\Om_2(\pi_U) \cdot \Om_3(\pi_U)$ of the periods of $\pi_U$. Under some reasonable conditions on $\pi_U$ and $\Pi$, one may interpret these results (see \cite{thesis}[\S 1.1] for details, in particular Proposition 1.2) as an automorphic analog of the Bloch-Kato Tamagawa number conjecture for the adjoint motives of $\pi_U$ and $\Pi$, involving the product of the two extremal periods of $\pi_U$ and $\Pi$ instead the motivic periods appearing the original  formulation of the Bloch-Kato conjecture \cite{BK07}[Conjecture 5.15]. Consequently, given the decomposition (\ref{L_func_dec}), one needs to study the Bloch-Kato conjecture for the twisted adjoint motive of $\pi_U$ in order to relate the periods of $\pi_U$ to the periods of $\Pi$. \\

The main result of this paper is that we prove one divisibility of (an automorphic version of) the Bloch-Kato conjecture for the twisted motive $\Ad(\pi_U) \otimes \chi_E$ (see \ref{thmB} below). As our formula involves the middle-degree period $\Om_5(\Pi,\s,+)$ of $\Pi$, we need to establish an adjoint $L$-value formula involving the middle-degree periods of $\Pi$ as well, instead of the top and bottom periods of $\Pi$ appearing in Balasubramanyam-Raghuram and Chen formula. This is \ref{thmC} below.  Modulo some congruence module algebra, the periods divisibility in \ref{thmA} then easily follows from these results, which we now present in more details. \\

\subsection{The Bloch-Kato conjecture for the twisted adjoint motive of $\pi_U$} Let $\Pi = \mathrm{SBC}(\pi_U)$ be as before, except that in this paragraph we no longer need to assume that $\pi_U$ is globally generic. Let $\mu_U \in X^+(T_3)$ be the cohomological weight of $\pi_U$, so that $\mu_E = (\mu_U,\mu_U^\vee) = (\mu_U,\mu_U)$ is the cohomological weight of $\Pi$. Let $\mathfrak{n}$ be the mirahoric level of $\Pi$ and let $K_f := K_1(\n)$ be the mirahoric subgroup of level $\mathfrak{n}$. We denote by $\TT_E$ be the spherical Hecke algebra of level $K_f$, localized at the maximal ideal $\m_\Pi$ corresponding to $\Pi$. Let $M$ denote the degree $5$ cuspidal cohomology group of $\GL(E)$ localized at $\m_\Pi$:
$$
M = H_{cusp}^5(Y_E(K_f), \L_{\mu_E}(\O))_{\m_\Pi}
$$
where $Y_E(K_f)$ is the adelic variety of level $K_f$ for $\GL(E)$, and $\L_{\mu_E}(\O)$ is a locally constant sheaf associated to the irreducible algebraic representation $L_{\mu_E}$ of $\GL(E)$ of highest weight $\mu_E$. Then, $M$ is a $\TT_E$-module endowed with a semi-linear action of $\s$, the Galois involution of $E$. Let $\eta_{\Pi}^\#(M^*)$ be the relative (for the stable base change) congruence number of $\Pi$ on the dual module $M^* = \Hom_\O(M,\O)$. See \S\ref{sss_relative} for its definition. We prove the following result:

\begin{mainthm}[\ref{SBC_divisibility}]
\label{thmB}
Let $\pi_U$ and $\Pi = SBC(\pi_U)$ as above. Assume that the Galois representation associated to $\Pi$ is residually absolutely irreducible. Suppose that the cohomological weight $\mu_E$ of $\Pi$ is $p$-small and that $p~\nmid~6N_{E/\Q}(\mathfrak{n}) h_E(\mathfrak{n})D_E$. Then:
$$
\eta_\Pi^\#(M^*)[+] \quad | \quad \frac{\Lambda^{imp}(\pi_U,\Ad \otimes \chi_{E},1)}{\Om_5(\Pi,\s,+)}
$$
where $\eta_\Pi^\#(M^*)[+]$ is the $+$-part of $\eta_\Pi^\#(M^*)$ for the action of $\s$.
\end{mainthm}

In the above theorem, $\Lambda^{imp}(\pi_U,\Ad \otimes \chi_{E},s)$ is the imprimitive completed twisted adjoint $L$-function of $\pi_U$, defined as an Euler product including the $\G$-factor at the archimedean place, and imprimitive local factors at places of ramification for $\pi_U$. Note that this theorem implies in particular that the normalized twisted adjoint $L$-value appearing on the right-hand side belongs to $\O$. The strategy of the proof of \ref{thmB} is explained in greater detail in paragraph \S1.5 of this introduction. We also prove a similar divisibility when the representation $\pi_U$ is no longer assumed to be self-dual. However, in this case, rather restrictive conditions must be imposed on the ramification of the representation $\pi_U$ in order to obtain the same result, and the periods appearing at the denominator of the $L$-value is different. For the sake of clarity, we do not include this result in the  introduction, and we refer the reader to \S4.1 for its presentation. \\

\ref{thmB} can be viewed as one half of an automorphic analogue of the Bloch-Kato Tamagawa number conjecture for the adjoint motive of $\pi_U$, twisted by the even character $\chi_E$. More precisely, under certain Calegari-Geraghty  assumptions \cite{CG18} on $\Pi$, the divisibility in \ref{thmB} can be shown to be equivalent to a divisibility between the order of the ($p$-divisible) Selmer group of this motive and the ratio of the $L$-value $L(\pi_U,\Ad  \otimes \chi_E,1)$ by the automorphic period $\Om_5(\Pi,\e,+)$ (instead of the motivic period appearing in the original formulation of the Bloch-Kato conjecture \cite{BK07}[Conjecture 5.15]. We refer to \cite{TR_BC} (see the discussion following Theorem B in the introduction) for a more detailed presentation of this interpretation in the case of the real quatratic base change for $\GL$.\\

As already explained, in order to deduce \ref{thmA} from \ref{thmB}, we must  relate the two adjoint $L$-values $L(\Pi,\Ad,1)$ and $L(\pi_U,\Ad,1)$ appearing in the decomposition (\ref{L_func_dec}) to the automorphic periods of $\Pi$ and $\pi_U$. This is the content of our next two results, which we present in the following two paragraphs. While these results are essential for the proof of \ref{thmA}, they are also of independent mathematical interest.

\subsection{A middle-degree adjoint $L$-value formula for $\Pi$} For $\Pi$, we have mentionned that there already exists such a formula, first established by Balasubramanyam and Raghuram \cite{BR17} up to some uncomputed archimedan factor, which has latter been computed by Chen \cite{Che22}. However, this formula relates the adjoint $L$-value $L(\Pi,\Ad,1)$ to the top and bottom periods of $\Pi$. As \ref{thmB} involves the new middle-degree periods $\Om(\Pi,\s,\pm)$, we thus need to prove a similar formula involving these periods. We keep the notation from \ref{thmA}. Let $\eta_{\Pi}(M)$ be the congruence number of $\Pi$ on the degree-5 inner cohomology group $M$ (see \S\ref{congruence_modules} for its definition). We then prove the following theorem:

\begin{mainthm}[\ref{adjoint_L_value}]
\label{thmC}
Let $\Pi$ be a cohomological automorphic cuspidal representation of $\GL(\A_E)$ which is self-conjugate. Under the same hypothesis on $\Pi$ and $p$ as in \ref{thmA}, one has:
$$
\eta_{\Pi}(M)[\pm] \quad \sim \quad \frac{\Lambda^{imp}(\Pi,\Ad,1)}{\Om_5(\Pi,\s,\pm) \cdot \Om_5(\Pi^\vee,\s,\mp)}
$$
where $\eta_{\Pi}(M)[\pm]$ is the $\pm$-part for the action of $\s$ on $\eta_{\Pi}(M)$ (in the sense of \S\ref{congruence_modules_with_an_involution}).
\end{mainthm}

In the above theorem, $\Lambda^{imp}(\Pi,\Ad,s)$ is the imprimitive completed twisted adjoint $L$-function of $\Pi$, defined as an Euler product including the $\G$-factor at the archimedean places, and imprimitive local factors at places of ramification for $\Pi$. In particular, \ref{adjoint_L_value} implies that the normalized adjoint $L$-values appearing on the right-hand sides is in $\O$. Since $\eta_{\l_\Pi}(H^5)[\pm]$ divides the Hecke congruence number $\eta_{\l_\Pi}$ of $\Pi$, the above equality implies that if the $\wp$-adic valuation of the right hand side is non-zero (where $\wp$ is the prime ideal of $\O$), then there exists a cuspidal automorphic representation $\Pi' \in \mathrm{Coh}(G_E,\mu_E,K_f)$ distinct from $\Pi$ which is congruent to $\Pi$ modulo $\wp$. However, we stress that $\eta_{\l_\Pi}(H^5)[\pm]$ may not be equal to $\eta_{\l_\Pi}$, since $H^5$ is not free as a $\TT$-module. \\

In some separate paper, we prove an analogous theorem (see \cite[Theorem 3.2]{TR_BC}) for conjugate self-dual automorphic representations of $\GL(E)$, involving the so-called $\e$-periods of $\Pi$ instead its $\s$-periods (which are not defined for conjugate self-dual representation). Since the proofs of these two theorems are essentially identical, we do not detail the proof of \ref{thmC} in this paper, and we refer to \cite[Section 3]{TR_BC}. Otherwise, we refer to Section 4 of the author's thesis \cite{thesis} where both results are proven in details. \\

Finally, we mention that \ref{thmC}, combined with Balasubramanyam-Raghuram's formula \cite[Theorem A]{BR17}, implies an integral relation between the middle-degree periods and the top-bottom periods of $\Pi$, under Calegari-Geraghty setting (see \ref{relation_middle_top_bottom}). \\

\subsection{An adjoint $L$-value formula for $\pi_U$} The Balasubramanyam-Raghuram formula is only valid for the split group $\GLn$. Our last result is an analog of there formula for the quasi-split group $U_E$. More precisely, we prove an adjoint $L$-value formula relating the (cohomological) congruence number of the representation $\pi_U$ to the quotient of the special value $L(\pi_U,\Ad,1)$ by the product of the automorphic periods of $\pi_U$. In fact, our result is more general. Let $E/F$ be a quadratic extension of totally real number fields and assume that the different $\mathfrak{d}_F$ of $F$ splits in $E$. Let $U_{E/F}$ be the quasi-split group in $3$-variables associated with $E/F$. 

Let $\pi$ be a cohomological cuspidal automorphic representation of $U_{E/F}(\A_F)$ and assume that $\pi$ is non-endoscopic and globally generic. Assume that $\pi$ is spherical at places of $F$ which are non-split in $E$. Then, $\pi$ can be associated with two automorphic periods $\Om_b(\pi)$ and $\Om_t(\pi)$  in $\C^\x /\O^\x$, with $b= 2 \dim(F)$ and $t = 3 \dim(F)$ being the bottom and top degrees of the cuspidal cohomology of $U_{E/F}$ (see \S\ref{unitary_3}). We then prove the following result:

\begin{mainthm}[\ref{unitary_adjoint_L_value}]
\label{thmD}
Let $\pi$ be as above. Assume that the cohomological weight $\mu$ of $\pi$ is $p$-small, that $p \notin S_\partial$ and that $p \nmid  2N_{F/\Q}(\mathfrak{c})h_U(\mathfrak{c})D_F$. Then:
$$
\eta_\pi(H^t) \sim u_{E/F} \cdot \frac{\Lambda^{imp}(\pi, \Ad,1)}{\Om_b(\pi) \cdot \Om_t(\pi^{\vee})}
$$
where $u_{E/F} \in \Q(\pi_f)^\x$ is a constant conjectured to be $1$. For $F=\Q$, one has that $u_{E/\Q} =1$.
\end{mainthm}

The set $S_\partial$ is some finite set of primes which consists of primes supporting the torsion of the boundary cohomology of the adelic variety of $U_{E/F}$. This set is expected to be empty under some conditions on $\pi$. The ideal $\mathfrak{c} \subset \O_F$ is the mirahoric level of $\pi$ and $h_U(\mathfrak{c})$ is some class number. The ideal $\eta_{\pi}(H^t) \subset \O$ is the congruence number of $\pi_U$ on the inner cohomology group $H^t$ of (top) degree $t$. Finally, the constant $u_{E/F} \in \C^\x$ is some uncomputed product of local factors at places of $F$ which ramifies in $E$. When $F = \Q$, this factor is equal to $1$. In general, we also expect $u_{E/F}$ to be trivial but we can only prove that it belongs to the rationality field $\Q(\pi_{f})$ of the finite part $\pi_{f}$ of $\pi$. We refer to Section \ref{part_adjoint} for precise definitions, proof and comments.\\

The proof of \ref{thmD} proceeds along the same lines as that of the adjoint $L$-value formulas for $\mathrm{GL}_n$ (\cite{Hi81a,BR17}) One first needs to relate the adjoint $L$-value of $\pi$ to the Petersson product $\langle \phi,\phi' \rangle$ of some forms $\phi \in \pi$ and $\phi \in \pi^\vee$. For $\GLn$ such a relation is given by the Jacquet-Shalika formula. In the case of quasi-split groups, such a relation can be obtained by using the conjectural Lapid-Mao formula, which has recently been proved by Morimoto for quasi-split unitary groups. 
One then has to compute some local factors at bad places. Since $E$ and $F$ are totally real fields and we assume that $\pi$ is only ramified at split places of $F$, the computation of these local factors at archimedean and ramification places reduces to the case of $\GLn$. It then remains the local factors at places of $F$ which are ramified in $E$. This yields the uncomputed factor $u_{E/F}$. From the obtained formula, \ref{thmD} is proved by cohomologically interpreting the Petersson product  as a Poincaré pairing between the top and bottom degree inner cohomological groups of $U_{E/F}$. These groups are Hecke modules of rank $1$, from which automorphic periods can be defined, as in the case of $\mathrm{GL}_n$. One needs to exclude the primes $p \in S_\partial$ so that the Poincaré pairing is perfect on $\O$. The formula then follows from the theory of congruence modules developed by Hida in \cite{Hi81a}. \\

\subsection{Proof of \ref{thmB} and definition of the $\s$-periods} Since this is a general method for establishing period relation, we now outline the proof of \ref{thmB}. In doing so, we explain how the $\s$-periods are defined. The key point is to construct an $\O$-linear form on the cuspidal cohomology of $Y_E(K_f)$ with coefficients in $\O$, which has the following two properties :
\begin{itemize}
\item it vanishes on classes which correspond to automorphic representations which are not stable base-changes from $U_E$
\item it takes the expected twisted adjoint $L$-value on (the cohomological class associated to) some newform $\phi_f^\circ$ in $\Pi_f$. 
\end{itemize}

The idea is to consider the following linear form, called the Flicker-Rallis integral period:
$$
\P_{\GL}: \phi \in \Pi \mapsto \int_{Z_{3}(\mathbb{A}) \GL(\Q) \backslash \GL(\mathbb{A})} \phi(g)  dg
$$
and to factorizes it as a linear form on the cohomology of $\GL(E)$, through some Eichler-Shimura maps. The cuspidal cohomology of $\GL(E)$ is concentrated in degree $q=4,5,6$. As the modular submanifold associated to $\GL(\Q)$ is $5$-dimensional, the integral period $\P_{\GL}(\phi)$ should be cohomologically interpreted as an integral involving the differential $5$-form $\d(\phi)$ associated to $\phi$ through some Eichler-Shimura map $\d$. However, here we face two problems. The first problem concerns the normalisation of the Eichler-Shimura map $\d$. In fact, this map is only defined up to a complex scalar since its definition depends on the choice of a generator of certain $(\g,K_\inf)$-cohomology groups. Since we want to compare the $p$-adic integral stucture given by the Whittaker model on $\Pi_f$ to the $p$-adic integral structure on the cuspidal cohomology, we need to make a canonical choice for this generator. For $\mathrm{GL}_2$, this has been done by Hida (see \cite[Section 4]{Hi94}), and the maps are normalized so that they preserve the completed $L$-functions on both sides. To the best of our knowledge, this has not been done for $\GLn$ yet. However, in the special case of $\GL$ over a totally real field, Chen \cite{Che22} has explicited some canonical choices for these generators, and shown that they give the correct adjoint $\G$-factors. The second problem is that the Eichler-Shimura map $\d$ lands in:
$$
H^5_{cusp}(Y_E(K_f), \L_{\mu_E}(\C))[\Pi_f] 
$$
which is a $2$-dimensional $\C$-vector space. Hence it may not be possible to normalize $\d$ by some complex scalar (a period) so that it lands in the canonical $\O$-structure of the above $\C$-vector space. This is why the Betti-Whittaker periods of $\Pi$ are only defined using the top degree $t=6$ and the bottom degree $b=4$ cuspidal cohomology groups, whose $\Pi_f$-part is $1$-dimensional. However, following an idea of Hida, we remark that when $\Pi$ is self-conjugate and if $K_f$ is Galois invariant, the $\Pi$-part of the cuspidal cohomology is endowed with a non-trivial action of the Galois involution $\s \in \Gal(E/\Q)$. Moreover, this action preserves the integral $\O$-structure. Thus, we can define two canonically normalized Eichler-Shimura maps:
$$
\d_\s^\pm: \Pi_f^{K_f} \to H^5_{cusp}(Y_E(K_f), \L_{\mu_E}(\C))[\Pi,\s = \pm]
$$
landing into the $\pm$-eigenspace for $\s$. The right-hand side is a complex $1$-dimensional vector space admitting a canonical integral $\O$-structure, so we can attach two periods $\Om(\Pi,\s,\pm)$ to $\Pi$ provided it is self-conjugate (see paragraph~\S\ref{s_periods} for details).

We can then define a linear form $\Ld$ on the middle-degree cuspidal cohomology group $H^5_{cusp}(Y_E(K_f), \L_{\mu_E}(\C))$, so that the Flicker-Rallis period factorizes into $\Ld$ through the Eichler-Shimura maps. The first vanishing property is thus ensured by the Flicker-Rallis conjecture, which has been proven by Mok \cite{Mok}. It states that the Flicker-Rallis period is a non-zero linear form if and only if $\Pi$ is the stable base change lift of some cuspidal automorphic representation of $U_E$. To check the second property we need to compute the explicit value of this linear form on the $5$-differential form $\d^{+}_\s(\phi_\Pi^\circ)$ associated the essential vector $\phi_\Pi^\circ$ of $\Pi$, which is some normalized $K_f$-newform in $K_f$ used to define the periods $\Om_5(\Pi,\s,\pm)$ of $\Pi$. This is done using a formula of Flicker \cite{Flicker88}, which relates the Flicker-Rallis period to the residue at $s=1$ of the Asai $L$-function $L(\Pi,\As,s)$ of $\Pi$. Flicker's formula involves uncomputed ramified and archimedean local factors, which we thus have to compute. The archimedean factor is computed by adapting some computations of Chen \cite{Che22}, and give the (twisted) adjoint $\G$-factors. The ramified local factors are computed by using some explicit expressions for the essential vector given separately by Miyauchi and Matringe (see paragraph~\ref{mirahoric_theory}), and are equal to the imprimitive local $L$-factors. Finally, the divisibility of \ref{thmB} is proven by using the key \ref{lf_lemma}, which is an adaptation of \cite[Lemma 2.9]{TU22} to the context of Hecke-modules with a semi-linear involution. \\

\subsection{Comments} We conclude this introduction with a few comments. A natural question is how to prove the reciprocal divisibility in \ref{thmA} (or equivalently in \ref{thmB}). In the few case where such a divisibility is established, the proof generally involves some non-vanishing modulo $p$ result. For instance, in the case of the Yoshida lift of two cuspidal modular forms, Liu and Hsieh (in some work in preparation) have been able to establish such a divisibility by using the non-vanishing modulo $p$ of the Yoshida lift's Fourier coefficients. In the case of the quadratic base change for $\mathrm{GL}_2$, the reciprocal divisibility is proven in \cite{TU22} by using a theorem of Cornut-Vatsal \cite{Cornut02,Vat02} on the non-vanishing modulo $p$ of twisted standard $L$-values for weight-$2$ modular forms, generalized by Chida-Hsieh \cite{CH18} to $p$-small weights. However, to the knowledge of the author, no such result is known for quasi-split unitary groups.  \\

We now discuss to what extent the results of the present paper could be generalized to larger $n$ and general extensions $E/F$ of number fields. The problem as $n$ grows and $E$ gets bigger is that the size $\ell_E$ of the cuspidal range of $\GLn(E)$ becomes greater than $2$. In this situation, this is not clear how to define intermediate degree periods as the dimension of the intermediate cohomology groups can be strictly bigger than $2$, whereas the number of involutions available for cutting out $1$-dimensional subspaces remains the same. Hence in general, some numerical coincidence is needed between the dimension of the adelic varieties associated with $\mathrm{GL}_{n}(F)$ and the extremal degrees of the cuspidal cohomology of $\mathrm{GL}_{n}(E)$. Such a numerical coincidence happens for any $n$ when $E/F$ is CM extension, in the case of the stable base change. This situation has been studied rationally by \cite{GHL16}, and integrally by Balasubramanyam and Tilouine \cite{BT}. Nevertheless, the methods presented in this paper may apply to other cases where $\ell_E=2$, which include:
\begin{itemize}
\item  Case 1: $n=2$ and $E$ is an imaginary bi-quadratic number field;
\item  Case 2: $n=4$ and $E$ is real quadratic;
\item  Case 3: $n=3$ and $E$ is imaginary quadratic.
\end{itemize}
The first two cases are similar to the situation of this paper in the sense that the archimedean places of $F$ are split in $E$. Case 1 has been studied by Hida \cite[Section 7]{Hi99}. Our results for $\GL$ should easily be extended to Case 2, provided one is able to extend Chen’s archimedean computations to $\mathrm{GL}_4$. Case 3 is investigated by Balasubramanyam and Tilouine \cite{BT}. \\

Of course, one can also consider the classical base change from $\GLn(F)$ to $\GLn(E)$. In a separate paper \cite{TR_BC}, we study this situation for $n=3$ and when $E$ is real quadratic. We prove a periods divisibility result for this automorphic transfer using methods which are similar to the methods developed in the present paper. In some sense, the situation for the base change is symmetric to that of the stable base change considered here, and we refer to the introduction of the author’s thesis \cite{thesis} for a unified presentation of these two settings.\\

Finally, the results of this paper could be used to construct $p$-adic $L$-functions interpolating the values of the twisted adjoint $L$-function $L(\pi_U,\Ad \otimes \chi_E,1)$ as $\pi_U$ varies in a Hida family of self-dual cuspidal automorphic representations of $U_E$. We refer the interested reader to the last paragraph of the introduction of \cite{TR_BC}, where the construction of such a $p$-adic $L$-function is sketched for $\GL(\Q)$. \\

\subsection{Outline} This paper is organized as follows. Section~\ref{section_periods} introduces the main objects and defines the middle-degree $\s$-periods associated with self-conjugate automorphic representations of $\GL(E)$.  The section concludes with a precise statement of \ref{thmC}. In Section~\ref{part_adjoint} we recall some basic definitions and properties of quasi-split unitary groups and their automorphic representations, and we prove \ref{thmD}. Section~\ref{part_SBC} recalls basic definitions and results related to the stable base change, proves \ref{thmB}, and establishes the main period divisibility of \ref{thmA}. The final paragraph of this section is devoted to the case of non-self-dual representations.

\subsection{Acknowledgements} The author is deeply indebted to his advisor, J. Tilouine, for his constant help and support during the preparation of this paper. The author also would like to thank B. Balasubramanyam, S.-Y. Chen, E. Ghate, G.Grossi, H. Hida, M.-L.Hsieh, E. Lapid, Z. Liu, N. Matringe, M. Moakher, D. Prasad, K. Prasanna and E. Urban for stimulating conversations, comments and suggestions regarding this work. This paper was partly completed during two visits at IISER Pune on an invitation by B. Balasubramanyam and a visit at IIT Bombay on an invitation by D. Prasad, as well as during a visit at the NCTS on an invitation by M.-L. Hsieh. The author would like to thank them heartily for their hospitality. I would also like to thank Eknath Ghate and Sandeep Varma, organizers of the \textit{p-adic Methods in Number Theory conference} at TIFR in Mumbai, for giving me the opportunity to present my work there.

\subsection{Notation}

Throughout this paper, we consider an odd prime number $p$. We fix an embedding $j: \overline{\Q} \inj \C$ and an isomorphism $ j_p: \overline{\Q}_p \simeq \C$. Let $\K \subset \overline{\Q}_p$ be some sufficiently large finite extension of $\Q_p$. We denote by $\O$ its integers ring, and $\wp$ its prime ideal.

Using $j_p$, we can see complex numbers as elements of $\overline{\Q}_p$. Thus, if $z_1$ and $z_2$ are two non-zero complex numbers, we write $z_1 \sim z_2$ (resp. $z_1 \mid z_2$) if the quotient $z_2/z_1$ lies in $\O^\x$ (resp. in $\O$).

Unless otherwise specified, $E$ will always denote a real quadratic field, and $\s \in \mathrm{Gal}(E/\Q)$ is the non-trivial element in $\mathrm{Gal}(E/\Q)$. We fix an embedding $\tau: E \inj \overline{\Q}$, so that via $j$ we can identify the set $\{\tau,\s\tau\}$ with the set of archimedean places of $E$. Moreover, we denote by $\chi_E$ the quadratic Dirichlet character associated with $E$.

\subsubsection{Number fields, groups and additive characters}
\label{notations}
Let $F$ be some number field. We denote by $\O_F$ its ring of integer, $\mathfrak{d}_F$ its different and $D_F$ its discriminant. If $v$ is any place of $F$, we denote by $F_v$ the completion at $v$. When $v$ is finite, we denote by $\O_{F_v}$ (or simply $\O_v$ when the context is clear) its ring of integers, $\wp_v$ its prime ideal and $\varpi_v$ an uniformizer of $F_v$. Let $ S_\inf(F)$ be the set of archimedean places of $F$, $\A_F$ be the adeles of $F$ and $\A_{F,f}$ its finite part. When $F=\Q$, we simply write $\A$ and $\A_f$. \\

Let $n\geq 1$ be an integer. Let $\GLn$ denote the general linear group of dimension $n$, $B_n$ its standard Borel, $N_n$ and $T_n$ respectively the maximal unipotent and the maximal torus of $B_n$. Let $Z_n$ denote the center of $\GLn$. We also denote by $G_F$ the $\Q$-algebraic group $\mathrm{Res}_{F/\Q}(\mathrm{GL}_{n/F})$, whose $A$-points are given, for any $\Q$-algebra $A$, by:
$$
G_{F}(A) = \mathrm{GL}_n(F\otimes_\Q A).
$$
Let $K_n$ be the identity component of the maximal compact subgroup modulo center of $\GLn(\R)$:
$$
K_n= \mathrm{SO}(n) \R_+^\x
$$
where $\R^\x$ is seen as the center of $\GLn(\R)$. Let $\g_{n}$ and $\k_{n}$ be the Lie algebras of $\GLn(\R)$ and $K_n$, and let $\p_{n} := \g_{n}/\k_{n}$. If $F$ is a totally real field, let $K_{\inf} = \prod_{\tau \in  S_\inf(F)} K_n
$ be the identity component of the maximal compact subgroup modulo center of $G_F(\R)$ and let $\g = \oplus_\tau \g_{n}$ and $\k =\oplus_\tau \k_{n}$ be the Lie algebras of $G_F(\R)$ and $K_\inf$, and let $\p := \g /\k$. Let $\g_\C$, $\k_\C$ and $\p_\C = \g_\C /\k_\C$ denote their complexification, and $\p_\C^*$ the dual space of $\p_\C$. \\

We denote by $\psi_\Q = \bigotimes_v \psi_{\Q_v} : \Q\bs\A \to \C^\x$ the additive unramified character defined by:
$$
\begin{aligned}
\psi_{\Q_p}(x) &= \exp(-2\sqrt{-1}\pi[x]_p), \quad x \in \Q_p \\
\psi_\R(x) &= \exp(2\sqrt{-1}\pi x), \quad x \in \R \\
\end{aligned}
$$
Here $[x]_p := \sum_{k=-m}^{-1} c_k p^k$ is the $p$-fraction part of $x = \sum_{k=-m}^\inf c_kp^k$. Let $\psi_F : F \bs \A_F\to \C^\x$ be the standard non-trivial character  on $\A_F$, defined by $\psi_F = \psi_\Q \circ \Tr_{F/\Q}$. If $\psi : F \bs \A_F\to \C^\x$ is any additive character, it defines a character of $N_n(F)\bs N_n(\A_F)$, also denoted by the same symbol $\psi$, and defined by:
$$
\psi(n) = \psi(n_{1,2} + \dots + n_{n-1,n}), \quad n \in N_n(\A_F)
$$ 

\subsubsection{Haar measures}
\label{general_haar_measures}
Write $\psi_F = \bigotimes_v \psi_v$ for the standard additive character on $\A_F$. The additive Haar measures on $\A_F$ are fixed as follows. Let $v$ be a place of $F$. If $v$ is non-archimedean, the Haar measure $dx_v$ is the self-dual measure on $F_v$ with respect to $\psi_v$. The volume of $\O_v$ with respect to $dx_v$ is:
$$
\vol(\O_v,dx_v) = N(\mathfrak{d}_v)^{-1/2}
$$
where $\mathfrak{d}_v$ is the different of $F_v$. If $v$ is a real archimedean place, the Haar measure $dx_v$ on $F_v$ is the Lesbesgue measure. The global Haar measure on $\A_F$ is then $dx = \prod_v dx_v$. Note that:
$$
\vol(\A_F/F,dx) = 1.
$$

The local Haar measures on the idèles $\A_F^\x$ are fixed as follows. Let $v$ be a place of $F$. If $v$ is non-archimedean, corresponding to some prime ideal $\wp$, we set the Haar measure $d^\x x_v$ on $F_v^\x$ to be:
$$
d^\x x_v = (1-N\wp^{-1})^{-1} \frac{dx_v}{|x_v|}
$$
The volume of $\O_v^\x$ with respect to $d^\x x_v$ is  $\vol(\O_v^\x,d^\x x_v) = N(\mathfrak{d}_v)^{-1/2}$. If $v$ is a real archimedean place, we normalize the Haar measure $d^\x x_v$ on $F_v^\x = \R^\x$ by $d^\x x_v = |x_v|^{-1} dx_v$. The global Haar measure on $\A_F^\x$ is then $d^\x x = \prod_v d^\x x_v$. Note that this is not the Tamagawa measure on $\A_F^\x$. In particular, one has:  \\
$$
\vol(\A_F^1/F^\x,d^\x x) = \mathrm{Res}_{s=1} \zeta_E(s).
$$
where $\A_F^1$ denotes the subgroup of idèles of norm $1$ and $\zeta_E$ is the Dedekind zeta function of $E$. \\

The Haar mesure $dg_\inf$ on $\GLn(\R)$ is normalized using the Iwasawa decomposition $g = nak$, with $n \in N_n(\R)$, $a \in (\R^\x)^n$ and $k \in \mathrm{O}(n)$ by:
$$
dg = \d^{-1}_{B_n(\R)}(a)du d^\x a dk
$$
where $du$ is the Haar measure on $N_n(\R)$, $d^\x a =  \prod_{i=1}^n d^\x a_i$ and $dk$ is the probability measure on $\mathrm{O}(n)$. The modulus character $\d_{B_n(\R)}$ is defined by $\d_{B_n(\R)}(a) = \prod_{i=1}^n |a_i |^{n-2i+1}$. \\

\section{Eichler-Shimura maps and middle-degree automorphic periods}
\label{section_periods}

In this section, we first recall the new vector theory for the mirahoric subgroups of $\GLn$. We then introduce the cohomological framework and the automorphic objects needed throughout the paper. Finally, we construct the middle-degree Eichler–Shimura maps for $\GL(E)$, and define the middle-degree $\s$-periods attached to a self-conjugate cohomological cupsidal representation of $\GL(E)$.

\subsection{Mirahoric new vector theory for $\GLn$}
\label{mirahoric_theory}

\subsubsection{Local theory}

Let $L$ be a non-archimedean local field of characteristic zero, with valuation ring $\O$ and prime ideal $\wp$. Let $q = \#(\O/\wp)$ and let $\varpi$ be an uniformizer of $L$. For an integer $c \geq 0$, let $K_1(\wp^c)$ be the mirahoric subgroup of level $\wp^c$, which is the open compact subgroup of $\GLn(L)$ formed by matrices in $\GLn(\O)$ whose last row is congruent to:
$$
e_n:= (0, \dots,0,1)
$$
modulo $\wp^c$. Let $\pi$ be an irreducible admissible representation of $\GLn(L)$ which is generic. Let $\psi$ be a non-trivial unramified additive character of $L$, i.e. such that $\psi(\O_L) = 1$ and $\psi(\varpi^{-1}) \neq 1$. We denote by $\W(\pi,\psi)$ the Whittaker model of $\pi$ with respect to $\psi$, and by $\phi \mapsto W_\phi$ the isomorphism $\pi \toeq \W(\pi,\psi)$. For any non-negative integer $c$, we denote by $V(c)$ the space of $K_1(\wp^c)$-fixed vector in $\pi$. The following theorem is due to Jacquet, Piatetski-Shapiro, and Shalika (see~\cite[Section 5]{J-PS-S81}):

\begin{theorem}
\label{essential_vector}
Let $\pi$ be an irreducible admissible representation of $\GLn(L)$ which is generic. Then, there exists a non-negative integer $c$ such that $V(c) \neq 0$. Moreover, if $c(\pi) \geq 0$ is the minimal integer with this property, then: 
$$
\dim V(c(\pi)) = 1
$$
and $c(\pi)$ coincides with the analytic conductor of $\pi$, i.e the power of $q^{-s}$ in the $\e$-factor of $\pi$ with respect to an unramified additive character $\psi$ of $L$.
\end{theorem}

The integer $c(\pi)$ of the above theorem is called the \textit{mirahoric conductor} of $\pi$. We will say that $\wp^c$ is the \textit{mirahoric level} of $\pi$. A non-zero form $\phi$ in $V(c(\pi))$ will sometimes be called a newvector or a newform. The values of the Whittaker function $W_\phi$ associated with a newform $\phi$ on (a part of) the diagonal torus have been explicitly computed independently by Miyauchi~\cite[Theorem 4.1]{Miyauchi2012} and Matringe~\cite[Formula (1)]{Matringe13}. The proof of Miyauchi is a generalisation of Shintani's method for unramified representation and assume \ref{essential_vector}, while the proof of Matringe deduce these formulas from a new constructive proof of the results of Jacquet, Piatetski-Shapiro, and Shalika.

In order to state their formula, we need to introduce a few notation. First, for any $f = (f_1,\dots,f_{n-1}) \in \Z^{n-1}$, we note $\varpi^f = \mathrm{diag}(\varpi^{f_1},\dots, \varpi^{f_{n-1}}) \in \mathrm{GL}_{n-1}(L)$. A tuple $f \in \Z^{n-1}$ can also be seen as the element $(f_1,\dots,f_{n-1},0)$ in $\Z^n$. When $f$ is dominant as a weight for $\GLn$, i.e if $f_1 \geq \dots \geq f_{n-1} \geq 0$, we denote by $s_{f}(X_1,\dots,X_n)$ the Schur polynomial associated to $f$. Next, we recall that the standard $L$-function of $\pi$ is of degree $r \leq n$ with $r<n$ if and only if $\pi$ is ramified, i.e. if $c(\pi) >0$ (see \cite[Section 3]{Jacquet79}). Thus, we can write it as:
$$
L(\pi,s) = \prod_{i=1}^r (1-\a_i q^{-s})^{-1}
$$
with $\a_i \in \C^\x$. These notations having been presented, we have the following theorem:
\begin{theorem}
\label{explicit_essential_values}
Let $\pi$ be an irreducible admissible generic representation of $\GLn(L)$, and $\phi$ be a newvector in $\Pi$, and $W_\phi \in \W(\pi,\psi)$ its Whittaker function. Then, for all $f \in \Z^{n-1}$, we have:
$$
W_\phi \left(\begin{array}{ll}
\varpi^f & \\
& 1
\end{array}\right) = \left\{
    \begin{array}{ll}
        \d^{1/2}_{B_n}(\varpi^f)s_f(\a)W_\phi(1) & \mathrm{if} \,\,  f_1 \geq \dots \geq f_{n-1} \geq 0 \\
        0 & \mathrm{otherwise.}
    \end{array}
\right.
$$
where $s_f(\a) = s_f(\a_1, \dots,\a_r,0,\dots,0)$ and $\d_{B_n}$ is the modulus character of ${B_n}(L)$ whose value on $\varpi^f$ is given by $\d_B(\varpi^f) = q^{-\sum_{i=1}^{n-1}(n+1-2i)f_i}$.
\end{theorem}

The above result implies that $W_\phi(I_n) \neq 0$ for any non-zero newform $\phi$ in $\Pi$ (see \cite[Corollary 4.4]{Miyauchi2012}). The only newform $\phi$ such that  $W_\phi(I_n) = 1$ will be called the \textit{essential vector} of $\pi$ with respect to $\psi$, and is denoted $\phi_\pi^\circ$. Its Whittaker function is denoted $W_\pi^\circ$. When $c(\pi) = 0$, i.e. when $\pi$ is unramified, $\phi_\pi^\circ$ is rather called the \textit{spherical vector} (with respect to $\psi$). \\

\subsubsection{Global theory}
\label{global_mirahoric_theory}

Let $E$ be a number fields and let $\psi_f := \psi_{E,f}$ be the finite part of the additive character of $\A_E$ defined in \S \ref{notations}. Let $\mathfrak{n} = \prod_{\wp \mid \mathfrak{n}} \wp^{c_\wp}$ be an ideal of $\O_E$. We define the mirahoric subgroup $K_1(\mathfrak{n})$ of level $\mathfrak{n}$ as the following open compact subgroup of $\GLn(\A_{E,f})$:
$$
K_1(\mathfrak{n}) := \prod_{\wp \nmid \mathfrak{n}} \mathrm{GL}_{n}(\O_\wp) \x \prod_{\wp \mid \mathfrak{n}} K_1(\wp^{c_\wp})
$$

Let $\Pi$ be a cuspidal automorphic representation of $\GLn(\A_E)$ and let $\Pi_f$ denote its finite part. Since $\Pi$ is cuspidal, $\Pi$ and (hence) $\Pi_f$ are generic. We denote by $\W(\Pi_f,\psi_f)$ the Whittaker model of $\Pi_f$ with respect to $\psi_f$ of $\A_{E,f}$, and by $\phi_f \mapsto W_{\phi_f}$ the isomorphism $\Pi_f \toeq \W(\Pi_f,\psi_f)$. We can decompose $\Pi_f$ as a tensor product of local representations $\Pi_f = \otimes_{w} \Pi_w$ over the set ${S_f(E)}$ of finite places of $E$ (see \cite[Theorem 3]{Flath}). For any ${w \in S_f(E)}$, the representation $\Pi_w$ is an irreducible admissible generic representation of $\GLn(E_w)$. Let $c_w = c(\Pi_w)$ be its mirahoric conductor. Consider $\mathfrak{n}(\Pi_f) = \prod_{w \in S_f(E)} \wp_w^{c_w}$, where $\wp_w$ is the prime ideal of $\O_E$ corresponding to $w$. Since all but finitely many $\Pi_w$ are unramified, $\mathfrak{n}(\Pi)$ is a well defined ideal of $\O_E$, called the \textit{mirahoric level} of $\Pi$. From the local theory, we know that the space of $K_1(\mathfrak{n}(\Pi))$-fixed vectors in $\Pi_f$ is one dimensional. \\

We now explain how to choose a particular form in the 1-dimensional space of $K_1(\mathfrak{n}(\Pi))$-fixed vectors. Write $\psi_f = \otimes_w \psi_w$. Let $\mathfrak{d}$ be the different of $E$. At each place $w$ dividing $\mathfrak{d}$, the conductor of $\psi_w$ is $\mathfrak{d}_w^{-1}$. Hence, since $\psi_w$ is ramified, one cannot normalize the choice of a newform $\phi_w \in \pi_w$ by the condition $W(I_n) =1$ on $\W(\pi_w,\psi_w)$ since we don't know that this value is non-zero. To avoid this complication, following \cite[Section 8]{JH24}, we consider a representative $d_w \in E_w$ of $\mathfrak{d}_{w}^{-1}$ for each place $w$ dividing $\mathfrak{d}$. Then $\psi_w = d_w \cdot \psi^\circ_w$ where $\psi^\circ_w$ is an unramified additive character of $E_w$ and $d_w \cdot \psi^\circ_w(x) := \psi^\circ_w(d_wx) $. We then consider the form $\phi^{\circ}_\Pi := \otimes_{w} \phi_w$ in $\Pi_f$ with:

\begin{itemize}
\item if $w$ does not divide $\mathfrak{d}$, then $\phi_w$ is the essential vector $\phi_{\Pi_w}^\circ$ of $\Pi_w$ with respect to $\psi_w$,
\item if $w$ divides $\mathfrak{d}$, then $\phi_w$ is $t_{d_w}(\phi_{\Pi_w}^\circ)$ where $\phi_{\Pi_w}^\circ$ is the essential vector of $\Pi_w$ with respect to $\psi_w^\circ$. In other words, $\phi_w$ is the form whose Whittaker function $W_{\phi,w} \in \W(\Pi_w,\psi_w)$ with respect to $\psi_w$ is given by:
$$
W_{\phi,w}(g) = W_{\Pi_w}^\circ(\mathrm{diag}(d_w^{n-1},\dots,d_w,1)g)
$$
where $W_{\Pi_w}^\circ$ is the essential vector in $\W(\Pi_w,\psi_w^\circ)$.
\end{itemize}
Then $\phi^{\circ}_\Pi$ is a $K_1(\mathfrak{n}(\Pi))$-fixed vector in $\Pi_f$ which is called the \textit{essential vector} of $\Pi$. \\

\subsection{Cohomology groups and Hecke algebras}

\subsubsection{Irreducible algebraic representations of $\SO$}
\label{alg_irrep_SO3}
The irreducible algebraic representations of $\mathrm{SO}(3)$ are indexed by $\ell \in \Z_{\geq 0}$. For such a $\ell$, we denote by $(\tau_\ell,V_\ell)$ the corresponding irreducible representation of $\mathrm{SO}(3)$. It is a representation of dimension $2\ell + 1$. More precisely, following \cite[2.4.1]{Che22}, we fix a model as follows. $V_\ell$ is the quotient of the space of homogeneous polynomials over $\C$ of degree $\ell$ in variable $X_1,X_2,X_3$, by the subspace generated by $X_1^2+X_2^2+X_3^2$. The action of $\SO$ on this space is given by:
$\tau_\ell(g) \cdot P(\overline{X_1},\overline{X_2},\overline{X_3}) = P((\overline{X_1},\overline{X_2},\overline{X_3})g)$, for $g \in \SO$ and $P \in V_\ell$. A basis of $V_{\ell}$ is given by $\{ \mathbf{v}_i, \, -\ell \leq i \leq \ell \}$, where:
$$
\mathbf{v}_i = (\mathrm{sgn}(i)\overline{X_1} + \sqrt{-1}\cdot\overline{X_2})^{|i|}\overline{X_3}^{\ell-|i|}
$$
For each triplet $(j_1,j_2,j_3) \in \Z_{\geq0}^3$ such that $j_1+j_2+j_3 = \ell$, we also define:
$\mathbf{v}_{\left(\ell ;\left(j_{1}, j_{2}, j_{3}\right)\right)} = \overline{X_1}^{j_1}\overline{X_2}^{j_2}\overline{X_3}^{j_3} \in V_{\ell}
$. \\

\subsubsection{Irreducible algebraic representations of $\GL(E)$}
\label{alg_irrep}
\label{coefficients}

 Let $\K$ be a field and let $X^+(T_3)$ be the set of triplets $X^+(T_3)= \{ \mu = (m^+,m^-,v) \in \Z^3,\quad m^+\geq 0, \,m^-\geq 0\}$. For $\mu~=~(m^+,m^-,v) \in X^+(T_3)$, let $\mathcal{P}_{\mu}(\K) = \K[X,Y,Z ; A,B,C]_{m^+,m^-}$ be the space of 6-variables polynomials homogeneous of degree $m^+$ in $X,Y,Z$ and $m^-$ in $A,B,C$ with coefficients in $\K$, on which $g \in \GL(\K)$ acts by:
$$
\rho_{\mu}(g)(P(X,Y,Z ; A,B,C)) = (\det g)^{v} P((X,Y,Z)g ; (A,B,C){}^\top g^{-1}).
$$ 
We now consider the following differential operator:
$$
i_{m^+,m^-} = \frac{\partial^2}{\partial X \partial A} + \frac{\partial^2}{\partial Y \partial B} + \frac{\partial^2}{\partial X \partial A}:  \mathcal{P}_{\mu}(\K) \to \mathcal{P}_{\mu-(1,1,0)}(\K)
$$
It is equivariant with respect to $\rho_{\mu}$ and $\rho_{\mu-(1,1,0)}$. Let $L_{\mu}(\K)$ be the kernel of $i_{m^+,m^-}$. We thus obtain a representation $(\rho_{\mu},L_{\mu}(\K))$ of $\GL(\K)$. All the irreducible algebraic representations of $\GL(\K)$ are of the form $(\rho_{\mu},L_{\mu})$ where $\mu$ runs over $X^+(T_3)$ (see \cite[Theorem 13.1 and Claim 13.4]{RT:FC}). Note that the highest weight of $(\rho_{\mu},L_{\mu})$ is $\mu = \left( v+m^+, v, v-m^-\right)$, and that a vector of highest weight is given by $P_{\mu}^+ = X^{m^+}C^{m^-}$. Since $X^+(T_3)$ is in bijection with the set of dominant weights for $\GL$, by abuse we will sometimes refer to $\mu$ as the highest weight of $(\rho_{\mu},L_{\mu})$. \\

The dual weight of $\mu \in X^+(T_3)$ is defined to be $\mu^\vee := (m^-,m^+,-v) \in X^+(T_3)$. We then consider the pairing $\langle \cdot, \cdot \rangle_{\mu}: L_{\mu}(\K) \x L_{\mu^\vee}(\K) \to \K$ by:

\begin{equation}
\label{pairing_coefficients}
\langle \sum_{\underline{i}^+,\underline{i}^-} a_{\underline{i}^+,\underline{i}^-} X^{\underline{i}^+} Y^{\underline{i}^-} , \sum_{\underline{j}^+,\underline{j}^-} b_{\underline{j}^+,\underline{j}^-} X^{\underline{j}^+} Y^{\underline{j}^-}\rangle_{\mu} = \sum_{\underline{i}^+,\underline{i}^-} \binom{m^+}{i_1^+,i_2^+,i_3^+}^{-1} \binom{m^-}{i_1^-,i_2^-,i_3^-}^{-1}a_{\underline{i}^+,\underline{i}^-}  b_{\underline{i}^-,\underline{i}^+} 
\end{equation}
where $X^{\underline{i}^+}:= X_1^{i_1^+}X_2^{i_2^+}X_3^{i_3^+}$ and $Y^{\underline{i}^-}:=Y_1^{i_1^-}Y_2^{i_2^-}Y_3^{i_3^-}$. This paring is perfect and equivariant for the action of $\GL(\K)$:
$$
\langle \rho_{\mu}(g)(P),\rho_{\mu^\vee}(g)(Q)\rangle_{\mu} =\langle P,Q\rangle_{\mu}
$$
for $P,Q \in L_{\mu}(\K)$ and $g \in \GL(\K)$. \\

Now, let $E$ be a totally real field and $G_E = \mbox{Res}_{E/\Q} \mathrm{GL}_{3/E}$. Let $S_\inf = S_\inf(E)$ denote the set of archimedean places of $E$, i.e. the set of embeddings of $E$ into $\overline{\Q}$. Let $X^+(T_E) = X^+(T_3)^{ S_\inf}$, i.e. the set of tuples $\mu = (\mu_\t)_{\t \in S_\inf}$, with $\mu_\t =(m^+_\t,m^-_\t,v_\t) \in X^+(T_3)$ for all $\t \in S_\inf$. Let $E^g$ be the Galois closure of $E$ in $\overline{\Q}$ and suppose that $\K$ is a field extension of $E^g$. 

We define $L_{\mu}(\K):= \bigotimes_{\tau \in  S_\inf} L_{\mu_\tau}(\K)$. For each $\tau: E \to \overline{\Q}$, we can define a map $\tau: E \otimes_\Q \K \to \K$ by $ x \otimes k \mapsto \tau(x)k$ which induces a morphism $\tau: G_E(\K) \to \GL(\K)$. Then, we get an action of $G_E(\K)$ on $L_{\mu}(\K)$, given for $g \in G_E(\K)$ and for a pure tensor $P= \otimes_\tau P_\t \in L_{\mu}(\K)$, by:
$$
\rho_{\mu}(g)(P) = \bigotimes_{\t \in  S_\inf}  \rho_{\mu_\t}(\t(g))(P_\t).
$$
Thus, we get a representation $(\rho_{\mu},L_{\mu})$ of $\GL(\K)$. All the irreducible algebraic representations of $\GL(\K)$ are of this form, for $\mu \in X^+(T_E)$. By abuse, $\mu$ will sometimes be refered to as the highest weight of $L_{\mu}$. Moreover we have a perfect $\GL(\K)$-equivariant pairing:
$$
\langle \cdot, \cdot \rangle_{\mu}: L_{\mu}(\K) \x L_{\mu^\vee}(\K) \to \K
$$
obtained as a tensor product over $ S_\inf$ from the pairing (\ref{pairing_coefficients}). \\

\subsubsection{Integral structure and $p$-smallness.} We now suppose that $\K$ is some finite extension of $\Q_p$ and that $\O$ is its valuation ring. All the above definitions make sense over $\O$, and we get a representation $\rho_{\mu}$ of the group scheme $G_E = \mathrm{Res}_{\O/\Z_p} \mathrm{GL}_{3/\O}$ on the module $L_{\mu}(\O)$ consisting of polynomials in $L_{\mu}(\K)$ with coefficients in $\O$. We see from the expression of $\langle \cdot, \cdot \rangle_{\mu}$ that the $\O$-dual of $L_{\mu}(\O)$ in $L_{\mu^\vee}(\K)$ may not be $L_{\mu^\vee}(\O)$. To avoid this problem, we will say that $\mu = (\mu_\tau)_{\t \in S_\inf}$ is $p$\textit{-small} if for each $\tau \in  S_\inf$: 
$$
p > \mathrm{max}(m^+_\tau, m^-_\t)
$$
where $\mu_\tau = (m^+_\t,m^-_\t,v_\t)$. For a given weight $\mu$, this condition excludes a finite number of small primes. Then, as long as $\mu$ is $p$-small $\langle \cdot, \cdot \rangle_{\mu}$ restricts to a pairing defined on $\O$:
$$
\langle \cdot, \cdot \rangle_{\mu}: L_{\mu}(\O) \x L_{\mu^\vee}(\O) \to \O.
$$
It follows from (\ref{pairing_coefficients}), and from the perfectness of the pairing on $\K$, that this pairing is perfect on $\O$. See Section 1 of Polo-Tilouine in \cite{CSV} for a more theoretical explanation of these facts (be careful though that $p$-small there means $p \geq m^+_\tau + m^-_\t + 2$, for $\t \in S_\inf$).

Let $\K$ be a $p$-adic field $\K$, and $\O$ be its valuation ring. In this subsection, we consider the algebraic group $G_E = \mathrm{Res}_{E/\Q}(\mathrm{GL}_{n/E})$ for any integer $n \geq 1$ and any number field $E$. When $n=3$ and $E$ is totally real, which is our main case of concern in this thesis,  we have introduced in the precedent subsection a precise model $L_{\mu}$ for the algebraic representation of $G_E$ of highest weight $\mu$ (as explained $\mu \in X^+(T_E)$ is not strictly speaking a dominant weight but some convenient avatar of it, but since $X^+(T_E)$ is in bijection with the set of dominant weights for $G_E$, we call it the highest weight of $L_{\mu}$ by abuse). For a general $n \neq 3$ and any number field $E$, one can similarly define, for each dominant weight $\mu$ of $G_E$, a representation $L_{\mu}(\O)$ of $G_E(\O)$ which extends over $\K$ to $L_{\mu}(\K)$, the irreducible algebraic representation of $G_E(\K)$ of highest weight $\mu$ (see for example \cite{Hi98}). \\

\subsubsection{Adelic variety and sheaves}
\label{variety_sheaves}

Recall that $G_E = \mathrm{Res}_{E/\Q}(\mathrm{GL}_{3/E})$. For any open compact subgroup $K_f$ of $G_E(\A_{f})$, the adelic variety of level $K_f$ for $G_E$ is defined as:
$$
Y(K_f):= G_E(\Q) \bs G_E(\A) / K_f K_\inf
$$
where $K_\inf = \SO \R_+^\x$. Let $\mu \in X^+(T_E)$ be a dominant weight and $A$ denote $\O$, $\K$ or $\C$. Let $\L_\mu(A)$ be the locally constant sheaf of $A$-modules on $Y(K_f)$ attached to the representation $L_\mu(A)$. We have an inclusion of sheaves $\L_\mu(\O) \subset \L_\mu(\K) \subset \L_\mu(\C)$. \\

\subsubsection{Cohomology groups}
\label{cohomology_groups}

Let $A$ denote the fields $\K$ or $\C$. We consider the Betti cohomology groups $H^q(Y(K_f), \L_{\mu}(A))$ and the Betti cohomology with compact support $H^5_c(Y(K_f), \L_{\mu}(A))$. There is a natural map:
\begin{equation}
\label{std_to_compact}
i_A: H^\bullet_c(Y(K_f), \L_{\mu}(A)) \to H^\bullet(Y(K_f), \L_{\mu}(A))
\end{equation}
The interior cohomology $H^5_!(Y(K_f), \L_{\mu}(A))$ is defined to be the image of $i_A$. If $A$ is the ring $\O$, and if $?$ denotes $\emptyset, c \mbox{ or }!$ then $H_?^\bullet(Y(K_f), \L_{\mu}(\O))$ denotes the $\O$-torsion free part of the corresponding cohomology groups with coefficients in $\O$. \\

The cuspidal cohomology $H^\bullet_{cusp}(Y(K_f), \L_{\mu}(\C))$ is defined to be the following $(\g,K_\inf)$-cohomology group (or relative Lie algebra cohomology, see \cite[Chapter I]{BW00}):
$$
H^\bullet_{cusp}(Y(K_f), \L_{\mu}(\C)) = H^\bullet(\g,K_\inf ;\mathcal{A}_{cusp}(G_E(\Q)\bs G_E(\A)/K_f)\otimes L_{\mu}(\C))
$$
where $\mathcal{A}_{cusp}(G_E(\Q)\bs G_E(\A)/K_f, \w)$ is the space of $K_f$-fixed cusp forms on $G_E(\A)$. A priori, the cuspidal cohomology is just contained in $H^\bullet(Y(K_f), \L_{\mu}(\C))$. In fact, there is an injection:
$$
H^\bullet_{cusp}(Y(K_f), \L_{\mu}(\C)) \inj H^\bullet_!(Y(K_f),\L_{\mu}(\C))
$$
More precisely, there exists a canonical map (see for exemple \cite[2.1]{Hi99}):
$$
s: H^q_{cusp}(Y(K_f), \L_{\mu}(\C)) \inj H^q_c(Y(K_f),\L_{\mu}(\C))
$$
which a section of $i_\C: H^q_c(Y(K_f),\L_{\mu}(\C)) \to H^q(Y(K_f),\L_{\mu}(\C))$. Hence, $H^q_{cusp}(Y(K_f), \L_{\mu}(\C))$ can be viewed as a subgroup of $H^q_{c}(Y(K_f), \L_{\mu}(\C))$. 
We define the cuspidal cohomology groups with coefficients in $\K$ to be:
$$
H^\bullet_{cusp}(Y(K_f), \L_{\mu}(\K)) = H^\bullet_{cusp}(Y(K_f), \L_{\mu}(\C)) \cap H^\bullet_c(Y(K_f), \L_{\mu}(\K))
$$
It follows from \cite[Théorème 3.19]{Clozel90} that:
$$
H^\bullet_{cusp}(Y(K_f), \L_{\mu}(\K)) \otimes_\K \C = H^\bullet_{cusp}(Y(K_f), \L_{\mu}(\C))
$$
We then define the cuspidal cohomology groups with coefficients in $\O$ to be:
$$
H^\bullet_{cusp}(Y(K_f), \L_{\mu}(\O)) = H^\bullet_{cusp}(Y(K_f), \L_{\mu}(\C)) \cap H^\bullet_c(Y(K_f), \L_{\mu}(\O))
$$
In particular $H^\bullet_{cusp}(Y(K_f), \L_{\mu}(\O))$ is torsion free. \\

\subsubsection{Hecke algebras}
\label{hecke_corr}

Let $A$ denote $\O$, $\K$ or $\C$. Let $K_f = \prod_w K_w$ be some open compact subgroup of $G_E(\A_f)$ and let:
$$
S_{K_f}~=~\{ w \mbox{ finite place of }E\mbox{, s.t. } K_{w} \neq \GLn(\O_w)\}.
$$
This is a finite set. Let $S_p = \{ w \mbox{ finite place of E, s.t. } w~\mid~p \}$. We suppose that $p$ is outside the level of $K_f$, i.e that $S_p \cap S_{K_f} =\varnothing$, and we put $S :=  S_{K_f} \sqcup S_p$. 

Let $w \notin S$, and let $\varpi_w$ be an uniformizer of $E_w$. We define the Hecke operators $T_{w,i}$ at $w$ (for $i = 1,2,3$) as the characteristic functions of the double coset  $\GL(\O_w)\mathrm{diag}(\varpi_w I_i, I_{n-i})\GL(\O_w) \subset \GL(E_w)$. We will write $S_w$ for $T_{w,3}$. The spherical abstract Hecke algebra outside of $S$ with coefficients in $A$ is defined as the following tensor product:
$$
\H(K_f;A) = \bigotimes_{w \notin S} A[T_{w,1},T_{w,2},{S_{w}}^\pm]
$$

There is a well known action of $\H(K_f;A)$ (see for instance \cite[\S3.2.3]{thesis}) on the various cohomology groups $H^\bullet_?(Y(K_f), \L_\mu(A))$ (for $? = \emptyset, c, !$ or $cusp$) described in \S\ref{cohomology_groups}. Moreover, all the comparison maps described in \S\ref{cohomology_groups} are equivariant for this action. Finally, since we don't consider Hecke operators at $p$, the maps $H^\bullet_?(Y(K_f), \L_{\mu}(\O)) \to H^\bullet_?(Y(K_f), \L_{\mu}(\K))$ are $\H(K_f;\O)$-equivariant. \\

We denote by $h(K_f;\O)$ the Hecke algebra acting faithfully on the cohomology:
$$
h(K_f;\O) := \mathrm{Im}(\H(K_f;\O) \to \End_\O H^\bullet(Y(K_f),M_\mu(\O)))
$$
It is a finite commutative $\O$-algebra. Thus, $h(K_f;\O)$ is semi-local and we have a decomposition $h(K_f;\O) \simeq \prod_{\m} h(K_f;\O)_\m$, where $\m$ ranges through maximal ideals of $h(K_f;\O)$ (there are finitely many of them), and $h(K_f;\O)_\m$ be the completion of $h(K_f;\O)$ at $\m$. In this paper, whenever such a maximal ideal $\m \subset h(K_f;\O)$ is fixed, we denote by ${\TT} = h(K_f;\O)_\m/(\O-{tors})$ the torsion-free quotient of $h(K_f;\O)_\m$.

\subsection{Cohomological cuspidal automorphic representations}

\label{GL3_notations}

From now, we assume that $E$ is a totally real number field.

\subsubsection{Cohomological automorphic representations} Let $K_3 := \R_{>0} \cdot \SO \subset \GL(\R)$, and $K_\inf= \prod_{v\mid \inf} K_3 \subset \U_3(\R) \simeq \GL(F_\inf)$. Let $\g_{3}$ be the Lie algebra of $\GL(\R)$, and let $\g = \oplus_{v\mid \inf}  \g_{3}$ be the Lie algebra of $\GL(F_\inf)$. We say that a cuspidal automorphic representation $\Pi$ of $G_E$ is cohomological of cohomological weight $\mu \in X^+(T_E)$ if the following relative Lie algebra cohomology group (see \cite[Chapter I]{BW00}):
$$
H^q(\g,K_\inf ; \Pi_\inf \otimes M_\mu(\C))
$$
is non trivial for some $q \geq 0$. In that case, it is non-zero if and only if $b\leq q \leq t$, where $b=2t$ and $t= 3d$ are called the bottom and top degree of the cuspidal range (see \cite[Lemme 3.14]{Clozel90}). We denote by $\mathrm{Coh}(G_E,\mu)$ the set of such representations. For each open compact subgroup $K_f$ of $G_E(\A)$, we also define $\mathrm{Coh}(G_E,\mu,K_f)$ to be the subset of $\mathrm{Coh}(G_E,\mu)$ consisting of cohomological representations $\Pi$ such that $\Pi_f$ has non-zero $K_f$-fixed vectors. By definition, the property of being cohomological for a representation $\Pi$ only depends on its archimedean part $\Pi_\infty$. The cohomological weight of a cohomological representation $\Pi$ is unique \cite[Section 3.5]{Clozel90} and can be expressed from the Langlands parameter of $\Pi_\infty$ (see for example \cite[Proposition 4.1]{HN20}).

The cohomological weight $\mu$ of a cohomological representation $\Pi$ is necessarily \textit{pure} \cite[Lemme de pureté 4.9]{Clozel90}, in the sens of the following definition:
\begin{definition}[Pure weights]
Let $\mu \in X^+(T_E)$ be a highest weight for $G_E$. Write $\mu= (\mu_\tau)_{\tau \in  S_\inf}$, with $\mu_\tau = (m^+_\tau,m^-_\tau,v_\tau) \in X^+(T_3)$ for each $\tau \in S_\inf$. We say that $\mu$ is \textit{pure} if:
\begin{itemize}
\item $v_\tau$ does not depend on $\tau \in  S_\inf$ ;
\item $m^+_\tau=m^-_\tau$ for each $\tau \in  S_\inf$.
\end{itemize}
\end{definition}

Let $\Pi \in \mathrm{Coh}(G_E,\mu)$ be a cohomological cuspidal automorphic representation. Write $\Pi_\inf = \bigotimes_{\tau \in  S_\inf} \Pi_\tau$ for the archimedean part of $\Pi$. For each archimedean place $\tau \in  S_\inf$, on has that:
$$
\Pi_\tau = \mathrm{Ind}_{P_{2,1}(\R)}^{\GL(\R)} (D_{\ell_\tau} \otimes \e_\tau) \otimes | \cdot |^{v_\tau}
$$
for some quadratic character $\e_\tau$ of $\R^\x$. Here $\ell_\tau= 2m_\tau + 3$ is the minimal $\mathrm{SO}(3)$-type of $\Pi_\tau$, and $D_{\ell_\tau}$ is the discrete series representation of $\mathrm{GL}_2(\R)$ of weight $\ell_\tau$ (see \cite[Theorem 2.1(3)]{Che22} or \cite[\S 4.1]{HN20}). Note that the central character $\w_{\Pi_\tau}$ of $\Pi_\tau$ satisfies:
$$
\w_{\Pi_\tau}|_{\{\pm1\}} = \e_\tau \cdot \mathrm{sgn}
$$
where $\mathrm{sgn} : \{ \pm1\} \to \{ \pm1\}$ is the non-trivial character. Moreover, for each archimedean place $\tau \in  S_\inf$, the minimal $\mathrm{SO}(3)$-type of $\Pi_\tau$ is $\ell_\tau = 2m_\tau + 3$ (see \S\ref{alg_irrep_SO3}).

\subsubsection{Hecke eigensystems}
 Let $\Pi = \otimes_w \Pi_w$ be a cuspidal automorphic representation of $\GL(\A_E)$ of level $K_f$ (i.e. with non-trivial $K_f$-fixed vectors), and let $\H(K_f;\C)$ be the abstract Hecke algebra defined in \S\ref{hecke_corr}. Then, the Hecke eigensystem $\l_\Pi$ associated with $\Pi$ is the morphism of $\C$-algebras $\l_\Pi: \H(K_f;\C) \to \C$ such that for all $w \notin S$ and $i\in \{1,2,3 \}$, $\l_\Pi(T_{w,i})$ is defined to be the complex scalar through which $T_{w,i}$ acts on the $1$-dimensional $\C$-vector space of $K_w$-fixed vectors in $\Pi_w$.

Let $h(K_f;\C)$ be the Hecke algebra acting faithfully on the cohomology with coefficients in $M_\mu(\C)$. Then, if $\Pi \in \mathrm{Coh}(G_E,\mu,K_f)$ is cohomological, the Hecke eigensystem $\l_\Pi : \H(K_f;\C) \to \C$ attached to $\Pi$ factorizes into a $\C$-algebra morphism $\l_\Pi : h(K_f;\C) \to \C$. Finally, let $\K$ be a $p$-adic field containing $E$ and the rationality field $\Q(\Pi)$ of $\Pi$ and let $\O$ be its valuation ring. Let $h(K_f;\O)$ be the cohomological Hecke algebra with coefficients in $\O$, defined in \ref{hecke_corr}. Then, the image of $h(K_f;\O)$ through $\l_\Pi$ is included in $\O$, yielding an integral Hecke eigensystem $\l_\Pi : h(K_f;\O) \to \O$. \\

\subsubsection{Galois representations}
\label{galois_reps}
Let $\Pi \in \mathrm{Coh}(G,\mu,K_f)$ be a cohomological cuspidal automorphic representation of $\GLn(\A_E)$. Let $\K$ be some sufficiently large $p$-adic field and $\O$ be its valuation ring.  Let $S$ be the finite set of places of $E$ containing the places where $\Pi$ is ramified and the places dividing $p$. The Galois representation associated with $\Pi$ according to the Langlands philosophy has been constructed by \cite[Theorem A]{HLTT16} (see also \cite[Corollary V.4.2]{Scholze15}). This is the following theorem:
\begin{theorem}
There exists a unique continuous semisimple Galois representation $\rho_{\Pi} : \Gal(\overline{E}/E) \to \GLn(\K)$ which is unramified outside of $S$, and such that for all $w \notin S$, the characteristic polynomial of $\rho_\Pi(\Frob_{w})$ is:
$$
\sum_{i=0}^n (-1)^i q_w^{i(i-1)/2} \l_\Pi(T_{w,i}) X^{n-i}
$$
where $\l_\Pi : \H(K_f;\C) \to \C$ is the Hecke eigensystem associated with $\Pi$.
\end{theorem}
One can find a $\Gal(\overline{E}/E)$-invariant $\O$-lattice in the representation space of $\rho_{\Pi}$. We then obtain a Galois representation $\rho_{\Pi} : \Gal(\overline{E}/E) \to \GLn(\O)$. If the residual representation $\overline{\rho_{\Pi}}$ is irreducible then there is a unique homothety classe of such an $\O$-lattice, and $\rho_{\Pi} : \Gal(\overline{E}/E) \to \GLn(\O)$ is unique. \\

\subsubsection{Structure of the cuspidal cohomology}

We now give the structure of the cuspidal cohomology. Let $q \in \Z_{\geq 0}$ and $\mu$ be some dominant weight for $G_E$. The cuspidal cohomology admits the following direct sum decomposition:
\begin{equation}
\label{deco_coho_cusp}
H^\bullet_{cusp}(Y(K_f), \L_{\mu}(\C)) = \bigoplus_{\Pi \in \mathrm{Coh}(G_E,\mu,K_f)} H^\bullet(\g, K_\inf  ; \Pi_{\inf} \otimes L_{\mu}(\C)) \otimes \Pi_f^{K_f}.
\end{equation}

In particular, it follows from \cite[Section VII]{BW00} and \cite[Section 4.5]{Clozel90} that the cuspidal cohomology group $H^q_{cusp}(Y(K_f), \L_{\mu}(\C))$ doesn't vanish only when $\mu$ is pure and $q$ belongs to the so-called Borel-Wallach interval $[b_E,t_E]$, for $b_E = 2d$ and $t_E = 3d$. In particular, as already said, the set $\mathrm{Coh}(G_E,\mu)$ is empty unless the weight $\mu$ is pure, and $\mathrm{Coh}(G_E,\mu,K_f)$ is a finite set. This decomposition justifies the nomenclature: a cuspidal automorphic representation is cohomological of cohomological weight $\mu$ if and only if it appears as a subspace of the cuspidal cohomology with coefficients in $\L_{\mu}(\C)$. A cohomological cuspidal automorphic representation $\Pi$ occurs in the cuspidal cohomology of degree $q$ for all degree $q \in [b_E,t_E]$, with multiplicity equal to the dimension of $H^q(\g, K_\inf  ; \Pi_{\inf} \otimes L_{\mu}(\C))$. Finally, recall that the cuspidal cohomology of level $K_f$ is endowed with an action of $\H(K_f,\C)$. Then, the above decomposition is a decomposition of $\H(K_f,\C)$-modules. In particular, $H^\bullet_{cusp}(Y(K_f), \L_{\mu}(\C))$ is a semi-simple $\H(K_f,\C)$-module. 

\subsection{Eichler-Shimura maps and middle-degree automorphic periods}

\subsubsection{Generators of the $(\g,K_\inf)$-cohomology}
\label{chen_generators}
 In this paragraph, we construct some explicit elements in the $(\g,K_\inf)$-cohomology groups appearing on the right hand side of (\ref{deco_coho_cusp}). Let $\Pi \in \mathrm{Coh}(G,\mu,K_f)$ and let $\Pi_\inf$ be its archimedean part. Let $q \in \N$ be a non-negative integer. From the Kunneth formula, we have the following direct sum decomposition into tensor products of local cohomology groups:
$$
\begin{aligned}
H^q(\g, K_\inf ; \Pi_\inf \otimes L_{\mu}(\C)) & = \bigoplus_{\sum_\tau  q_\tau = q } \left( \bigotimes_{\tau \in S_\inf} H^{q_\tau}(\g_\tau, K_\tau^\circ ; \Pi_{\tau} \otimes L(\mu_\tau ;\C)) \right)
\end{aligned}
$$
where the direct sum is indexed by tuples $(q_\tau)_{\tau \in  S_\inf}$ of non-negative integers such that $\sum_{\tau \in  S_\inf}  q_\tau~=~q$. Let $\tau$ be an archimedean place of $E$. From a general result of Clozel (see \cite[Lemme  3.14]{Clozel90}), we know that:
$$
H^{q_\tau}(\g_3, K_3; \Pi_{\tau} \otimes L_{\mu_\tau}(\C))=\left\{
    \begin{array}{ll}
        \C & \mbox{if } q_\tau = 2,3 \\
        0 & \mbox{otherwise}
    \end{array}
\right.
$$
where $\g_3 = \glr$ and $K_3 = \mathrm{SO}(3)\R_+^\x$. In particular, together with the Kunneth formula, this gives that:
$$
\dim_\C H^q(\g, K_\inf ; \Pi_\inf \otimes L_{\mu}(\C)) = \binom{t-b}{q-b}
$$
{} \\

For the sake of brevity, the local cohomology groups $H^{q_\tau}(\g_3, K_3; \Pi_{\tau} \otimes L_{\mu_\tau}(\C))$ are denoted by $H^{q_\tau}_\tau$ in the following. Let $J \subset  S_\inf$ be a subset of the archimedean places. We define:
$$
H^q(\g, K_\inf  ; \Pi_\inf \otimes L_{\mu}(\C))_{J}:= \bigotimes_{\tau \in J} H^2_\tau  \otimes \bigotimes_{\tau \notin J} H^3_\tau
$$
It is a $1$-dimensional $\C$-vector space. A generator of this space can be constructed by choosing generators $[\Pi_v]_i$ of $H^i_\tau$ for $i = 2, 3$ at each archimedean place $\tau$. Chen \cite[Lemma 2.2]{Che22} has constructed explicit non-zero elements (thus generators)  $[\Pi_\tau]_i$ in $H^i_\tau$ and proven that they yields the correct archimedean local factors (the so-called $\G$-factors) for the adjoint $L$-function. Using these generators, we then consider and fix the following generator of $H^q(\g, K_\inf ; \Pi_{\inf} \otimes L_{\mu}(\C))_J$:
$$
[\Pi_\inf]_J = \bigotimes_{\tau \in J} [\Pi_\tau]_2 \otimes \bigotimes_{v \notin J} [\Pi_\tau]_3.
$$
for $J \subset  S_\inf$. In particular, when $E$ is a real quadratic field, we obtain the following two elements:
$$
[\Pi_\inf]_{\{ \tau\}} = [\Pi_\tau]_2 \otimes [\Pi_{\s\tau}]_3, \quad \mbox{ and } \quad [\Pi_\inf]_{\{ \s\tau\}} = [\Pi_\tau]_3 \otimes [\Pi_{\s\tau}]_2,
$$
which are respective generators of $H^5_{\{ \tau\}}$ and $H^5_{\{ \s\tau\}}$. \\

\subsubsection{Eichler-Shimura maps}
\label{eichler-shimura_maps}

For each $J\subset  S_\inf$, we also denote by $H^q_{cusp}(Y_E(K_f),\L_{\mu}(\C))_J$ the subspace of $H^q_{{cusp}}(Y_E(K_f),\L_{\mu}(\C))$ defined as the direct sum:
$$
\bigoplus_{\Pi \in \mathrm{Coh}(G,\mu,K_f)}H^q(\g, K_\inf ; \Pi_\inf \otimes L_{\mu}(\C))_{J} \otimes \Pi_f^{K_f}.
$$

Let $\mu \in X^+(T_E)$ be a pure weight. Let $K_f$ be some open compact subgroup of $G_E(\A_{f})$. The space $S_{\mu}(K_f)$ of automorphic cusp forms of weight $\mu$ and level $K_f$ is defined by:
$$
S_{\mu}(K_f) := \bigoplus_{\Pi} \Pi_f^{K_f},
$$
where $\Pi$ runs through $\mathrm{Coh}(G_E,\mu,K_f)$, i.e. through the cuspidal automorphic representations of $\GL(\A_E)$ of level $K_f$ whose archimedean part is given by $\Pi_\inf = \bigotimes_{\tau \in  S_\inf} \Pi_\tau$, with:
$$
\Pi_\tau = \mathrm{Ind}_{P_{2,1}(\R)}^{\GL(\R)} (D_{\ell_\tau} \otimes \e_\tau) \otimes | \cdot |^{v_\tau}
$$
for some quadratic character $\e_\tau$ of $\R^\x$. Let $J \subset  S_\inf$ be a subset of the archimedean places. Then, the Eichler-Shimura map of type $J$: 
$$
\d_J^q :S_{\mu}(K_f) \inj H^{b + | J |}_{{cusp}}(Y_E(K_f),\L_{\mu}(\C))
$$
is defined by sending an element $\phi_f \in \Pi_f^{K_f}$ to $\phi_f \otimes [\Pi_\inf]_J \in H^{b + | J |}_{{cusp}}(Y_E(K_f),\L_{\mu}(\C))_J$, where $[\Pi_\inf]_J$ is the generator of $H^q(\g, C_\inf ; \Pi_{\inf} \otimes L_{\mu}(\C))_J$ specified in the preceding paragraph. Thus, $\d_J^q$ is an injective $\C$-linear map whose image is $H^q_{{cusp}}(Y_E(K_f),\L_{\mu}(\C))_J$. It is equivariant for the action of $\H(K_f,\C)$ on both side. Moreover, for each $\Pi \in \mathrm{Coh}(G,\mu,K_f)$, it induces an isomorphism:
$$
\d_J^q : \Pi_f^{K_f} \toeq H^q_{{cusp}}(Y_E(K_f),\L_{\mu}(\C))_{J}[\Pi_f].
$$

\subsubsection{Eichler-Shimura maps for real quadratic fields} We here specify the above presentation to the special case where $E$ is a real quadratic field ($d=2$). We recall that in this case, we have fixed an embedding $\tau: E \inj \C$ and that $\s$ is the generator of the Galois group $\Gal(E/\Q)$, so that $ S_\inf$ can be identified with $\{\tau, \s\tau\}$.  In this situation, the cohomology is non-zero in degrees $q=4,5,6$. In middle degree $q=5$, Kunneth's formula is written:
$$
H^5(\g, K_\inf  ; \Pi_\inf \otimes L_{\mu}(\C)) = H^5(\g, K_\inf  ; \Pi_\inf \otimes L_{\mu}(\C))_{\{\tau\}} \oplus H^5(\g, K_\inf  ; \Pi_\inf \otimes L_{\mu}(\C))_{\{\s\tau\}}
$$
where $H^5(\g, K_\inf  ; \Pi_\inf \otimes L_{\mu}(\C))_{\{\tau\}} = H^{2}_{\tau} \otimes H^{3}_{\s\tau}$ and $H^5(\g, K_\inf  ; \Pi_\inf \otimes L_{\mu}(\C))_{\{\s\tau\}} = H^{3}_{\tau} \otimes H^{2}_{\s\tau}$. Consequently, there are two Eichler-Shimura maps:
$$
\d^5_J: S_{\mu}(K_f) \toeq H^5_{cusp}(Y_E(K_f),\L_{\mu}(\C))_J \subset H^5_{cusp}(Y_E(K_f),\L_{\mu}(\C))
$$
for $J = \{\tau \}, \{\s\tau \}$. \\

\subsubsection{The Galois involution $\s$}
\label{s_involution}
We now assume that $\Pi$ is a cohomological cuspidal automorphic representation which is self-conjugate, i.e. $\Pi^\s \simeq \Pi$. Let $\mathfrak{n}(\Pi)$ be the mirahoric level of $\Pi$ and let $K_f$ be the mirahoric subgroup $K_1(\mathfrak{n}(\Pi))$ of level $\mathfrak{n}(\Pi)$, and choose a $K_f$-newform $\phi_f \in \Pi_f$ to be the essential vector $\phi_{\Pi}$. \\

Let $\s: G_E(\A) \to G_E(\A)$ denote the conjugation-duality involution, defined by $\s:= g \mapsto \s(g)$.  We may sometimes write $g^\s:= \s(g)$. Since $G_E(\Q)$, $K_\inf$ and $K_f$ are invariant by $\s$, the involution $\s$ induces an involution of the adelic variety $Y(K_f)$, which we also denote $\s$. 

Let $\mu \in X^+(T_E)$ and let $\mu^\s:= (\mu_{\s\tau} \mu_{\tau}) \in X^+(T_E)$. Assume that $\mu^\s = \mu$. Let $A$ denote $\O$, $\K$ or $\C$. We let $\s$ acts on $L_{\mu}(A)$ by $\s:  P_\tau \otimes P_{\s\tau} \mapsto P_{\s\tau} \otimes P_{\tau}$.

The involutive action of $\s$ on $Y(K_f)$ and on $L_{\mu}(A)$ induces a $A$-linear involutive map on the cohomology groups, which we also denote $\s$:
$$
\s : H^\bullet(Y(K_f), \L_{\mu}(A)) \to H^\bullet(Y(K_f), \L_{\mu}(A))
$$

As in paragraph~\ref{hecke_corr}, we can similarly construct involutions $\s$ on the compactly supported cohomology. One can then check that the arrow $\i_A$ from the compactly supported cohomology to the cohomology is equivariant for the action of $\s$ on both sides. Therefore, we obtain an $A$-linear map on the interior cohomology: 
$$
\s : H^\bullet_!(Y(K_f), \L_{\mu}(A)) \to H^\bullet_!(Y(K_f), \L_{\mu}(A)).
$$
Finally, the cuspidal cohomology $H_{cusp}^\bullet(Y(K_f), \L_{\mu}(A))$ is invariant by the action of $\s$. \\

Let $S$ be the set of finite places associated with $K_f$ in \S\ref{hecke_corr} (i.e. containing the finite places of $E$ above $p$ and the places where $K_f$ is not hyperspecial). If $w$ is a finite place corresponding to some prime ideal $\mathfrak{q}$ of $E$, we denote by $w^\s$ the finite place of $w$ corresponding to the prime ideal $\s(\mathfrak{q})$, where $\s$ is the non-trivial Galois involution of $E$. For our choice of $K_f$, one has $\s(S) = S$. We get an involution of $\O$-algebra $\s : T \mapsto T^\s$ on $\H({K_f};\O)$, by sending a double coset $K_w x K_w$ ($w \notin S$, $x \in \GL(E_w)$) to $K_{w^\s} \s(x) K_{w^\s}$. One checks that, we simply have $T_{w,i}^\s = T_{w^\s,i}$, for $i=1,2,3$. One can also check that the action of $\s$ is semi-linear with respect to the Hecke action, i.e. that for all $T \in \H({K_f};\O)$, we have:
$$
T \circ \s = \s \circ T^\s
$$
in $\End_\O(H^\bullet(Y(K_f), \L_{\mu}(\O)))$. In particular, if $\Si \in \mathrm{Coh}(\mathrm{GL}_{n/E},\mu,K_f)$ is a cohomological cuspidal automorphic representation,  then $\s$ sends the $[\Si]$-isotypic part of $H_{cusp}^\bullet(Y(K_f), \L_{\mu}(\C))$ to the $[\Si^\s]$-isotypic part.

\subsubsection{The $\s$-periods}
\label{s_periods}
Let us write $H^5(A) = H^5_{cusp}(Y_E(K_f),\L_{\mu}(A))[\Pi]$ for the $\Pi$-isotypic component of $H^5_{cusp}(Y_E(K_f),\L_{\mu}(A))$. It follows from (\ref{deco_coho_cusp}) and the choice of $K_f$ that $H^5(\C)$ is a 2-dimensional $\C$-vector space, and is a direct sum of the two $1$-dimensional subspaces $H^5_{\{\t\}}$ and $H^5_{\{\s\t\}}$. Since $\Pi$ is self-conjugate, $H^5(\C)$ is invariant by $\s$. Moreover, $\s$ exchanges $H_{\{\tau\}}^5$ and $H_{\{\s\tau\}}^5$, and thus is not trivial on $H^5(\C)$. Hence, since the $\O$-lattice $H^5(\O) \subset H^5(\C)$ is invariant by $\s$, the two eigenspaces corresponding to the $\pm1$-eigenvalues for $\s$:
$$
H^5(\C)[\pm] = H^5_{cusp}(Y_E(K_f),\L_{\mu}(\C))[\Pi; \pm]
$$
are $1$-dimensional vector spaces over $\C$ with integral structure over $\O$. We then define two maps (one for $+$, one for $-$):
$$
\d^\pm_\s : \Pi_f^{K_f} \to H^5(\C)[\pm]
$$
by $\d^\pm_\s : = \d_{\{\tau\}}^5 \pm \s \circ \d_{\{\tau\}}^5$. Let $\xi_\s^\pm$ be an $\O$-base of $H^5(\O)[\pm]$. Then the two $\s$-periods $\Om_5(\Pi, \s, +)$ and $\Om_5(\Pi, \s,-)$ of $\Pi$ are defined to be the complex numbers $\Om_5(\Pi, \s, \pm) \in \C^\x$ such that:
$$
\d^\pm_\s(\phi_f) = \Om_5(\Pi, \s, \pm) \cdot \xi_\s^\pm
$$
Since the definition of $\Om_5(\Pi, \s, \pm)$ depends on the choice of an $\O$-base $\xi_\s^\pm$, the periods are only defined up to multiplication by an element of $\O^\x$. These two automorphic periods are called the $\s$-periods of $\Pi$. \\

\subsection{An adjoint $L$-value formula involving the $\s$-periods}
\label{adjoint_L_formula}
We now state an adjoint $L$-value formula relating the adjoint $L$-value of a self-conjugate cohomological cuspidal representation of $\GL(E)$ to its $\s$-periods. Let $p$ be an odd prime number. Let $\K$ be some sufficiently large $p$-adic field, $\O$ its valuation ring, and $\wp$ its prime ideal. Let $\Pi$ be a self-conjugate cohomological automorphic cuspidal representation of $\GL(\A_E)$ of mirahoric level $\n$ and cohomological weight $\mu_E = (\mu,\mu) \in X^+(T_E)$. Let $K_f = K_1(\n)$ be the mirahoric subgroup of level $\n$. Let $h(K_f;\O)$ be the spherical Hecke algebra of level $K_f$ acting faithfully on the cohomology, and let $\m_\Pi$ denote the maximal ideal of $h(K_f;\O)$ corresponding to $\Pi$ and let $\TT:= h(K_f;\O)_{\m_{\Pi}}/(\O-\mathrm{tors})$. Let $\l_\Pi: \TT \to \O$ be the Hecke eigensystem associated with $\Pi$, and let $\eta_{\l_\Pi}(M)$ be the congruence number of $\l_\Pi$ on the $\TT$-module:
$$
M = H_{cusp}^5(Y_E(K_f), \L_{\mu_E}( \O))_{\m_\Pi},
$$
defined in \S\ref{congruence_modules}. Since $\Pi$ is self-conjugate, $M$ is endowed with a semi-linear action of the Galois involution $\s$. We prove the following theorem:

\begin{theorem}
\label{adjoint_L_value}
Assume that the Galois representation associated with $\Pi$ is residually absolutely irreducible. Assume that the cohomological weight $\mu_E \in X^+(T_E)$ of $\Pi$ is $p$-small, and that $p \nmid 6N_{E/\Q}(\mathfrak{n}) h_E(\mathfrak{n})D_E$. Then:
$$
\eta_{\l_\Pi}(M)[\pm] \quad \sim \quad \frac{\Lambda^{imp}(\Pi,\Ad,1)}{\Om_5(\Pi,\s,\pm)\cdot\Om_5(\Pi^\vee,\s,\mp)}
$$
where $\eta_{\l_\Pi}(M)[\pm]$ is the $\pm$-part for the action of $\s$ (in the sense of \S\ref{congruence_modules_with_an_involution}).
\end{theorem}

In the above theorem, $\Lambda^{imp}(\Pi,\Ad,s)$ is the completed imprimitive adjoint $L$-function of $\Pi$, defined as an Euler product including the $\G$-factor at the archimedean places, and imprimitive local factors at places of ramification for $\Pi$. In fact, an analogous theorem is proven in \cite{TR_BC} for conjugate self-dual automorphic representations of $\GL(\A_E)$, involving the so-called $\e$-periods of $\Pi$ instead its $\s$-periods (which are not defined for conjugate self-dual representation). Since the proofs of these two theorems are essentially identical, we refer to \cite[Theorem 3.2]{TR_BC} for precise definition of $\Lambda^{imp}(\Pi,\Ad,s)$ and for a proof of \ref{adjoint_L_value}. Otherwise, we refer to Section 4 of the author's thesis \cite{thesis} where both results are proven in details. \\

\subsubsection{Relation with top and bottom degree periods} As explained in the introduction, a cohomological cuspidal automorphic representation $\Pi$ of $\GLn(\A_E)$ (for $E$ a general number field) is associated with some top and bottom degree periods. In particular, when $n=3$ an $E$ is a real quadratic field, one can define two $p$-integrally normalized periods attached to $\Pi$, the top degree period $\Om_6(\Pi) \in \C^\x / \O^\x$ and the bottom degree period $\Om_4(\Pi) \in \C^\x / \O^\x$, defined respectively within the top and bottom degrees of the cuspidal cohomology of $\GL(E)$. We refer to \cite[\S3.2.5]{BR17} for a precise definition. Then, combining \ref{adjoint_L_value} with Balasubramanyam-Raghuram's formula \cite[Theorem A]{BR17}, we can deduce an integral relation between the middle-degree periods and the extremal-degree periods of $\Pi$, under Calegari-Geraghty setting. \\

More precisely, let $\tilde{\TT}_E$ denote the \textit{full} Hecke algebra of $\GL(E)$ acting on the cohomology, localized at the maximal ideal $\m_\Pi$. By \textit{full}, we mean here that $\tilde{\TT}_E$ is defined by acting on the full cohomology $\tilde{H}_E$ of $Y_E(K_f)$ with coefficients in $\O$, not only on its $\O$-torsion-free part, as in \S\ref{hecke_corr}. In particular $\tilde{\TT}_E$ may contain $\O$-torsion, and $\TT_E = \tilde{\TT}_E /(\O-tors)$ is simply the torsion-free quotient of $\tilde{\TT}_E$. We now assume that the following three conjectures hold :
\begin{itemize}
\item $\mathrm{(Gal_{\m_\Pi})}$ The Galois representation $\rho_{\m_\Pi}$ with coefficients in $\tilde{\TT}_E$ attached to $\Pi$ exists
\item $\mathrm{(LGC_{\m_\Pi})}$ The Galois representation $\rho_{\m_\Pi}$ satisfies local-global compatibilities at minimal, Fontaine-Laffaille, and Taylor-Wiles places 
\item $\mathrm{(Van_{\m_\Pi})}$ The residual cohomology groups $H^q(Y_E(K_f),\L_{\mu}(\FF))_{\m_\Pi}$ vanishes unless $q \in [4,6]$
\end{itemize}
Moreover, we will say that $\rho_{\m_\Pi}$ satisfy $\mathrm{(CG)}$ if it satsifies the following three conditions:
\begin{itemize}
\item $\rho_{\m_\Pi}$ is $\n$-minimal;
\item $p-3 > 2m$, where $\mu = (m,m,0) \in X^+(T_3)$ (this ensures that $\rho_{\m_\Pi}$ is Fontaine-Laffaille above $p$);
\item the residual representation $\overline{\rho}_{\m_\Pi}$ has enormous image. \\
\end{itemize}

Under these conditions, and conditionally on the three conjectures $\mathrm{(Gal_{\m_\Pi})}$, $\mathrm{(LGC_{\m_\Pi})}$ and $\mathrm{(Van_{\m_\Pi})}$, the Calegari-Geraghty theory \cite{CG18} implies that the full cohomology is free over the full Hecke algebra $\tilde{\TT}_E$. We refer to the author's thesis \cite[\S3.2.4-5]{thesis} for a more detailed presentation of Calegari-Geraghty theory.  Be careful that $H^\bullet$ and $\TT$ denote the full cohomology and Hecke algebra there, while their torsion-free parts are denoted by $\bar{H}^\bullet$ and $\bar{\TT}$. Then, under Calegari-Geraghty setting, we obtain the following corollary:

\begin{corollaire} 
\label{relation_middle_top_bottom}
Assume $p >2$. Let $\Pi$ be a cohomological cuspidal automorphic representation of $\GL(\A_E)$ which is conjugate self-dual, and only ramified above split primes. Assume that conjectures $\mathrm{(Gal_{\m_\Pi})}$, $\mathrm{(LGC_{\m_\Pi})}$ and $\mathrm{(Van_{\m_\Pi})}$ hold, and that $\rho_{\m_\Pi}$ satisfies $\mathrm{(CG)}$. Then:
$$
\Om_5(\Pi,\s,\pm) \cdot \Om_5(\Pi^\vee,\s,\mp) \sim \Om_4(\Pi) \cdot \Om_6(\Pi^\vee) \cdot  \nu_{\Pi}^\pm
$$
where $\nu_{\Pi}^\pm := \eta_{\Pi} \cdot \eta_{\Pi}(M)[\pm]^{-1} \in \O$.
\end{corollaire}

It is worth noting that there is no need to exclude specific primes (except in $\mathrm{(CG)}$ for primes which are small with respect to $\mu$), as the constants in the two adjoint $L$-value formulas cancel each other out. \\

\section{An adjoint $L$-value formula for quasi-split unitary groups}
\label{part_adjoint}

In this section, we prove \ref{thmD} (see \ref{unitary_adjoint_L_value} below).
This section is organized as follows. In \ref{unitary_facts}, we fix the setting. In  \ref{unitary_1}, we present the new vector theory for quasi-split unitary groups. In \ref{unitary_2}, we state the Lapid-Mao-Morimoto formula for quasi-split unitary groups and deduce from it a precise relation between the Petersson product and some special value of the adjoint $L$-function. In \ref{unitary_3}, we recall some well-known facts about cohomological representations and define automorphic periods. Finally, in \ref{unitary_4}, we provide a cohomological interpretation of the Lapid-Mao-Morimoto formula and prove the adjoint $L$-value formula of \ref{unitary_adjoint_L_value}.

\subsection{Basic definitions and notation}
\label{unitary_facts}

Let $F$ be a number field and $E$ be a quadratic extension of $F$. We fix the standard non-trivial additive character $\psi_F := \psi_\Q \circ \mathrm{Tr}_{F/\Q}$ of $F \bs \A_F$ and we denote by $\psi_E$ the additive character of $E \bs \A_E$ defined by:
$$
\psi_E(x) = \psi_F\left(\frac{x - \s(x)}{2}\right), \quad x \in \A_E
$$
We write $\psi_E = \bigotimes_w \psi_{E,w}$ for its factorization as a tensor product of local additive characters.

\subsubsection{Quasi-split unitary groups} Let $U = U_{E/F}$ denote the quasi-split unitary group associated with $E/F$, i.e. the $F$-algebraic group whose $A$-points are given, for any $F$-algebra $A$, by:
$$
U(A) = \{ g \in \mathrm{GL}_n(E\otimes_F A), \,\, {}^t\s(g) J_n g= J_n \}
$$
where $\s$ is the non-trivial involution in $\Gal(E/F)$ and $J_n = \mathrm{antidiag}(1, \dots,1)$. Let $v$ be a place of $F$ which is split in $E$, and fix an identification $E \otimes_F F_v = F_v \x F_v$. In that case we have that:
$$
U(F_v) = \{ (g_1,g_2) \in \mathrm{GL}_n(F_v) \x \mathrm{GL}_n(F_v), \,\, g_2 = J_n {}^tg_1^{-1} J_n^{-1} \}
$$
and we identify $U(F_v)$ with $\GLn(F_v)$ via the first projection. \\

 At each finite place $v$ of $F$ we consider the maximal open compact subgroup $K_v$ of $U_{E/F}(F_v)$ defined by:
 $$
 K_v = U(F_v) \cap \GL(\O_{E,v})
 $$ 
 where $\O_{E,v}$ is the valuation ring of $E_v := E \otimes_F F_v$. When $v$ is unramified in $E$, $K_v$ is hyperspecial. Moreover, when $v$ is split, one has one course that $K_v = \GLn(\O_{F,v})$ \\

Let $N_U$ be the subgroup of upper triangular matrices in $U$. We define a non-degenerate additive character $\psi$ of $N_U(\A_F)$ by the formula:
$$
\psi(u) = \psi_E(u_{1,2} + \dots + u_{n-2,n-1} + u_{n-1,n})
$$
In particular, for $n=3$, one has that:
$$
N_U(F) =\left\{ n(x,y) := \left(\begin{array}{ccc}1 & x & y - \frac{1}{2} x x^\s \\ 0 & 1 & -x^{\s} \\ 0 & 0 & 1\end{array}\right) \in \GL(E), \quad x,y \in E, \quad \Tr_{E/F}(y) = 0 \right\}
$$
and $\psi(n(x,y)) = \psi_E(x)^2$. \\

Let $Z_U$ be the center of $U$. The group of $F$-points of $Z_U$ is isomorphic to $E^{(1)} := \Ker(N_{E/F}) \subset E^\x$ via:
$$
\l \in E^{(1)} \mapsto \left(\begin{array}{ccc} \l &  &  \\  & \l &  \\  &  & \l \end{array}\right) \in Z_U(F)
$$
We will also denote by $\A_E^{(1)}$  the kernel of $N_{E/F} : \A_E^\x \to \A_F^\x$, which is isomorphic to $Z_U(\A_E)$.

\subsubsection{Measures.} On $U(\A_F)$, let $dg$ the global Haar measure defined by:
$$
dg = \left( \prod_{i=1}^n L(\chi_{E/F}^i,i) \right) \cdot dg^{Tam}
$$
where $dg^{Tam}$ is the Tamagawa measure on $U(\A_F)$. For each place $v$ of $F$, we take the local measure $dn_v$ on  $N_U(F_v)$ to be $dn_v = (dx_v)^{\dim N_U}$, where $dx_v$ is the additive Haar measure on $F_v$ defined in \S\ref{general_haar_measures}. In particular, $\vol(N_U(\O_v),dn_v) = 1$ when $v$ is unramified in $E$. Then $dn = \prod_v dn_v$ is the Tamagawa measure on $N_U(\A_F)$ and:
$$
\vol(N_U(F)\bs N_U(\A_F),dn) = 1
$$

\subsection{New vector theory for quasi-split unitary groups}
\label{unitary_1}

Everything stated in this subsection applies to any quasi-split unitary group $U_{E/F}$ associated with a quadratic extension $E/F$ of number fields.

\subsubsection{Local theory}

We temporarily adopt local notation. Let $K$ be a non-archimedean local field of characteristic zero, and let $L$ be a quadratic extension of $K$. Let $\O_K$ (resp. $\O_L$) be the valuation ring of $K$ (resp. $L$), $\wp_K$ (resp. $\wp_L$) its prime ideal, and $q_K = \#(\O_K/\wp_K)$ (resp. $q_L = \#(\O_L/\wp_L)$). Let $U$ be the quasi-split unitary group in $n$-variables associated with $L/K$. Let $\psi_L$ be a non-trivial additive character of $L$. In particular, such a character can be obtained from a non-trivial additive character $\psi_K$ of $K$ by the formula:
$$
\psi_L(x) = \psi_K\left(\frac{x - \s(x)}{2}\right), \quad x \in L
$$
where $\s$ is the Galois involution of $L/K$. From a character $\psi_L$ of $L$, we define a character $\psi$ of $N(K) = N_U(K)$ by the formula $\psi(u) = \psi_L(u_{1,2} + \dots + u_{n-1,n})$ for $n \in N(K)$. \\

\textit{Genericity.} Let $(\pi,V)$ be an irreducible admissible representation of $U(K)$. Let $\Hom_{N}(\pi, \psi)$ be the space of functionals $W : V \to \C$ such that:
$$
W( \pi(u) \phi) = \psi(u) W(\phi), \quad \phi \in V, \quad u \in N(K)
$$
We say that $\pi$ is generic if $\Hom_{N}(\pi, \psi) \neq 0$. In that case, $\dim_\C \Hom_{N}(\pi, \psi) = 1$. Moreover, by Frobenius reciprocity, one has $\Hom_{N}(\pi, \psi) = \Hom_{U}(\pi, \mathrm{Ind}_{N}^{U}\psi)$ and thus, up to scalar, there exists a unique embeddind of $\pi$ into the space:
$$
\mathrm{Ind}_{N}^{U}\psi := \{  f : U(K) \to \C, \; f(ug) = \psi(u) f(g), \; u \in N(K), g \in U(K) \}
$$
Concretely, if $W \in \Hom_{N}(\pi, \psi)$ is some non-zero functionnal, then the corresponding embedding is given by $\phi \mapsto W_\phi$, with:
$$
W_\phi(g) = W(\pi(g)\phi), \quad g \in U(K)
$$
The image of $\pi$ inside $\mathrm{Ind}_{N}^{U}\psi$ is denoted $\W(\pi,\psi)$ and called the Whittaker model of $\pi$ with respect to $\psi$. We will sometimes also write $\W(\pi,\psi_L)$, or $\W(\pi,\psi_K)$ when $\psi_L$ is constructed from a character $\psi_K$ of $K$ as specified above.  \\

Assume $n=2m+1$ is odd. Let $a \in K^\x$ and $\psi_L^a$ the additive character of $L$ defined by $\psi_L^a(x) = \psi_L(ax)$, for $x \in L$. Let $\psi^a$ the corresponding character on $N(K)$. Then one check that:
$$
\psi^a(u) = \psi(\diag(a^m,\dots,1,\dots, a^{-m})\cdot u \cdot \diag(a^{-m},\dots,1,\dots, a^{m}))
$$
Thus, the map $t_a : W \mapsto W^a$ defined by:
$$
W^a(g) = W(\diag(a^{-m},\dots,1,\dots, a^{m}) g)
$$
is a $U(K)$-equivariant isomorphism $\W(\pi,\psi) \toeq \W(\pi,\psi^a)$. \\

\textit{Spherical vector at ramified primes.} Assume that $n=3$. Let $\psi_K$ be an unramified additive character of $K$ and let $\psi$ be the corresponding additive character of $N(K)$. Let $C$ be the maximal open compact subgroup of $U(K)$ defined by $U(K) \cap \GL(\O_L)$.  Let $\pi$ be an irreducible generic unitary representation of $U(K)$ which is $C$-spherical, i.e. has a non-zero $C$-fixed vector. Since the local Hecke algebra $\H(U,C)$ is commutative (as $C$ is a special subgroup), the space of $C$-fixed vectors in $\pi$ is one dimensional. Let $\W(\pi,\psi)$ be the Whittaker model of $\pi$ with respect to $\psi$. The following lemma is implicitly proven in \cite[\S 4.3.2]{Che25}:

\begin{lemma}
Let $W \in \W(\pi,\psi)$ be non-zero $C$-invariant Whittaker function. Then:
$$
W(I_3) \neq 0
$$
\end{lemma}

\begin{proof}
As explained in \cite[\S 4.3.2]{Che25}, every $C$-spherical representation $\pi$ is isomorphic to an irreducible composition factor of the induced representation $\mathrm{Ind}_{B(K)}^{U(K)} \chi$, for some unramified  character $\chi$ of $L^\x$. The character $\chi$ is seen as a character of the Borel subgroup $B(K)$ of $U(K)$ consisting of upper triangular matrices via $\chi(a,b,\s(a)^{-1}) = \chi(a)$. In the space of $C$-fixed vectors in $\mathrm{Ind}_{B(K)}^{U(K)} \chi$, there is a canonical spherical vector $f^\circ \in \pi$ normalized so that $f^\circ(1) =1$. Chen then proves that (see the proof of \cite[Proposition 4.6]{Che25}):
$$
W_{f^\circ}(I_3) = 1 -\chi(\varpi_L)q^{-1}
$$
where $q = q_K = q_L$, $\varpi_L$ is an uniformizer of $\O_L$, and $W_{f^\circ}$ is some explicit function associated with $f^\circ$ in $\W(\pi,\psi)$. Consequently, $W(I_3) \neq 0$ for all non-zero $C$-fixed Whittaker function $W \in \W(\pi,\psi)$. 
\end{proof}

Fix $\phi \mapsto W_\phi$ an isomorphism $\pi \toeq \W(\pi,\psi)$. The only $C$-invariant form $\phi$ such that $W_\phi(I_3) = 1$ will be called the \textit{spherical vector} of $\pi$ with respect to $\psi$, and is denoted $\phi_\pi^\circ$. \\

\textit{Local newform theory at inert primes.} Suppose that $n = 2m+1$ is odd. We recall the newform theory for odd unramified quasi-split unitary group over local fields, du to \cite{Miyauchi2013,AOY22,Cheng23}. We assume that the extension $L/K$ is unramified. For any integer $c \geq 0$, let $K_1(\wp^c)$ be the open compact subgroup of $U(K)$ defined by:
$$
K_1(\wp^c) := \left(\begin{array}{ccc} \mathfrak{M}_m(\O_L) & \O_L & \mathfrak{M}_m(\wp_L^{-c}) \\ \wp_L^{c} & 1 + \wp_L^{c} & \O_L \\  \mathfrak{M}_m(\wp_L^{c}) & \wp_L^{c} &  \mathfrak{M}_m(\O_L) \end{array}\right) \cap U(K)
$$
where $\mathfrak{M}_m(A)$ is the set of $m \x m$ matrices with coefficients in $A \subset L$. Then, there is the following newform theory:
\begin{theorem}[\cite{Cheng23}, Theorem 1.2]
Let $(\pi,V)$ be an irreducible generic representation of $U(K)$. Then, there exists a non-negative integer $c(\pi)$ such that, for all $c \geq 0$, one has:
$$
\dim V^{K_1(\wp^c)} = \left\{
    \begin{array}{ll}
        0 & \mbox{if } c \leq c(\pi) \\
        1& \mbox{if } c = c(\pi)
    \end{array}
\right.
$$
Moreover, $c(\pi)$ coincides with the analytic conductor of $\pi$, i.e. the power of $q^{-s}$ in the $\e$-factor of $\pi$ with respect to an unramified additive character $\psi$ of $L$ which is trivial on $K$.
\end{theorem}
Any non-zero $K_1(\wp^{c(\pi)})$-fixed vector in $\pi$ is called a newform. \\

Assume that $n=3$. Let $\psi_L$ be a non-trivial additive character of $L$ which is unramified and let $\psi$ be the corresponding character of $N(K)$. Let $\W(\pi,\psi)$ be the Whittaker model of $\pi$ with respect to $\psi$, and fix $\phi \mapsto W_\phi$ an isomorphism $\pi \toeq \W(\pi,\psi)$. Then, for any non-zero newform $\phi$ in $\pi$, one has that $W_\phi(I_3) \neq 0$ \cite[Proposition 5.1]{Miyauchi2011}. The only newform $\phi$ such that  $W_\phi(I_3) = 1$ will be called the \textit{essential vector} of $\pi$ with respect to $\psi_L$, and is denoted $\phi_\pi^\circ$.  When $\pi$ is unramified (i.e. when $c(\pi) = 0$) $\phi_\pi^\circ$ is rather called the \textit{spherical vector} of $\pi$.

\subsubsection{Global theory}
\label{global_unitary_newforms}

We switch to global notation. Let $E/F$ be a quadratic extension of number fields, and $U$ be the quasi-split unitary group of degree $n$ associated with $E/F$. Let $\pi$ be a cuspidal automorphic representation of $U(\A_F)$, realized as a sub-representation of the space of automorphic forms on $U(\A_F)$. Let $\psi_E$ be the non-degenerate additive character of $E \bs \A_E$ we have fixed in \S\ref{unitary_facts} and let $\psi$ be the corresponding additive character of $N_U(F) \bs N_U(\A_F)$. If $\phi \in \pi$, the Whittaker function $W_\phi$ of $\phi$ with respect to $\psi$, is the function on $U(\A_F)$ defined by:
$$
W_\phi(g) = \int_{N_{n}(F) \bs N_{n}(\A_F) } \phi(ng) \psi^{-1}(n)dn, \quad \quad g \in U(\A_F)
$$
If the space $\W(\pi,\psi) := \{ W_\phi, \phi \in \pi \}$ is non-zero, we say that $\pi$ is \textit{globally generic} and $\W(\pi,\psi)$ is called the Whittaker model of $\pi$. Write $\pi = \otimes_v \pi_v$ and $\psi_E = \otimes_v \psi_v$ as tensor products over the places $v$ of $F$. If $\pi$ is generic, then $\pi_v$ is generic for all place $v$ of $F$. Moreover, the isomorphism $\phi \mapsto W_\phi$ fixes local isomorphisms $\pi_v \toeq \W(\pi_v,\psi_v)$ for all $v$. \\

We now present the newform theory for a globally generic cuspidal automorphic representation $\pi$ of $U(\A_F)$. When $\pi$ is only ramified at split places (which is the case we will be mainly interested in), the theory simply reduces to the newform theory for $\GLn$, presented in \S\ref{mirahoric_theory}. However, for completeness, we also allow $\pi$ to be ramified at inert places when $n$ is odd, since a local newform theory exists for odd unramified quasi-split unitary groups. \\

Let $\n = \prod_{\wp \mid \n} \wp^{c_\wp}$ be an ideal of $\O_F$ and assume that $\n$ has no ramified prime factors. We define the mirahoric subgroup $K_1(\n)$ of level $\n$ to be the following open compact subgroup of $U(\A_{F,f})$:
$$
K_1(\mathfrak{n}) := \prod_{\wp \nmid \mathfrak{n}} K_\wp \x \prod_{\wp \mid \mathfrak{n}} K_1(\wp^{c_\wp})
$$
where $K_\wp = U(F_\wp) \cap \GLn(\O_{E,\wp})$ is the particular maximal open compact subgroup of $U(F_\wp)$ we have fixed, $K_1(\wp^{c})$ is the local mirahoric subgroup of $ \GLn(\O_\wp) = U(\O_\wp)$ when $\wp$ is split (see \S \ref{mirahoric_theory}), and $K_1(\wp^{c})$ is the open compact subgroup of $U(F_\wp)$ defined in the previous paragraph when $\wp$ is inert. \\

Let $\pi$ be a globally generic cuspidal automorphic representation of $U(\A_F)$, and let $\pi_f$ be its finite part. The following global result can easily be deduced from the local theory, as detailed in \S\ref{mirahoric_theory}:

\begin{prop}
\label{unitary_mirahoric}
Assume that $\pi_f$ is spherical at ramified prime ideals. If $n$ is even assume moreover that $\pi_f$ is spherical at inert prime ideals. Then, there exists an unique ideal $\n$ of $\O_F$ such that:
$$
\dim \pi_f^{K_1(\n)} = 1
$$
The ideal $\n$ is called the mirahoric level of $\pi$, and will sometimes be denoted by $\n(\pi)$.
\end{prop}

We now explain how to choose a particular newform $\phi_f \in \pi_f$ in the situations we consider in the sequel. Assume that the different $\mathfrak{d}_F$ of $F$ is split in $E$. Assume that $\pi_f$ is spherical at ramified prime ideals. Let $\psi_F$ be the standard non-trivial additive character of $\A_F/F$ and let decompose its finite part $\psi_f = \otimes_v \psi_v$ as a tensor product over the finite places $v$ of $F$. At each place $v$ dividing $\mathfrak{d}_F$, $\psi_v$ is ramified, of conductor $\mathfrak{d}_v^{-1}$. For such a place, we choose $d_v \in F_v^\x$ such that $\psi_v^\circ = d_v^{-1} \cdot \psi_v$ is an unramified additive character of $F_v$. We then consider the form $\phi^\circ_\pi := \otimes_{v} \phi_v$ in $\pi_f$ with:

\begin{itemize}
\item if $v$ does not divide $\mathfrak{d}_F$, then $\phi_v$ is the essential (or spherical, depending whether $\pi_v$ is ramified or not) vector $\phi_{\pi_v}^\circ$ of $\pi_v$ with respect to $\psi_v$,
\item if $v$ divides $\mathfrak{d}_F$ (and thus $v$ is split by assumption), then $\phi_v$ is the form whose Whittaker function $W_{\phi,v} \in \W(\pi_v,\psi_v)$ with respect to $\psi_v$ is given by:
$$
W_{\phi,v}(g) = W_{\pi_v}^\circ(\mathrm{diag}(d_v^{n-1},\dots,d_v,1)g), \quad \forall g \in U(F_v) = \GLn(F_v)
$$
where $W_{\pi_v}^\circ$ is the essential (or spherical) vector in $\W(\pi_v,\psi_w^\circ)$.
\end{itemize}
Then $\phi^{\circ}_\pi$ is a $K_1(\mathfrak{n}(\pi))$-fixed vector in $\pi_f$ which is called the \textit{essential vector} of $\pi$. \\

\subsection{The Petersson product as an adjoint $L$-value}
\label{unitary_2}

\subsubsection{Lapid-Mao formula}

Let $\pi = \otimes_v \pi_v$ be a globally generic cuspidal automorphic representation of $U(\A_F)$, and let $\pi^\vee$ denote its dual representation. The Petersson product of $\ph \in \pi$ and $\ph' \in \pi^\vee$ is defined to be:
$$
\langle \ph,\ph' \rangle := \int_{U(F)\bs U(\A_F)}\ph(g)\ph'(g)dg
$$
For each place $v$ of $F$, we choose a $U(F_v)$-equivariant pairing $\langle \cdot, \cdot \rangle_v : \pi_v \x \pi_v^\vee \to \C$ such that for any non-zero decomposable vectors $\ph = \otimes_v \ph_v \in \pi$ and $\ph' = \otimes_v \ph_v' \in  \pi^\vee$, one has:
$$
\langle \ph,\ph' \rangle = \prod_{v} \langle \ph_v,\ph_v' \rangle_v
$$
For each $v$, we define the local Whittaker functional $I_v(\cdot,\cdot)$ by the stable integral (see \cite[\S 2]{LM15}):
$$
I_v(\ph_v,\ph_v') = \int_{N_v}^{st} \langle \pi_v(n_v)\ph_v,\ph_v' \rangle_v \psi_{v}(n_v)^{-1} dn_v
$$
where $N_v = N_U(F_v)$ and $dn_v$ is the Haar measure on $N_v$ specified in \S \ref{unitary_facts}. Note that $I_v$ is $(N_v,\psi_{v})$-equivariant in the first variable and $(N_v,\psi_{v}^{-1})$-equivariant in the second variable. We then define the normalized local Whittaker functional by:
$$
I_v^\natural(\ph_v,\ph_v')= \frac{I_v(\ph_v,\ph_v')}{\langle\ph_v,\ph_v'\rangle_v}
$$
Note that $I_v^\natural$ does not depend on the choice of $\langle \cdot, \cdot \rangle_v$ up to scalar.
Assume that $v$ does not divide the different of $F$ and is unramified in $E$, and assume that $\pi$ is unramified at $v$. Then, if $\ph_v$ and $\ph_v'$ are spherical vectors, we have (see \cite[Proposition 2.14]{LM15}):
$$
I_v^\natural(\ph_v,\ph_v') = \frac{\prod_{i=2}^{n} L_v(\chi_{E/F}^i,i) }{L_v(\pi,\Ad,1)}
$$
where $L_v(\chi_{E/F}^i,i)$ is the local $L$-factor at $v$ of the Dirichlet $L$-function of $\chi_{E/F}^i$ and $L_v(\pi,\Ad,1)$ is the local $L$-factor at $v$ of the adjoint $L$-function of $\pi$ (see \S\ref{primitive_L_functions}). Note that for our choice of the Haar measure on $N_v$, we have $\mathrm{vol}(N_v \cap U(\O_v)) = 1$. The following theorem is the Lapid-Mao conjecture for $U$, established by Beuzart-Plessis \& Chaudouard \cite{BP-C23} and Morimoto \cite{M24}:

\begin{theorem}
\label{Lapid_Mao_conjecture}
Let $\pi$ be a globally generic cuspidal automorphic representation of $U(\A_F)$. Let $\ph = \otimes_v \ph_v \in \pi$ and $\ph' = \otimes_v \ph_v' \in \pi^\vee$ be non-zero decomposable vectors. Then, there exists a finite set $S$ of places of $F$ such that we have:
$$
\frac{\langle \ph,\ph' \rangle}{W_\ph(e) \cdot W_{\ph'}(e)} = 2^k \cdot L^S(\chi_{E/F},1) \cdot L^S(\pi,\Ad,1) \cdot \prod_{v\in S} \frac{\prod_{i=1}^{n}L_v(\chi_{E/F}^i,i)}{I^{\natural}_v(\ph_v,\ph_v')}
$$
\end{theorem}

Note that our choice of Haar measure on $U(\A)$ is different from Morimoto, who uses the Tamagawa measure, which is why we get $L_S$ at the numerator  instead of $L^S$ at the denominator of the right-hand side in the above formula.

\subsubsection{Local factors at bad places}

We denote by $S_\pi$ the set of prime ideals where $\pi$ is ramified, by $S_{E/F}$ be the set of prime ideals of $F$ which ramifies in $E$, by $S_{F}$ be the set of prime ideals dividing the (absolute) different $\mathfrak{d}_F$ of $F$, and by $S_\inf$ the set of archimedean places of $F$. The following proposition precises the above theorem for a special choice of forms in $\pi$ and $\pi^\vee$:

\begin{prop}
\label{full_lapid_mao}
Assume that the different $\mathfrak{d}_F$ splits in $E$. Assume that $\pi$ is only ramified at split primes. Let $\phi= \phi_\pi^\circ \otimes \phi_\inf \in \pi$ and $\phi'= \phi_{\pi^\vee}^\circ \otimes \phi_\inf' \in \pi^\vee$, where $\phi_\pi^\circ \in \pi_f$ and $\phi_{\pi^\vee}^\circ \in \pi_f^\vee$ are the respective essential vectors of $\pi$ and $\pi^\vee$, and $\phi_\inf \in \pi_\inf$ and $\phi_\inf' \in \pi_\inf^\vee$ are some archimedean forms. Then:

$$
\langle \phi,\phi' \rangle = C \cdot {L(\chi_{E/F},1) \cdot L^{imp}(\pi,\Ad,1)} \cdot \prod_{v \in S_\inf} [ W_{\phi_v}, W_{\phi_v'}]_v
$$

where $C = 2^k \cdot D_F^{n(n-1)/4} \cdot u_{E/F} \cdot \prod_{\wp \in S_\pi} \left( 1 - N\wp^{-n} \right)^{-1}$ and:
$$
u_{E/F} = \prod_{v \in S_{E/F}} \frac{\prod_{i=2}^{n} L_v(\chi_{E/F}^i,i)}{L_v(\pi_v,\Ad,1)} \cdot I^{\natural}_v(\phi_v,\phi_v')^{-1}
$$
\end{prop}

The imprimitive adjoint $L$-function $L^{imp}(\pi,\Ad,1)$ of $\pi$ is defined in \S\ref{imp_unitary_adjoint_L_funs}.

\begin{proof}
For our choice of $\phi$ and $\phi'$, the set $S$ in \ref{Lapid_Mao_conjecture} can be chosen to be:
$$
S = S_\pi \cup S_{E/F} \cup S_{F} \cup S_\inf
$$
By assumptions, all the places in $S_\pi$, $S_F$ and $S_\inf$ are split in $E$. Let $v$ a place of $F$ which splits in $E$. In that case, $U(F_v) \simeq \GLn(F_v)$. We recall that $(\GLn(F_v) \x\GLn(F_v), \GLn(F_v))$ is a Gelfand pair i.e. $\Hom_{\GLn(F_v)}(\pi_v \otimes \pi_v^\vee, \C)$ is a $1$-dimensional vector space (see \cite[Theorem A]{Bernstein84} for non-archimedean places \cite[Theorem 8.2.5]{AG09} for archimedean places). Thus, in the above theorem, at each place $v \in S$ which is split in $E$, one can choose $\langle \cdot,\cdot \rangle_v$ to be $(\phi,\phi') \mapsto [W_\phi,W_{\phi'}]_v$ where $[\cdot,\cdot]_v$ is the pairing defined by:
\begin{equation}
\label{whittaker_pairing}
[ W,W' ]_v = \int_{N_{n}(F_v)\bs P_{n}(F_v)}W(p) {W'}(p)dp
\end{equation}
for $W \in \W(\pi_v,\psi_{v})$ and $W \in \W(\pi_v^\vee,\psi_{v}^{-1})$ (for the equivariance of this pairing, see \cite[Proof of Proposition 3.1]{Zh14}). Here $P_n(F_v)$ is the mirabolic subgroup, consisting of matrices in $\GLn(F_v)$ whose last row is equal to $(0,\dots,0,1)$. For this choice of $\langle \cdot,\cdot \rangle_v$, one has (see \cite[Lemma 4.4]{LM15}),:
$$
I_v(\phi_v,\phi_v') = W_{\phi_v}(e) \cdot W_{\phi_v'}(e)
$$
Assume moreover that $v$ is finite. In \cite[Lemma 4.4]{LM15}, the Haar measure on $P_n(F_v)$ in (\ref{whittaker_pairing}) is normalized so that $\vol(P_n(\O_v))= \prod_{i=1}^{n-1} \zeta_{F_v}(i)^{-1}$. Thus \cite[Lemma 3.2]{TR_BC}, where the Haar measure is normalized so that $\vol(P_n(\O_v))= 1$, gives:
$$
\langle \phi_v,\phi_v' \rangle_v = \frac{L^{imp}(\pi_v \x \pi_v^\vee,1)}{\prod_{i=1}^{n-1}\zeta_{F_v}(i)} = \frac{L_{v}(\chi_{E/F},1) \cdot L^{imp}(\pi_v, \Ad,1)}{ \prod_{i=1}^{n-1}L_{v}(\chi_{E/F}^i,i)}
$$
and the formula follows from \ref{Lapid_Mao_conjecture}. \\

Now assume that $v \in S_F$ (and thus $v$ is split by assumption). One still have that:
$$
I^{\natural}_v(\phi_v,\phi_v') = \frac{W_{\phi,v}(e) \cdot W_{\phi',v}(e)}{[W_{\phi,v},W_{\phi',v}]_v}
$$
with $W_{\phi,v} \in \W(\pi_v,\psi_v)$ and $W_{\phi',v} \in \W(\pi_v^\vee,\psi_v^{-1})$. By definition of $\phi$, we have:
$$
W_{\phi,v}(g) = W_{\pi_v}^\circ(\mathrm{diag}(d_v^{n-1},\dots,d_v,1)g)
$$
where $W_{\pi_v}^\circ$ is the essential vector in $\W(\pi_v,\psi_v^\circ)$ and $\psi^\circ_v$ is an unramified additive character of $F_v$ such that $\psi_v = d_v \cdot \psi^\circ_v$, with $d_v \in F_v^\x$ a representative of $\mathfrak{d}_{F,v}^{-1}$. Thus:
$$
[W_{\phi,v},W_{\phi',v}]_v = [W_{\pi_v}^\circ,W_{\pi_v^\vee}^\circ]_v = N(\mathfrak{d}_v)^{n(n-1)/4} \cdot  \frac{L_{v}(\chi_{E/F},1) \cdot L^{imp}(\pi_v, \Ad,1)}{\prod_{i=1}^{n-1}L_{v}(\chi_{E/F}^i,i)}
$$
where the factor $N(\mathfrak{d}_v)^{n(n-1)/4}$ comes from the different choice of Haar measure on $N_U(\O_v)$ in \cite[Lemma 4.4]{LM15}, where it is normalized so that $\vol(N_U(\O_v)) = N(\mathfrak{d}_v)^{-n(n-1)/4}$, and in \cite[Lemma 3.2]{TR_BC}, where it is normalized so that $\vol(N_U(\O_v)) = 1$.

\end{proof}

\begin{remark}
Let us no longer assume that $\mathfrak{d}_F$ is split in $E$, and explain how to generalize \ref{full_lapid_mao} by choosing appropriate local forms. Assume that $n=3$. Let $v \mid \mathfrak{d}_F$ which is non split in $E$. For such a place, we choose $d_v \in F_v^\x$ such that $\psi_v^\circ = d_v^{-1} \cdot \psi_v$ is an unramified additive character of $F_v$. Then, for $\phi_v \in \pi_v$ and $\phi_v' \in \pi_v^\vee$ one has:
$$
I_v(\phi_v,\phi_v')^{\psi_v} = | \mathfrak{d}_v |_{v}^{r_v} \cdot   I_v(\pi(t_v)\phi_v,\pi(t_v)\phi_v')^{\psi_v^\circ}
$$
where $t_v = \diag(d_v^{-1},1,d_v)$ and $r_v$ is some explicit integer (depending whether $v$ is inert or ramified in $E$). The exponents indicate with respect to which additive character the integral $I_v$ is computed.
Let $\phi_v = \pi(t_v^{-1}) \big( \phi_{\pi_v}^\circ\big)$ where $\phi_{\pi_v}^\circ$ is the spherical vector of $\pi_v$ (with respect to $\psi_v^\circ$). Then $\phi_v$ is an element of the 1-dimensional space of $t_v^{-1}K_v t_v$-fixed vectors in $\pi_v$, where $K_v = U(F_v) \cap \GLn(\O_{E,v})$. We define $\phi_v' \in \pi_v^\vee$ in a similar way. Then, \cite[Proposition 2.14]{LM15} implies:
$$
I_v^\natural(\phi_v,\phi_v')^{\psi_v} = | \mathfrak{d}_v |_{v}^{r_v'} \cdot  \frac{\prod_{i=2}^{n} L_v(i,\chi_{E/F}^i) }{L_v(\pi,\Ad,1)}
$$
for  some explicit integer $r_v'$. By choosing the local factors of $\phi$ and $\phi'$ at $v \mid \mathfrak{d}_F$ ($v$ being inert or ramified in $E$) to be $\phi_v$ and $\phi_v'$ in \ref{full_lapid_mao}, we thus get a similar formula for $\langle \phi,\phi' \rangle$.
\end{remark}

\subsubsection{Triviality and algebraicity of $u_{E/F}$} One may expect the factor $u_{E/F}$ in \ref{full_lapid_mao} to be $1$. This result would be a generalization of \cite[Proposition 2.14]{LM15}  to ramified quasi-split groups, and involves the computation of the normalized local Whittaker functional for spherical representations of such groups. It has been shown by Chen for ramified unitary groups on $\Q_p$:

\begin{prop}[\cite{Che25}, Proposition 4.6]
\label{triviality_uE}
Assume that $n=3$ and $F=\Q$. Then $u_{E/\Q} =1$.
\end{prop}

Chen’s computations should apply to arbitrary ramified quadratic extensions of $p$-adic fields, at least if $p \neq 2$, but we have not attempted to generalize his result. However, one can prove that $u_{E/F}$ belongs to the rationality field of $\pi_f$ by reasoning as in the proof of \cite[Theorem 3.4]{Che25}. To this end, let us observe that $u_{E/F}$ only depends on $\pi_f$ since its expression involves the spherical vectors of its local factors $\pi_v$ at ramified places $v$. We thus write $u_{E/F}(\pi_f) = u_{E/F}$ in order to emphasize this dependance. Let $\mathrm{Aut}(\C)$ be the set of automorphisms of $\C$. The following lemma states that $u_{E/F}$ is equivariant with respect to the action of $\mathrm{Aut}(\C)$:
\begin{lemma}
\label{rationality_uE}
For every $\vs \in \mathrm{Aut}(\C)$, one has that:
$$
\vs(u_{E/F}(\pi_f)) = u_{E/F}({}^\vs \pi_f)
$$
In particular, $u_{E/F}(\pi_f) \in \Q(\pi_f) := \C^{S(\pi_f)}$ with $S(\pi_f) = \{ \vs \in \mathrm{Aut}(\C), {}^\vs \pi_f = \pi_f \}$.
\end{lemma}

\begin{proof} The proof of this result is implicitly contained in \cite{Che25}. We recall it here. The constant $u_{E/F}$ is a product of local factors. To prove the lemma, it suffices to prove the $\mathrm{Aut}(\C)$-equivariant formula for each local factor. Thus, let $v \in S_{E/F}$ and let $w \mid v$. We write $L/K$ for $E_w/F_v$. To simplify the notation, we will now adopt local notation, and omit the subscript $v$. Accordingly, we will write $\pi = \pi_v$, $\psi = \psi_v$ for the unramified additive character of $K = F_v$, 
$\langle \cdot ,\cdot \rangle$ for $\langle \cdot ,\cdot \rangle_v :\pi_v \x \pi_v^\vee \to \C $ and $I(\cdot,\cdot)$ for $I_v(\cdot,\cdot)$.

Let $\vs \in \mathrm{Aut}(\C)$. We recall that the local action of $\vs$ on $\pi = \pi_v$ is defined as follows. Let $V$ be the representation space of $\pi$. Then ${}^\vs \pi$ is the representation acting on ${}^\vs V:= V \otimes_{\C,\vs^{-1}} \C$ via $\pi$ on the first factor. Let $a_{\vs} \in \O_K^\x$ be such that $\vs(\psi(x)) = \psi(a_{\vs}x)$ for all $x \in K$. Then, one has that:
$$
\vs(\psi_N(n)) = \psi_N\left(\left(\begin{smallmatrix} a_{\vs} & &\\  & 1& \\ &&a_{\vs}^{-1} \end{smallmatrix}\right) n \left(\begin{smallmatrix} a_{\vs}^{-1} & &\\  & 1& \\ &&a_{\vs} \end{smallmatrix}\right)\right)
$$
for $n \in N$. The representation ${}^\vs \pi$ is generic and the map $t_{\vs} : W \mapsto {}^\vs W$ defined by:
$$
{}^\vs W : g \mapsto \s(W(\left(\begin{smallmatrix} a_{\vs}^{-1} & &\\  & 1& \\ &&a_{\vs} \end{smallmatrix}\right)g))
$$
is a $U(K)$-equivariant isomorphism $\W(\pi,\psi) \toeq \W({}^\vs \pi,\psi)$ which is $\s$-semi-linear:
$$
t_{\vs}( \l \cdot W) = \vs(\l) \cdot t_{\vs}(W), \quad \l \in \C.
$$
We also define the $U(K)$-equivariant isomorphism $t_\vs : \pi \toeq {}^\vs \pi$, which sends $\phi \in \pi$ to the unique form ${}^\vs \phi \in {}^\vs \pi$ such that $W_{{}^\vs \phi} = {}^\vs W_\phi$. Since $\diag(a_{\vs},1, a_{\vs}^{-1}) \in U(K) \cap \GL(\O_L)$, we  see that the spherical form $\phi_{{}^\vs \pi}^\circ$ of ${}^\vs\pi$ (with respect to $\psi$) is:
$$
\phi_{{}^\vs\pi}^\circ = {}^\vs\phi_{\pi}^\circ
$$
where $\phi_{\pi}^\circ$ is the spherical vector of $\pi$ (with respect to $\psi$). Let $\langle \cdot,\cdot \rangle_{\vs} : {}^\vs \pi \x {}^\vs \pi^\vee \to \C$  be the $U(K)$-equivariant pairing defined by:
$$
\langle {}^\vs \phi, {}^\vs \phi' \rangle_{\vs} = \vs(\langle \phi, \phi' \rangle)
$$
Thus, we have that $\vs(\langle \phi_{\pi}^\circ, \phi_{\pi^\vee}^\circ \rangle) = \langle \phi_{{}^\vs\pi}^\circ, \phi_{{}^\vs \pi^\vee}^\circ \rangle_{\vs}$. We now denote by ${}^\vs I(\cdot,\cdot)$ the local Whittaker functionnal constructed using $\langle \cdot,\cdot \rangle_{\vs}$. Then, one has \cite[Formula (3.11)]{Che25} that:
$$
\vs(I(\phi, \phi')) = {}^\vs I( {}^\vs \pi(\diag(a_{\vs},1, a_{\vs}^{-1})({}^\vs \phi), {}^\vs \pi^\vee(\diag(a_{\vs},1, a_{\vs}^{-1}))({}^\vs \phi'))
$$
for any $\phi \in \pi$ and $\phi' \in \pi^\vee$. Consequently, we see that $\vs(I(\phi_{\pi}^\circ, \phi_{\pi^\vee}^\circ)) = {}^\vs I(\phi_{{}^\vs\pi}^\circ, \phi_{{}^\vs \pi^\vee}^\circ)$. Finally, since $\vs(L(\pi,\Ad,1)) = L({}^\vs \pi,\Ad,1)$, we get:
$$
\vs \left( I^\natural(\phi_{\pi}^\circ, \phi_{\pi^\vee}^\circ) \cdot L(\pi,\Ad,1) \right) = I^\natural(\phi_{{}^\vs\pi}^\circ, \phi_{{}^\vs \pi^\vee}^\circ) \cdot L({}^\vs \pi,\Ad,1)
$$
and the lemma follows.
\end{proof}

\begin{remark} When $\pi$ is a cohomological cuspidal automorphic representation, the rationality field $\Q(\pi_f)$ of $\pi_f$ is a number field, and the above lemma thus proves that $u_{E/F}(\pi_f)$ is algebraic.
\end{remark}

\subsection{Cohomological representations and automorphic periods}
\label{unitary_3}

From now we assume that $n=3$ and that both $E$ and $F$ are totally real number fields. Let $U = U_{E/F}$ denote the quasi-split group associated with $E/F$. Since $E/F$ is split at archimedean places, one has that $U(F_\inf) \simeq \GL(F_\inf)$.

\subsubsection{Adelic variety and cuspidal cohomology}
\label{unitary_cusp_coho}

Let $K_f$ be an open compact subgroup of $U(\A_{F,f})$ and let $K_\inf = \prod_{v \in S_\inf(F)} K_3$. The adelic variety of level $K_f$ associated with $U$ is defined by:
$$
Y_U(K_f) = U(F) \bs U(\A_F)/K_{f}K_\inf
$$
Let $A$ denote the rings $\O$, $\K$ or $\C$. Let $\mu \in X^+(T_F)$. The Betti cohomology groups $H^\bullet(Y_U(K_f),\L_{\mu}(A))$, the compactly supported Betti cohomology groups $H_c^\bullet(Y_U(K_f),\L_{\mu}(A))$ and the inner cohomology groups $ H_{!}^\bullet(Y_U(K_f),\L_{\mu}(A))$ of weight $\mu$ with coefficients in $A$ are defined as in \S\ref{cohomology_groups}. We recall that when $A=\O$, these groups are defined to be the $\O$-torsion free part of the corresponding cohomology groups with coefficients in $\O$. The cuspidal cohomology groups are also defined as the relative Lie algebra cohomology groups (see \cite[Chapter I]{BW00}):
$$
H_{cusp}^\bullet( Y_U(K_f), \L_{\mu}(\C)) := H^\bullet(\g,K_\inf ;\mathcal{A}_{cusp}(U(F) \bs U(\A_F)/K_f)\otimes L_{\mu}(\C)),
$$
where $\mathcal{A}_{cusp}(U(\Q)\bs U(\A)/K_f)$ is the space of $K_f$-fixed cusp forms on $G_E(\A)$. We recall from \S\ref{cohomology_groups} that there exists an injection (from a result of Borel which is true for any connected reductive group):
$$
H_{cusp}^q( Y_U(K_f), L_\mu(\C)) \inj H_{!}^q( Y_U(K_f), L_\mu(\C))
$$
It follows from \cite[Théorème 3.19]{Clozel90} that:
$$
H^\bullet_{cusp}(Y(K_f), \L_{\mu}(\K)) \otimes_\K \C = H^\bullet_{cusp}(Y(K_f), \L_{\mu}(\C))
$$
We then define the cuspidal cohomology groups with coefficients in $\O$ to be:
$$
H^\bullet_{cusp}(Y(K_f), \L_{\mu}(\O)) = H^\bullet_{cusp}(Y(K_f), \L_{\mu}(\C)) \cap H^\bullet_c(Y(K_f), \L_{\mu}(\O))
$$
In particular $H^\bullet_{cusp}(Y(K_f), \L_{\mu}(\O))$ is torsion free.

\subsubsection{Hecke algebra} 
\label{full_unitary_hecke}
Let $K_f = \prod_v K_{f,v}$ be some open compact subgroup of $U(\A_{F,f})$ such that $K_{f,v}$ is a subgroup of $K_v = U(F_v) \cap \GLn(\O_{E,v})$ for all $v$. Let $S_{K_f}$ be the set of finite places of $F$ such that $K_{f,v} \neq K_v$ and let $S = S_{K_f} \cup S_p(F) \cup S_{E/F}$. The spherical Hecke algebra of level $K_f$ outside of $p$ for $U$ is defined as the tensor product:
$$
\H^S(\O) = \bigotimes_{v \notin S} \H(K_{v};\O)
$$
where $\H(K_{v};\O)$ is the local spherical Hecke algebra of $U(F_v)$ with respect to the hyperspecial subgroup $K_v$. Then, the $\O$-algebra $\H^S(\O)$ acts on the inner cohomology groups $H^\bullet_!(Y_U(K_f), \L_{\mu}(\O))$ by Hecke-correspondences, exactly as described in \S\ref{hecke_corr} for $\GLn$. Let:
$$
h^S := \mathrm{Im}[\H^S(\O) \to \End_\O H^\bullet_!(Y_U(K_f), \L_{\mu}(\O))]
$$
be the cohomological Hecke algebra.

\subsubsection{Cohomological automorphic representations}

Let $K_3 := \R_{>0} \cdot \SO \subset \GL(\R)$, and $K_\inf= \prod_{v\mid \inf} K_3 \subset U(F_\inf)$. Let $\g_{3}$ be the Lie algebra of $\GL(\R)$, and let $\g = \oplus_{v\mid \inf}  \g_{3}$ be the Lie algebra of $U(F_\inf)$. Let $\mu \in X^+(T_F)$ be an algebraic weigth for $U_E$. We say that a cuspidal automorphic representation $\pi$ of $U$ is cohomological of weight $\mu$ if the following relative Lie algebra cohomology group (see \cite[Chapter I]{BW00}) $H^q(\g,K_\inf ; \pi_\inf \otimes L_\mu(\C))$ is non trivial for some $q \geq 0$. In that case, it is non-zero if and only if $q \in [b,t]$, where $b=2t$ and $t= 3d$ and we have (see \ref{eichler-shimura_maps} for details):
\begin{equation}
\label{g-K_dim}
\dim_\C H^q(\g,K_\inf ; \pi_\inf \otimes L_\mu(\C)) = \binom{t-b}{t-q}
\end{equation}
We denote by $\mathrm{Coh}(U,\mu)$ the set of such representations. For each open compact subgroup $K_f$ of $U(\A_{F,f})$, we also define $\mathrm{Coh}(U,\mu,K_f)$ to be the subset of $\mathrm{Coh}(U,\mu)$ consisting of cohomological representations $\pi$ such that $\pi_f$ has non-zero $K_f$-fixed vectors. Since $U(F_\inf) \simeq \GL(F_\inf)$ all the results about the archimedean part of a cohomological cuspidal automorphic representation of $\GL(\A_F)$ described in \S\ref{GL3_notations} apply to the archimedean part $\pi_\inf$ of a representation $\pi \in \mathrm{Coh}(U,\mu)$. In particular, for any archimedean place $\tau \in S_\inf(F)$, the minimal $\mathrm{SO}(3)$-type of $\pi_\tau$ is $\ell_\tau = 2m_\tau + 3$ and (see  \S\ref{GL3_notations} for notation):
$$
\pi_\tau = \mathrm{Ind}_{P_{2,1}(\R)}^{\GL(\R)} (D_{\ell_\tau} \otimes \e_\tau) \otimes | \cdot |^{v_\tau}
$$
Moerover, suppose that $\pi$ is non-endoscopic (see \ref{non-endoscopic}), so that the strong stable base change of $\pi$ to $\GL(E)$ is cuspidal (see \S \ref{sss_sbc}). Then Clozel's purity lemma \cite[Lemme de pureté 4.9]{Clozel90} (see also \S\ref{GL3_notations}) implies that for all archimedean place $\tau \in  S_\inf(F)$, there exists an integer $m_\tau \geq 0$, such that $\mu_\tau = (m_\tau,0,-m_\tau)$. \\

We have the following Hecke-equivariant decomposition of the cuspidal cohomology:

\begin{equation}
\label{unitary_cusp_coho_decomposition}
H^\bullet_{cusp}(Y_U(K_f), \L_{\mu}(\C)) =  \bigoplus_{\pi \in \mathrm{Coh}(U_{3},\mu,K_f)}  H^\bullet(\g, K_\inf ; \pi_{\inf} \otimes L_{\mu}(\C)) \otimes \pi_f^{K_f}
\end{equation}
Note that each cohomological representation appears with multiplicity $1$ by the multiplicity one theorem for $U$ (see \cite[Theorem 13.3.1]{ARU3}). In particular, from the above discussion, this decomposition implies that the cuspidal cohomology is concentrated in degrees $q \in [b,t]$. For this reason, $b$ and $t$ are respectively called the bottom and top degree of the cuspidal range. The following key lemma will allow us to define the automorphic periods of a non-endoscopic cuspidal representation of $U$. We refer to \ref{non-endoscopic} in \S\ref{endoscopy} of Section \ref{part_SBC} for the definition of non-endoscopic representations.

\begin{lemma}
\label{non-endo_lambda_rank}
Let $\pi \in \mathrm{Coh}(U,\mu,K_f)$ which is non-endoscopic. Let $\l_\pi : h^S \to \O$ be the Hecke eigenstystem associated to $\pi$. Assume that the space of ${K_f}$-fixed vectors in $\pi_f$ is $1$-dimensional. Then:
$$
\dim_\C H^q_{cusp}(Y_U(K_f), \L_{\mu}(\C))[\l_\pi] =  \binom{t-b}{t-q}
$$
In particular, the top ($q=t$) and bottom ($q=b$) cuspidal cohomology groups are Hecke modules of $\l_\pi$-rank $1$.
\end{lemma}

\begin{proof}
Write $\Pi = \otimes_v \Pi_v$ for the $L$-packet of $\pi$. If $v$ is an archimedean place of $F$, then $U_v$ is split and thus $\# \Pi_v = 1$. Moreover, since $\pi$ is non-endoscopic, it follows from \cite[Theorem 13.1.1.(1)]{ARU3} that $\# \Pi_v = 1$ when $v$ is a finite place. Hence $\Pi$ is the singleton $\{ \pi \}$.

Let $\s \in \mathrm{Coh}(U,\mu,K_f)$ and let $\l_\s : h^S \to \O$ be the associated Hecke eigensystem. Let $\Si = \otimes_v \Si_v$ denote the $L$-packet of $\s$. Assume that:
$$
H^q(\g, K_\inf ; \s_{\inf} \otimes L_{\mu}(\C)) \otimes \s_f^{K_f} \subset H^q_{cusp}(Y_U(K_f), \L_{\mu}(\C))[\l_\pi] 
$$
i.e that $\l_\s(T) = \l_\pi(T)$ for all $T \in h^S$. This implies that $\Si_v = \Pi_v$ for all $v \notin S \cup S_\inf(F)$. Thus, by strong multiplicity one theorem for $L$-packets on $U$ \cite[Theorem 13.3.5]{ARU3}, $\s$ belongs to the $L$-packet of $\pi$, i.e. $\s = \pi$. We have just proved that:
$$
\dim_\C H^q_{cusp}(Y_U(K_f), \L_{\mu}(\C))[\l_\pi] = H^q(\g, K_\inf ; \pi_{\inf} \otimes L_{\mu}(\C)) \otimes \pi_f^{K_f}
$$
The lemma then follows from (\ref{g-K_dim}).
\end{proof}

\begin{remark} With the same hypothesis, one can similarly prove that:
$$
\dim_\C H^q_{!}(Y_U(K_f), \L_{\mu}(\C))[\l_\pi] =  \binom{t-b}{t-q}
$$
Indeed, only discrete automorphic representations of $U$ occurs in the inner cohomology and the Rogawski's $L$-packets strong multiplicity one theorem for $U$ is valid for discrete automorphic representations.
\end{remark}

\subsubsection{Eichler-Shimura maps}
\label{ES_maps}

Let $K_f$ be some open compact subgroup of $U(\A_{F,f})$ and let $\pi \in \mathrm{Coh}(U,\mu,K_f)$. The top $q=t$ degree and bottom $q=b$ degree groups of the $(\g,K_\inf)$-cohomology are $1$-dimensional $\C$-vector spaces. Following Chen, \cite[Formula (2.9)]{Che22}, we consider the following explicit generators (i.e. non-zero elements) of these cohomology groups:
\begin{itemize}
\item $[\pi_\inf]_b = \otimes_{v \in  S_\inf} [\pi_v]_2 \in  H^b(\g, K_\inf ; \pi_{\inf} \otimes L_\mu(\C))$
\item $[\pi_\inf]_t = \otimes_{v \in  S_\inf} [\pi_v]_3 \in H^t(\g, K_\inf ; \pi_{\inf} \otimes L_\mu(\C))$
\end{itemize}
where $[\pi_v]_2$ and $[\pi_v]_3$ are the Chen's generators of the local cohomology groups, whose definition is detailled in \S \ref{chen_generators}. As explained there, such a choice of generators can be used to define normalized Eichler-Shimura maps as follows. For $q \in \{b,t\}$, let $\d^q$ be the map defined as the composition of the following three maps:
$$
\begin{aligned}
\d^q : \pi_f^{K_f} & \toeq \pi_f^{K_f} \otimes H^q(\g,K_\inf ; \pi_\inf \otimes L_\mu(\C)) \\
&\toeq H^q(\g, K_\inf ; \pi^{K_f} \otimes L_\mu(\C)) \\
&\to H^q_{{cusp}}(Y_U(K_f), L_\mu(\C))
\end{aligned}
$$
Here, the first map is given by $\ph_f \mapsto \ph_f \otimes [\pi_\inf]_q$ and third one comes from the definition of the cuspidal cohomology. Let $S_{\mu}(K_f)$ be the space of automorphic cusp forms on $U$ of weight ${\mu}$ and level $K_f$, defined to be:
$$
S_{{\mu}}(K_f) := \bigoplus_{\pi} \pi_f^{K_f},
$$
where $\pi$ runs through $\mathrm{Coh}(U,{\mu},K_f)$. The $q$-th Eichler-Shimura map with values in the cuspidal cohomology of degree $q$ is then defined as the direct sum of the above maps:
$$
\d^q : S_\mu(K_f) \to H^q_{cusp}(Y_U(K_f), L_\mu(\C))
$$
It follows from the construction of $\d^q$ and from (\ref{unitary_cusp_coho_decomposition}) that it is an isomorphism.

\subsubsection{Automorphic periods}
\label{unitary_automorphic_periods}
Let $\pi \in \mathrm{Coh}(U,\mu)$ be a cohomological cuspidal automorphic representation of $U$, which is globally generic and non-endoscopic. We now define the automorphic periods attached to $\pi$. Roughly speaking, the automorphic periods of $\pi$ are defined by comparing, through the map $\d^q$, the $\O$-integral structure on $\pi_f$ given by its Whittaker model, to the $\O$-integral structure of the inner cohomology.

More precisely, suppose that $\pi$ is spherical at all ramified places of $F$ and let $\n$ be the mirahoric level of $\pi$ (see \ref{unitary_mirahoric}). Let $K_f = K_1(\n)$ be the mirahoric subgroup of level $\n$ and let $\phi_\pi^\circ$ denote the essential vector of $\pi$. Recall that the space of $K_f$-fixed vectors in $\pi_f$ is one-dimensional, and contains $\phi_\pi^\circ$. Let $q \in \{b,t\}$. Since $\pi$ is cohomological and non-endoscopic, it follows from \cite[Proposition 2.4]{GR13} and from the remark bellow \ref{non-endo_lambda_rank}, that $H^q_{!}(Y_U(K_f), \L_\mu(\O))[\pi_f]$ is an $\O$-module of rank 1. Let $\xi^q$ be an $\O$-basis of $H^q_{!}(Y_U(K_f), \L_\mu(\O))[\pi_f]$. The automorphic periods of $\pi$ are thus defined to be the complex scalars $\Om_q(\pi) \in \C^\x$, such that:
$$
\d^q(\phi_\pi^\circ) = \Om_q(\pi) \cdot \xi^q
$$
in $H^q_{!}(Y_U(K_f), \L_\mu(\C))[\pi_f]$. Since they depend on the choice of the $\O$-basis $\xi^q$, the periods $\Om_q(\pi)$ are defined up to some element in $\O^\x$. Note that these periods are only defined for generic cohomological cuspidal representations which are non-endoscopic. When $\pi$ is endoscopic, the $\l_\pi$-rank of the top and bottom cuspidal cohomology groups may be strictly bigger than $1$, as the $L$-packet of $\pi$ may not be a singleton.

\subsection{A cohomological interpretation of the Lapid-Mao formula}
\label{unitary_4}

\subsubsection{A Poincaré pairing}
\label{unitary_poincare_pairing}

Let $A$ denote $\O$ or $\K$. Let $\m$ be a maximal ideal of $h^S$. Let $K_f$ be some open compact subgroup of $U(\A_{F,f})$.  As in \cite[\S3.3.1]{TR_BC}, we can construct a pairing between the top and bottom degree groups of the inner cohomology, using the Poincaré duality:
$$
[ \cdot, \cdot] : H^b_{!}(Y_U(K_f),\L_\mu(A))_\m \x H^t_{!}(Y_U(K_f),\L_\mu(A)^\vee)_\m \to A
$$
This pairing is equivariant for the action of $\H^S(A)$ (with the notation of \S\ref{full_unitary_hecke}), in the sense of \cite[Lemma 3.3]{TR_BC}. If $A=\K$, the pairing is perfect. If $A =\O$, the pairing is non-degenerate but may not be perfect. However, as in \cite[\S 4.2.2]{BR17}, one can show that $[ \cdot, \cdot]$ is perfect if $p$ does not belong to the following finite set of primes:
$$
S_\partial := \left\{ p, \, H^b(\partial Y_U(k_f), \L_{\mu}(\O))_\m \mbox{ has torsion} \right\}
$$
It might be possible to show that $S_\partial$ is empty, by assuming that the Galois representation associated with $\pi$ is residually absolutely irreducible.

\subsubsection{A cohomological interpretation of the Lapid-Mao formula}
\label{Lapid_Mao_coho}
Recall that we have the following expression for the relative Lie algebra cohomology groups
(see \cite[II, Proposition 3.2]{BW00}):
$$
H^q(\g, K_\inf ; \pi_\inf \otimes L_\mu(\C)) \simeq \bigotimes_{\tau \in  S_\inf(F)} \left(\bigwedge^q \p_{3,\C}^*\otimes \pi_\tau \otimes L_{\mu_\tau}(\C) \right)^{K_3}
$$
Following \cite[\S5]{Che22}, we define a bilinear pairing on these cohomology groups. First, for each $\tau \in  S_\inf(F)$, we define the local bilinear pairing:
$$
B_\tau : \left( \bigwedge^2 \p_{3,\C}^*\otimes \W(\pi_\tau,\psi_\tau) \otimes L_{\mu_\tau}(\C) \right)^{K_3} \times \left( \bigwedge^3 \p_{3,\C}^*\otimes \W(\pi_\tau^\vee,\psi_\tau^{-1}) \otimes L_{\mu_\tau^\vee}(\C) \right)^{K_3} \to \C
$$
as the tensor product of the following three pairings:
\begin{itemize}
\item $s_\tau : \bigwedge^2 \p_{3,\C}^* \x \bigwedge^3  \p_{3,\C}^* \to \C$ defined by :
$$
A^* \wedge B^* \wedge X^* \wedge Y^* \wedge Z^* =  s(A^* \wedge B^*, X^* \wedge Y^* \wedge Z^*) \times \bigwedge_{i=1}^5 X_i^*
$$
where $(X_i)_{i=1}^5$ is an ordered basis of $\p_{3,\C}$;
\item $[ \cdot, \cdot ]_\tau: \W(\pi_\tau,\psi_\tau) \x \W(\pi_\tau^\vee,\psi_\tau^{-1}) \to \C$ defined in (\ref{whittaker_pairing});
\item $\langle \cdot, \cdot \rangle_{\mu_\tau}: L_{\mu_\tau}(\C) \x L_{\mu_\tau^\vee}(\C) \to \C$ defined by \ref{pairing_coefficients}.
\end{itemize}

Then, we define a bilinear pairing:
$$
B : H^b(\g,K_\inf, \W(\pi_\inf,\psi_\inf) \otimes L_\mu(\C))\times H^t(\g,K_\inf, \W(\pi_\inf^\vee,\psi_\inf^{-1}) \otimes L_{\mu^\vee}(\C)) \to \C
$$
as the tensor product $B(\cdot,\cdot) = \bigotimes_{\tau \in  S_\inf(F)} B_\tau(\cdot,\cdot)$. \\

Let us write the Chen's generators of \S\ref{ES_maps} as $[\pi_\inf]_b = \sum_{i\in I} \w_i \otimes \ph_i \otimes P_i$ and $[\pi_\inf^\vee]_t = \sum_{j\in J} \eta_j \otimes \ph_j' \otimes P_j'$. Then, similarly to \cite[\S3.3.1]{TR_BC}, we get that:
$$
[\d^b(\ph_\pi^\circ),\d^t(\ph_{\pi^\vee}^\circ)] = 4^d \cdot \frac{h(K_f) \cdot \vol(K_f, dg)^{-1}}{\vol(Z_U(F) \bs Z_U(\A_F))} \cdot \sum_{i\in I} \sum_{j\in J} s(\w_i,\eta_j) \cdot \langle P_i, P_j' \rangle_\mu \cdot \langle \ph_\pi^\circ \otimes \ph_i,  \ph_{\pi^\vee}^\circ \otimes \ph_j' \big \rangle
$$
where $h(K_f) := \vol(Z_U(F) \bs Z_U(\A_{F,f}) / Z_U(\A_{F,f}) \cap K_f)$. In our case $K_f = K_1(\n)$ is the mirahoric subgroup of level $\n$, the mirahoric level of $\pi$. We thus write $h_U(\n)$ for $h(K_{f})$. \\

Assume that $p \nmid  2N_{F/\Q}(\n)h_U(\n)D_F$. Note that with our choices of Haar measures, the volume of $\vol(Z_U(F) \bs Z_U(\A_F))$ is $L(\chi_{E/F},1)$. Consequently:
$$
\frac{L^{imp}(\pi,\Ad,1) \cdot L(\chi_{E/F},1)}{\vol(Z_U(F) \bs Z_U(\A_F))} \sim L^{imp}(\pi,\Ad,1)
$$
where the imprimitive adjoint $L$-function $L^{imp}(\pi,\Ad,1)$ of $\pi$ is defined in \S\ref{imp_unitary_adjoint_L_funs}. Thus, it follows from \ref{full_lapid_mao} that:
$$
\begin{aligned}
[\d^b(\ph_\pi^\circ),\d^t(\ph_{\pi^\vee}^\circ)] &\sim u_{E/F}  \cdot {L^{imp}(\pi,\Ad,1)} \cdot \sum_{i\in I} \sum_{j\in J} s(\w_i,\eta_j) \cdot \langle P_i, P_j' \rangle_\mu \cdot [W_i,W_j']_\inf
\end{aligned}
$$
where $u_{E/F}$ is the constant of \ref{full_lapid_mao}. Note that the factor $\prod_{\wp \in S_\pi} ( 1 - N\wp^{-3})^{-1}$ appearing in the constant $C$ of \ref{full_lapid_mao} is canceled out by the factor $\zeta_{F,\n}(3)$ appearing in the volume of $K_f$ with respect to the unormalized Tamagawa measure. Finally, we obtain:
$$
[\d^b(\ph_\pi^\circ),\d^t(\ph_{\pi^\vee}^\circ)] \sim u_{E/F} \cdot {L^{imp}(\pi,\Ad,1)} \cdot B([\pi_\inf]_b,[\pi_\inf^\vee]_t)
$$

\subsubsection{Archimedean computations}
\label{unitary_arch_comp}

The archimedean term in the above equation is computed in \cite[Lemmas 5.13 \& 5.14]{Che22} and is equal to:
$$
B([\pi_\inf]_b,[\pi_\inf^\vee]_t) = (-1)^{a_1} \cdot 2^{a_2} \cdot \Gamma(\pi \x \pi^\vee,1)
$$
where $a_1 \in \Z$ and $a_2 \in \Z$ are some explicit integers.
Note that here we are using the Haar measure defined by the Gauge form $\w_\inf = \bigwedge_{i,j} dg_{i,j} / \det g^3$ on $U(F_\inf) = \GL(F_\inf)$ (see \S\ref{unitary_facts}), rather than the one used by Chen. This is why no $\pi^d$ factor appears, in contrast with \S\ref{archimedean_SBC}. Finally, we get the following formula:
\begin{equation}
\label{complete_unitary_L_value}
[\d^b(\ph_\pi^\circ),\d^t(\ph_{\pi^\vee}^\circ)] \sim u_{E/F} \cdot {\Lambda^{imp}(\pi,\Ad,1)}.
\end{equation}

\subsection{The adjoint $L$-value formula for $U_{E/F}$} 

\subsubsection{Congruence modules}
\label{congruence_modules}
Let $\TT$ be local $\O$-algebra which is finite, flat and reduced. Let $M$ be a $\TT$-module which is finite and flat over $\O$, and let $\l : \TT \to \O$ be an augmentation, i.e. a surjective map of $\O$-algebras. Let $M[\l] \subset M$ denote the submodule of elements canceled out by $\Ker(\l)$. The $\l$-rank of $M$ is $\mathrm{rank}_\l(M):= \dim_\K(M[\l] \otimes_\O \K)$. Consider the following composition of maps:
$$
\i_\l(M) : M[\l] \to M \to \mathrm{Hom}_\O(M^*,\O) \to \mathrm{Hom}_\O(M^*[\l],\O)
$$
where $M^*=\mathrm{Hom}_\O(M,\O)$ is the dual module of $M$, endowed with its natural $\TT$-module structure. The \textit{congruence module} of $\l$ on $M$, denoted $C_\l(M)$, is defined to be the cokernel of $\i_\l(M)$. Its Fitting ideal $\eta_\lambda(M) := \mathrm{Fitt}_\O(C_\lambda(M))$ is called the \textit{congruence number} of $\lambda$ on $M$. When $M = \TT$, $C_\lambda(M)$ and $\eta_\lambda(M)$ are denoted respectively by $C_\lambda$ and $\eta_\lambda$, and are simply called the {congruence module} and {congruence number} of $\lambda$. In that case $C_\lambda \simeq \O/\eta_\l$ is in fact a ring. More generally, for a $\TT$-module $M$ of $\l$-rank $1$, one has that $C_\lambda(M) \simeq \O/\eta_\l(M)$ and that $\eta_\l \subset \eta_\l(M)$. \\

We have the following useful lemma:

\begin{lemma}
\label{pairing_lemma}
Let $M$ and $N$ be two $\TT$-modules, finite flat over $\O$, and let:
$$
[\cdot,\cdot ]: M \x N \to\O
$$
be a perfect pairing which is $\TT$-equivariant. Then, if $\l: \TT \to \O$ is an augmentation of $\TT$, we have that $C_\l(M) \simeq C_\l(N)$. Moreover, assume that $M$ and $N$ have $\l$-rank $1$. Then, if $m_0$ and $n_0$ are respective $\O$-bases of $M[\l]$ and $N[\l]$, one has:
$$
\eta_\l(M) = [m_0,n_0]
$$
\end{lemma}

See \cite[Lemmas 3.5 \& 3.7]{BKM21} or \cite[Proposition 2.3]{TU22} for a proof.

\subsubsection{The main theorem}

Let $E/F$ be a quadratic extension of totally real number fields and let $U_{E/F}$ denotes the quasi-split unitary group in $3$-variables associated with $E/F$. Let $\pi$ be a cohomological cuspidal automorphic representation of $U_{E/F}(\A_F)$ and assume that $\pi$ is non-endoscopic (see \ref{non-endoscopic}) and globally generic. In this section we prove an adjoint $L$-value formula relating the (cohomological) congruence number of $\pi$ to the quotient of the special value $L(\pi,\Ad,1)$ by the product of the automorphic periods of $\pi$. \\

Assume that $\pi$ is spherical at places of $F$ which are non-split in $E$. Let $\n \subset \O_F$ be the mirahoric level of $\pi$ and let $K_f = K_1(\n) \subset U_{E/F}(\A_{F,f})$ be the mirahoric subgroup of level $\n$ (see \S\ref{global_unitary_newforms}). Let $\mu \in X^+(T_3)$ be the cohomological weight of $\pi$. Let $Y_U(K_f)$ be the adelic variety of level $K_f$ associated with $U_{E/F}$ and let $H^t_!(Y_U(K_f),\L_{\mu}(\O))$ be its inner cohomology group of (top) degree $t = 3 \cdot \dim(F)$ and weight $\mu$. Let $S$ be a finite set of finite places of $F$ containing the set $S_\pi$ of places where $\pi$ is ramified (i.e. dividing $\n$), the set $S_p$ of places dividing $p$, and the set $S_{E/F}$ of places which are ramified in $E$. Let $h^S$ be the spherical Hecke algebra outside of $S$ acting faithfully on the cohomology. Let $H^t_\m$ denote the localization of $H^t_!(Y_U(K_f),\L_{\mu}(\O))$ at the maximal ideal $\m$ of $h^S$ corresponding to $\pi$, and let $\TT_U := h^S_\m /(\O-tors)$ be the torsion-free quotient of the localization of $h^S$ at $\m$. Let $\eta_\pi(H^t_\m)$ be the congruence number of $\pi$ on the $\TT_U$-module $H^t_\m$. In \S\ref{unitary_automorphic_periods}, we define two automorphic periods $\Om_b(\pi)$ and $\Om_t(\pi)$ in $\C^\x /\O^\x$ attached to the representation $\pi$. The main result of this section is the following theorem:

\begin{theorem}
\label{unitary_adjoint_L_value}
Assume that the different $\mathfrak{d}_F$ of $F$ splits in $E$. Let $\pi$ be a cohomological cuspidal automorphic representation of $U_{E/F}(\A_F)$ which is non-endoscopic and globally generic. Assume that $\pi$ is  spherical at places of $F$ which are non-split in $E$. Assume that $\mu$ is $p$-small, that $p \notin S_\partial$ and that $p \nmid  2N_{F/\Q}(\n)h_U(\n)D_F$. Then:
$$
\eta_\pi(H^t_\m) \sim u_{E/F} \cdot \frac{\Lambda^{imp}(\pi, \Ad,1)}{\Om_b(\pi) \cdot \Om_t(\pi^{\vee})}
$$
where $u_{E/F} \in \Q(\pi_f)^\x$ is a constant conjectured to be $1$. For $F=\Q$, one has that $u_{E/\Q} =1$.
\end{theorem}

\begin{proof}[Proof of \ref{unitary_adjoint_L_value}]
The proof is similar to the proof of \ref{adjoint_L_value} but simpler. We quickly recall the notation. Let $\pi$ be a cuspidal automorphic representation of $U(\A_F)$ which is cohomological (of weight $\mu$) and globally generic. Assume that $\pi$ is only ramified at split places and let $\n \subset \O_F$ be its mirahoric level. Let $\phi_\pi^\circ \in \phi_f$ and $\phi_{\pi^\vee}^\circ$ be the essential forms of $\pi$ and $\pi^\vee$. Assume that $p \nmid 2 N_{F/\Q}(\n)h_U(\n) D_F$ .Then it follows from the previous computations (see formula \ref{complete_unitary_L_value}) that:
$$
[\d^b(\ph_\pi^\circ),\d^t(\ph_{\pi^\vee}^\circ)] \sim u_{E/F} \cdot {\Lambda^{imp}(\pi,\Ad,1)} 
$$
where $u_{E/F}$ is defined in \ref{full_lapid_mao}. We have $u_{E/\Q} = 1$ when $F =\Q$ (\ref{triviality_uE}) and $u_{E/\Q} \in \Q(\pi)^\x$ in general (\ref{rationality_uE}). In the above equality, $\d^q$ (for $q \in \{b,t\}$) are the Eichler-Shimura maps defined in \S\ref{ES_maps}, and:
$$
[ \cdot, \cdot] : H^b_{!}(Y_U(K_f),\L_{\mu}(\O))_\m \x H^t_{!}(Y_U(K_f),\L_{\mu}(\O)^\vee)_\m \to \O
$$
is the Poincaré pairing defined in \S\ref{unitary_poincare_pairing}. This pairing is perfect when $p \notin S_\partial$. Moreover, it is equivariant for the action of the torsion-free localized Hecke algebra $\TT_U$. Note that the action on the right factor is twisted so that $T \in \TT_U$ acts by $T^\vee$. Let $\l_\pi : \TT_U \to \O$ be the Hecke eigensystem associated with $\pi$, and let $\eta_\pi(H^t)$ be the congruence number of $\l_\pi$ on $H^t = H^t_{!}(Y_U(K_f),\L_{\mu}(\O)^\vee)_\m$. By the remark below \ref{non-endo_lambda_rank}, $H^t$ is a Hecke module of $\l_\pi$-rank~$1$. Thus we can apply \ref{pairing_lemma} to $[\cdot,\cdot]$. By definition of automorphic periods (see \S\ref{unitary_automorphic_periods}), we get that:
$$
\eta_\pi(H^t) \sim \left[\frac{\d^b(\ph_\pi^\circ)}{\Om_b(\pi)},\frac{\d^t(\ph_{\pi^\vee}^\circ)}{\Om_t(\pi^\vee)}\right]
$$
This proves the theorem.
\end{proof}

\section{A divisibility of automorphic periods for the real stable base change}
\label{part_SBC}

After recalling the relevant properties of the stable base change for $\GL$, we prove \ref{thmB} (see \ref{SBC_divisibility}). We then deduce the divisibility of \ref{thmA} (see \ref{divisibilite_unitaire}) between the periods of a cuspidal automorphic representation of $U_E$ and the periods of its stable base change to $\GL(E)$. In paragraph \S\ref{non_self_dual}, we present the results for non self-dual representations. In the introduction of this paper, automorphic representations of $U_E$ are denoted by $\pi_U$. In order to simplify the notation, such representations will be denoted by $\pi$ in this section. \\

\subsection{Stable base change}
\label{sss_sbc}

In this paragraph, we temporarily adopt general notation. We consider the case of an arbitrary quadratic extension $E/F$ of number fields and suppose that $n \geq 2$ is any integer. Write $E=F[\tau]$, for $\tau=\sqrt{\d}$, $\d \in F^\x$. Let $\psi_F := \psi_\Q \circ \mathrm{Tr}_{\A_F/\A} : F \bs \A_F \to \C^\x$ and $\psi_E : E \bs \A_E \to \C^\x$ defined by:
$$
\psi_E(x) = \psi_F(\mathrm{Tr}_{\A_E/\A_F}(\tau x)), \quad x \in \A_E
$$
Since $\s(\tau) = -\tau$,  we have that $\psi_E$ is trivial on $\A_F$. We write $\psi_E = \bigotimes_w \psi_{E,w}$ for its factorization as a tensor product of local additive characters.

\subsubsection{$\Ld$-groups and $\Ld$-embeddings.}

The $L$-group of $U_{E/F}$ is given by:
$$
{}^\Ld U_{E/F} = \GLn(\C) \rtimes \Gal(E/F) 
$$
where the non-trivial element $\s$ in $\Gal(E/F)$ acts on $ \GLn(\C)$ by $g \mapsto J_n^\top g^{-1}J_n^{-1}$; where:
$$
J_n= \mathrm{antidiag}((-1)^{n-1},\cdots, -1,1)
$$

We denote by $G_E$ the $F$-algebraic group $\mathrm{Res}_{E/F}(\mathrm{GL}_{n/E})$, whose $A$-points are given, for any $F$-algebra $A$, by $G_{E}(A) = \mathrm{GL}_n(E\otimes_F A)$. The $L$-group of $G_E$ is:
$$
{}^L G_{E} =( \GLn(\C) \x \GLn(\C)) \rtimes \Gal(E/F) 
$$
where $\s$ acts on $ \GLn(\C) \x \GLn(\C)$ via $(g_1,g_2) \mapsto (g_2,g_1)$. \\

We consider the morphism of $L$-groups $\mathrm{SBC}:  \,  {}^L U_{E/F} \to {}^L G_E$ defined by:
$$
g \rtimes 1 \mapsto (g,{}^t
 g^{-1}) \rtimes 1 \quad \mbox{ and } \quad 
I_n \rtimes \s \mapsto  (\Phi,\Phi^{-1}) \rtimes \s
$$

\subsubsection{Local stable base change.} 
For a place $v$ of $F$, we note $\L_{F_v}$ for the Weil-Deligne group of $F_v$. Suppose first that $v$ is split in $E$, so that $E_v = F_v \x F_v$. Then the local stable base change of $\pi_v$ is defined to be:
$$
\mathrm{SBC}(\pi_v) = \pi_v \otimes \pi_v^\vee
$$
Now if $v$ is non-split so that $E_v = E \otimes_F F_v$ is a field, 
and let $\phi_{\pi_v}: \L_{F_v} \to {}^L U_{E/F}$ be the $L$-parameter associated to $\pi_v$ by the local Langlands correspondence for $U_{E/F}(F_v)$ (see \cite[Theorem 2.5.1]{Mok}). Using the $L$-embedding $\mathrm{SBC}$ we get the $L$-parameter of $G_E(F_v)$ given by:
$$
\mathrm{SBC} \circ \phi_\pi: \L_{F_v} \to {}^L G_E
$$
Then the stable base change $\mathrm{SBC}(\pi_v)$ of $\pi_v$ is defined to be the unique irreducible admissible representation $\Pi$ of $G_E(F_v) =\GLn(E_v)$ corresponding to $\mathrm{SBC} \circ \phi_\pi$  through the local Langlands correspondence for $G_E(F_v)$. When $U_{E/F}(F_v)$ and $\pi_v$ are unramified, the local stable base change has a simpler expression in terms of Satake parameters that we briefly recall here. We refer to {\cite{Minguez07}} for details. Let $K_v$ be the (hyperspecial) maximal compact subgroup $U_{E/F}(F_v) \cap \GLn(\O_{E,v})$ of $U_{E/F}(F_v)$, and assume that $\pi_v$ is $K_v$-spherical. The $K_v$-spherical representations are classified by $m$-tuples of unramified characters of $L^\x$ up to Weyl group action. Let $(\chi_1,\dots,\chi_m)$ be the $m$-tuple corresponding to $\pi_v$. Then, the Satake parameter of $\pi_v$ is the $\GLn(\C)$-conjugacy class of:
$$
\diag(\chi_1(\varpi_L)^{1/2},\dots,\chi_m(\varpi_L)^{1/2},1,\chi_m(\varpi_L)^{-1/2},\dots,\chi_1(\varpi_L)^{-1/2}) \rtimes \s \in {}^L U_{E/F}
$$
Then $\mathrm{SBC}(\pi_v)$ is the spherical representation of $\GLn(E_v)$ whose Satake parameter is given by the $\GLn(\C)$-conjugacy class of:
$$
\diag(\chi_1(\varpi_L),\dots,\chi_m(\varpi_L),1,\chi_m(\varpi_L)^{-1},\dots,\chi_1(\varpi_L)^{-1}) \in \GLn(\C)
$$
In particular, one can see that the map SBC is injective on the set of $K_v$-spherical representations of $U_{E/F}(F_v)$ (to the set of spherical representations of $\GLn(E_v)$). By duality, it induces a surjective map between the spherical Hecke algebras of $\GLn(E_v)$ and $U_{E/F}(F_v)$ (with coefficients in $\C$):
$$
\H(\GLn(E_v),\C) \surj \H(U_{E/F}(F_v),\C)
$$

\subsubsection{Global stable base change}
\label{endoscopy}

Let $\pi$ be a cuspidal automorphic representation of $U_{E/F}(\A_F)$. We say that an automorphic representation $\Pi$ of $\GLn(E)$ is a weak base change of $\pi$ if for almost all places $v$ of $F$, we have that:
$$
\Pi_v \simeq \mathrm{SBC}(\pi_v)
$$

The following theorem, due to Mok (see \cite[Corrolary 4.3.8]{Mok}), ensures the existence and the unicity of a weak base change:

\begin{theorem}[Weak base change]
Let $\pi$ be a cuspidal automorphic representation of $U_{E/F}$. Then it admits a unique weak base change to $\GLn(E)$.
\end{theorem}

We will now restrict ourselves to the case $n=3$ we are interested in, and we describe the representations $\pi$ whose weak base change is cuspidal. For this, let $H$ be the group $U(2) \times U(1)$ and let :
$$
\xi_H: {}^L H \to {}^L U_{E/F}
$$
be the $L$-embedding defined in \cite[Section 4.8, 1.(a)]{ARU3}. By the principle of functoriality, this $L$-embedding should determine a transfer map from the $L$-packets of $H$ to the $L$-packets of $U_{E/F}$, whose existence has in fact been established by Rogawski. More precisely, Rogawski has proven \cite[Theorem 13.1.1]{ARU3} that for each finite place $v$ of $F$, there exists a map $\xi_H : \Psi(H_v) \to \Psi(U_v)$ from the set $\Psi(H_v)$ of $L$-packets on $H_v = H(F_v)$ to the set $\Psi(U_v)$ of $L$-packets on $U(F_v)$. We introduce the following definition:

\begin{definition}[Non-endoscopic automorphic representations of $U_{E/F}$]
\label{non-endoscopic}
We say that a discrete $L$-packet $\Psi = \otimes_v \Psi_v$ on $U_{E/F}$ is \textit{non-endoscopic} if there is no discrete $L$-packet $\psi =  \otimes_v \psi_v$ on $H$ such that $\Psi_v = \xi_H(\psi_v)$ for almost all $v$. We say that an automorphic representations of $U_{E/F}(\A_E)$ is \textit{non-endoscopic} if it belongs to a non-endoscopic $L$-packet.
\end{definition}

We then have the following more precise result \cite[Theorem 13.3.3]{ARU3}:

\begin{theorem}[Strong base change $(n=3)$]
Let $\pi$ be a cuspidal automorphic representation of $U_{E/F}$, and let $\Pi$ be its weak base change to $\GL(\A_E)$. Then $\Pi$ is cuspidal if and only if $\pi$ is non-endoscopic. Moreover, in this case $\Pi$ is a strong base change, that is for all place $v$:
$$
\Pi_v \simeq \mathrm{SBC}(\pi_v).
$$
\end{theorem}

\begin{remark}
The existence of a strong stable base change can be obtained under less restrictive assumptions (i.e. without assuming that $\pi$ is non-endoscopic), for a general $n$, by assuming only that the stable base change of $\pi$ is generic, i.e. that $\pi$ is a cuspidal representations of Ramanujan type.
\end{remark}

\subsubsection{Galois representations}
\label{SBC_galois}

To complete the picture, we now describe the stable base change on the side of Galois representations, following \cite[\S 1.1]{BLGGT14}. Let ${}^L U_p$ denote ${}^L U(\bar{\Q}_p) = \GLn(\bar{\Q}_p) \rtimes \Gal(E/F)$. A representation $\rho_U$ of $\G_F$ in ${}^L U_p$ is a group morphism $\rho_U : \G_F \to {}^L U_p$ such that the following diagram commutes:
$$
\begin{tikzcd}
 \Gal(\overline{F}/F) \arrow[r, "\rho_U"] \arrow[rd] & \GLn(\bar{\Q}_p) \rtimes \Gal(E/F) \arrow[d, "\mathrm{pr}_2"] \\
 &  \Gal(E/F) \\
\end{tikzcd}
$$

Let $\rho : \G_E \to \GLn(\bar{\Q}_p)$ be a Galois representation. For $\vs \in \G_F \bs \G_E$, we define the representation $\rho^{\vs} : \G_E \to \GLn(\bar{\Q}_p)$ by $\rho^{\vs} = \rho(\vs g \vs^{-1})$. We say that $\rho$ is conjugate self-dual if $\rho^{\vs} \simeq \rho^\vee$  i.e. if there exists an element $\ga_\vs \in \GLn(\bar{\Q}_p)$ such that:
\begin{equation}
\label{def_csd}
\rho(\vs g \vs^{-1}) = \ga_\vs {}^t \rho(g)^{-1}  \ga_\vs^{-1}, \quad \forall g \in  \G_E
\end{equation}
Of course the fact of being conjugate self-dual does not depend on the choice of $\vs \in \G_F \bs \G_E$, as one can take $\ga_{g \vs} = \rho(g)  \ga_\vs$ to see that $\rho^{g \vs} \simeq \rho^\vee$, for any $g \in \G_E$. \\

Let $\rho : \G_E \to \GLn(\bar{\Q}_p)$ be an irreducible Galois representation which is conjugate self-dual. Let us show that $\rho$ induces a Galois representation $\rho_U : \G_F \to {}^L U_p$. Let $\ga_0 \in \G_F \bs \G_E$ be the element such (\ref{def_csd}) holds. Then, we have:
$$
\rho(\vs^2) \rho(g) \rho(\vs)^{-2} = \rho(\vs^2 g\vs^{-2}) = \ga_\vs {}^t \rho(\vs g\vs^{-1})^{-1} \ga_\vs^{-1} = \ga_\vs {}^t \ga_\vs ^{-1}   \rho(g) (\ga_\vs {}^t \ga_\vs ^{-1})^{-1}
$$
Since $\rho$ is irreducible, we deduce by Schur's lemma that $\rho(\vs^2)$ is a scalar multiple of $\ga_\vs {}^t \ga_\vs ^{-1} $. By substituting $g = \vs^2$ in the above equation we see that the scalar multiple is in fact a sign:
$$
\rho(\vs^2) = \pm \ga_\vs {}^t \ga_\vs ^{-1}
$$
One can check that this sign does not depend on $\vs$. We denote it by $\mathrm{sgn}(\rho)$. Assume that $\mathrm{sgn}(\rho) = (-1)^{n-1}$. We then define $\rho_U : \G_F \to {}^L U_p$ by putting $\rho_U(g) = (\rho(g),1)$ if $g \in \G_E$ and $\rho_U(\vs) = (\ga_\vs \Phi^{-1},\s)$ if $\vs \in \G_F \bs \G_E$. Then one check that $\rho_U$ is a group morphism, and by contruction it is a representation of $\G_F$ in ${}^L U_p$.

Conversely, if $\rho_U : \G_F \to {}^L U_p$ is a representation of $\G_F$ then we obtain a representation of $\rho : \G_E \to \GLn(\bar{\Q}_p)$ by restriction : $\rho = \mathrm{pr}_1 \circ \rho_U|_{\G_E}$. Let $\vs \in \G_F \bs \G_E$ and put $\rho_U(\vs) = (\ga,\s)$. Then one check that:
$$
\rho(\vs g \vs^{-1}) = (\ga\Phi) {}^t \rho(g)^{-1}  (\ga\Phi)^{-1}
$$
Thus $\rho$ is conjugate self-dual. Moreover one check that:
$$
\rho(\vs^2) = \ga \Phi {}^t \ga^{-1} \Phi^{-1} = (-1)^{n-1} \ga \Phi \cdot {}^t (\ga \Phi)^{-1}
$$
Thus $\mathrm{sgn}(\rho) = (-1)^{n-1}$. One can sees from the above computations that the two operations are the two operations are inverse of each other. Thus there is a bijection:
$$
\left \{
    \begin{array}{c}
    	\\
        \text{Irreducible representations} \\ 
         \rho_U : \G_F \to {}^L U(\bar{\Q}_p) \\
         \\
    \end{array}
\right \} \longleftrightarrow
\left \{
    \begin{array}{c}
    	\text{Conjugate self-dual} \\
        \text{irreducible representations} \\
         \rho : \G_E \to \GLn(\bar{\Q}_p) \\
         \text{with } \mathrm{sgn}(\rho) = (-1)^{n-1}
    \end{array}
\right \}
$$

Let $\pi$ a cohomological cuspidal automorphic representation which is non-endoscopic and let $\Pi$ be its stable base change, which is a conjugate self-dual cohomological cuspidal automorphic representation of $\GLn(\A_E)$. In light of the above discussion, we define the Galois representation of $\pi$ to be the Galois representation $\rho_\Pi : \G_E \to \GLn(\K)$ attached to $\Pi$. In particular we will say that the Galois representation has residually absolutely irreducible image if $\rho_\Pi$ has residually absolutely irreducible (as a representation of $\G_E$).

\subsection{Stable base change for Hecke algebras}
\label{hecke_SBC}

In this paragraph, we explain how the stable base change is described from the perspective of Hecke algebras. More precisely we construct a morphism $\theta: \TT_{E} \to \TT_{U}$ between the localized cohomological Hecke algebras, which corresponds to the base change transfer. Under some conditions, we prove that this morphism is surjective.  \\

We begin by introducing some notation. We denote by $S_{F,s}$ the set of places of $F$ which are split in $E$. Let $S_{E,s}$ be the set of places of $E$ above the split places of $F$;  this is a set of density 1. For each $v \in S_{F,s}$, we have fixed a splitting $F_v = E_w \times E_{w^{\s}}$. For every $v \in S_{F,s}$, we choose the place $w \in S_{E,s}$ corresponding to the first factor and denote by $R_E$ the set of such places, so that $S_{E,s} = R_E \sqcup R_E^\s$. Finally, let $S_p(F)$ (resp.  $S_p(E)$) be the set of places of $F$ (resp. of $E$) above $p$. \\

Let $\pi$ be a non-endoscopic cuspidal automorphic representation of $U_{E/F}(\A_F)$ and let $\Pi$ be its strong stable base change, which is a cuspidal automorphic representation of $\GLn(\A_E)$. Let $k_f$ be the mirahoric level of $\pi$ and let $\H_{U} = \H_{U}(k_f;\O)$ be the abstract spherical Hecke algebras with coefficients in $\O$, defined in \S\ref{full_unitary_hecke}. Similarly, let $K_f = \prod_w K_{f,w}$ be the mirahoric level of $\Pi$, and let $S_{\Pi}$ be the set of finite places of $E$ such that $K_{f,w} \neq \GLn(\O_{E,w})$ and let $S_E = S_\Pi \cup S_p(E)$. Let $\H_{E,s} = \H_{E,s}(K_f;\O)$ be the \textit{split} spherical Hecke algebra of level $K_f$ outside of $p$ for $G_E$, defined by:
$$
\H_{E,s}(K_f;\O) = \bigotimes_{w \in S_{E,s} - S_E} \H(K_{w};\O)
$$
where $\H(K_{w};\O) = \O[T_{w,1},\dots, T_{w,n-1},T_{w,n}^\pm]$ (see \S\ref{hecke_corr}). Let:
$$
\theta : \H_{E,s} \to \H_{U}
$$
be the morphism sending $T_{w,i}$ to $T_{v,i}$ and $ T_{w^\s,i}$ to $(T_{v,i})^\vee$ for all $v \in S_{F,s}$ where $\pi$ is unramified and $w\in R_E$ such that $w \mid v$.  \\

From now we assume that $\pi$, and thus $\Pi$, are cohomological. Let $\mu_U \in X^+(T_F)$ and $\mu_E = (\mu_U,\mu_U^\vee) \in X^+(T_E)$ be the cohomological weights of $\pi$ and $\Pi$.  Let  $h_{U}$ and $h_{E,s}$ be the cohomological Hecke algebras acting faithfully on the cuspidal cohomology of $Y_U(k_f)$ and $Y_E(K_f)$ with coefficients in $\L_{\mu_F}(\O)$ and $\L_{\mu_E}(\O)$, defined in \S\ref{full_unitary_hecke} and \S\ref{hecke_corr}. The morphism $\theta$ induces a morphism between these cohomological Hecke algebras. To see this, we consider the following two decompositions of the cuspidal cohomology:
$$
H^\bullet_{cusp}(Y_E(K_f), \L_{\mu_E}(\C)) = \bigoplus_{\Si \in \mathrm{Coh}(G_E,\mu_E,K_f)} H^\bullet(\g, K_\inf  ; \Si_{\inf} \otimes L_{\mu_E}(\C)) \otimes \Si_f^{K_f}
$$
and (each $\pi$ appears with multiplicity one, see for exemple \cite[Theorem 13.3.1]{ARU3} for $n=3$):
$$
H^\bullet_{cusp}(Y_U(k_f), \L_{\mu_U}(\C)) =  \bigoplus_{\rho \in \mathrm{Coh}(U_{E/F},\mu_U,k_f)}  H^\bullet(\g_n, K_n ; \rho_{\inf} \otimes L_{\mu_U}(\C)) \otimes \rho_f^{k_f}
$$
These two decompositions are equivariant with respect to the action of $\H_{E,s}$ and $\H_U$, acting on the left-hand side by Hecke correspondences and on the right-hand side via Hecke operators. Thus, the Hecke action on the cuspidal cohomology groups is completely determined by the Hecke action on finite parts of cuspidal automorphic representations. Now, let $\rho \in \mathrm{Coh}(U_{E/F},n,k_f)$ and let $\Si = \mathrm{SBC}(\rho)$. If $v \in S_{F,s} - S_\pi$ and $w \in R_E$, $w \mid v$, we have $\Si_w = \rho_v$ and $\Si_{w^\s} = \rho_v^\vee$, and:
\begin{itemize}
\item $\l_\Si(T_{w,i}) = \l_\rho(T_{v,i})= \l_\rho(\theta(T_{w,i}))$;
\item $\l_\Si(T_{w^\s,i}) = \l_{\rho^\vee}(T_{v,i})= \l_\rho(T_{v,i}^\vee) =\l_\rho(\theta(T_{w^\s,i})).$
\end{itemize}
where $\l_\Si : \H_E(K_f,\O) \to \O$ and $\l_\rho : \H_F(k_f,\O) \to \O$ are the Hecke eigensystems associated with $\Si$ and $\rho$. Thus, for all $T \in\H_{E,s}(K_f;\O)$, we have:
$$
\l_\Si(T) = \l_\rho(\theta(T))
$$
Let $T,T' \in \H_{E,s}(K_f;\O)$ be two Hecke operators defining the same correspondence on $H^\bullet_{cusp}(Y_E(K_f), \L_{\mu_E}(\O))$. Then $\l_\Pi(T)= \l_\Pi(T')$ for all $\Si \in \mathrm{Coh}(G_E,\mu_E,K_f)$, and so $\l_\rho(\theta(T)) = \l_\rho(\theta(T'))$ for all $\rho \in \mathrm{Coh}(U_{E/F},\mu_U,k_f)$. Hence, $\theta(T)$ and $\theta(T')$ define the same correspondence on $H^\bullet_{cusp}(Y_U(k_f), \L_{\mu_U}(\C))$. Since $h_{U}(k_f;\O)$ acts faithfully on $H^\bullet_{cusp}(Y_U(k_f), \L_{\mu_U}(\C))$, they have the same image in $h_{U}(k_f;\O)$. This shows that $\theta$ induces a morphism on the (torsion-free) cohomological Hecke algebras, which we also denote $\theta$:
$$
\theta : h_{E,s} \to h_{U}
$$

Let $\l_\pi: h_{U} \to \O$ be the Hecke eigensystem associated with $\pi$. We denote by $\overline{\lambda_\pi}$ its reduction modulo $\varpi$ and by $\m_\pi$ the kernel of $\overline{\lambda_\pi}$. It is a maximal ideal of $h_{U}$, and we denote by $\TT_{U}$ the localization of $h_U$ at $\m_\pi$. Similarly, let $\l_{\Pi}: h_{E,s} \to \O$ be the eigensystem associated with $\Pi =\mathrm{SBC}(\pi)$. Thus, we have that $\lambda_{\Pi} = \lambda_\pi \circ \theta$. Let $\m_{\Pi,s} = \Ker \overline{\lambda_{\Pi}}$, and let $\TT_{E,s}$ be the localization of $h_{E,s}$ at $\m_{\Pi,s}$. Since $t \in \m_{\Pi,s} \Leftrightarrow \theta(t) \in \m_\pi$, $\theta$ induces, by localization, a morphism of $\O$-algebras, which we still denote $\theta$:
$$
\theta : \TT_{E,s} \to \TT_{U}
$$

Assume that the Galois representation attached to $\pi$ in \S\ref{SBC_galois} is residually absolutely irreducible. In that case, one can show, as for the Hecke algebra $\TT_\Q$ of $\GLn(\Q)$ in \cite[Lemma 4.2]{TR_BC}, that $\TT_U$ is generated by the Hecke operators at prime ideals of $F$ which are split in $E$. Thus, the morphism $\theta : \TT_{E,s} \to \TT_{U}$ is surjective by construction. \\

Let $\H_{E} = \H_{E}(K_f;\O)$ be the \textit{full} (outside of $p$) abstract spherical Hecke algebras of level $K_f$, with coefficients in $\O$, defined in \S\ref{hecke_corr}, and let $h_{E}$ be the {full} (outside of $p$) cohomological Hecke algebras acting faithfully on the cuspidal cohomology of $Y_E(K_f)$ with coefficients in $\L_{\mu_E}(\O)$. Let $h_\Pi : h_{E} \to \O$ be the Hecke eigensystem associated with $\Pi$ et $\m_\Pi \subset h_{E}$ the corresponding maximal ideal. The inclusion $h_{E,s} \subset h_{E}$ induces an inclusion $\TT_{E,s} \subset \TT_E$ by localizing at $\m_{\Pi,s}$ and $\m_{\Pi}$. We recall from \S\ref{galois_reps} that $\Pi$ is associated with a Galois representation $\rho_{\Pi} : \Gal(\overline{E}/E) \to \GLn(\O)$. We then have the following lemma:
\begin{lemma}
\label{SBC_surjectivity}
Assume that the Galois representation $\rho_{\Pi}$ associated with $\Pi$ is residually absolutely irreducible. Then:
$$
\TT_{E,s} = \TT_E
$$
\end{lemma}

\begin{proof}
The proof uses a similar argument that Carayol in \cite{Carayol}. More precisely, similarly to \cite[Section 2.2]{Carayol}, we can construct a Galois representation $\rho: \Gal(\overline{E}/E) \to \GLn(\TT)$ with coefficient in $\TT = \TT_E$ such that if $w \notin S_\Pi \sqcup S_p(E)$ is a finite place of $E$, we have:
$$
\det(X I_n - \rho(\Frob_{w})) = \sum_{i=0}^n (-1)^i q_w^{i(i-1)/2} T_{w,i} X^{n-i}
$$
Alternatively, one can deduce the existence of this representation from \cite[Corollary 5.4.4]{Scholze15}. Then for all $w \in S_{E,s} - (S_\Pi \sqcup S_p(E))$, the polynomial $\det(X I_n - \rho(\Frob_{w}))$ has coefficients in $\TT_s := \TT_{E,s}$. But this set of prime ideals is of density $1$, so by the Chebotarev's density theorem, $\Gal(\overline{E}/E)$ is topologicaly generated by $\{ \Frob_{w}, \,w \in S_{E,s} - (S_\Pi \sqcup S_p(E)) \}$. Since $\TT_s$ is closed in $\TT$, we have by continuity that $\det(X I_n - \rho(g))$ has coefficients in $\TT_s$, for all $g \in \Gal(\overline{E}/E)$. In particular $T_{w,i} \in \TT_s$, for all $w \notin S_\Pi \sqcup S_p(E)$ and $0 \leq i \leq n$. Since these operators generate $\TT$, this proves the lemma.
\end{proof}

To conclude, assuming that the Galois representation associated with $\pi$ is residually absolutely irreducible, we obtain a surjective morphism of $\O$-algebras:
$$
\theta  : \TT_E \to \TT_U
$$

\subsection{$L$-functions and stable base change}
\label{L-functions}
Let $\pi$ a cuspidal automorphic representation of $U_{E/F}$ which is non-endoscopic and let $\Pi$ be its stable base change, which is a cuspidal automorphic representation of $\GLn(\A_E)$. We decompose $\Pi = \otimes_v \Pi_v$ as a tensor product of local representations $\Pi_v$ of $\GLn(E_v)$, where $v$ is a place of $F$ and $E_v = E \otimes_F F_v$. For simplicity we assume that $n$ is odd (the even case is similar, with the role of the two Asai $L$-functions reversed). We first recall the definition of the primitive adjoint $L$-functions of $\pi$ and Asai $L$-functions of $\Pi$, and their relations.

\subsubsection{Primitive $L$-functions and stable base change}
\label{primitive_L_functions}

Let $\As^+ : {}^L G_{E} \to \mathrm{GL}(\C^n \otimes \C^n)$ be the $n^2$-dimensional representation of ${}^L G_{E}$ defined by:
$$
\As^+(g_1,g_2)(x\otimes y) = g_1 x\otimes g_2 y  \quad \mbox{ and } \quad 
\As^+(\s)(x \otimes y) =  y \otimes x
$$
where $(g_1,g_2)$ denotes  $(g_1,g_2) \rtimes 1 \in {}^L G_{E}$ and $\s$ denotes $(1,1) \rtimes \s \in {}^L G_{E}$. For each place $v$ of $F$, the local factor $\Pi_v$ determines, via the local Langlands correspondence for $\GLn(E_v)$, a Weil-Deligne representation $\phi_{\Pi_v} : \L_{F_v} \to {}^L G_{E}$. The Asai local $L$-factor $L(\Pi_v,\As^+,s)$ is defined to be the local $L$-factor associated with the $n^2$-dimensional Weil-Deligne representation $\As^+ \circ \phi_{\Pi_v}$. Then the Asai $L$-function $L(\Pi,\As,s)$ is defined as the Euler product, over the finite places $v$ of $F$, of the Asai local $L$-factors $L(\Pi_v,\As,s)$. It follows from the works of Flicker \cite[Theorem p.297]{Flicker88}, Flicker-Zinoviev \cite[Theorem]{FZ95}, and Matringe \cite[Theorem 5.3]{Mat09}, that $L(\Pi_v,\As,s)$ has an analytic continuation as a meromorphic function on some right half plane containing $s=1$. Moreover, it has a simple pole at $s=1$ since $\Pi$ is a stable base change, and thus $\GLn(\A_F)$-distinguished (see \ref{th_mok} below). We will sometimes simply write $L(\Pi,\As,s)$ for $L(\Pi,\As^+,s)$.

Let $\As^- :=  \As^+ \otimes \chi_{E/F}$. Similarly, the twisted Asai $L$-function $L(\Pi,\As^-,s)$ of $\Pi$ is defined as the Euler product of the local $L$-factors $L(\Pi_v,\As^-,s)$ attached, for each finite place $v$ of $F$, to the $n^2$-dimensional Weil-Deligne representation $\As^- \circ \phi_{\Pi_v}$. Similarly, one can prove that $L(\Pi,\As^-,s)$ has an analytic continuation as a meromorphic function to some right half plane containing $1$. Moreover, $L(\Pi,\As^-,s)$ is holomorphic and nonzero at $s=1$ as $\Pi$ is a stable base change. \\

Let $\Ad$ be the adjoint action of ${}^L U_{E/F}$ on $\mathfrak{sl}_n(\C)$ (the space of $n \x n$ traceless matrices with coefficients in $\C$) defined by:
\begin{itemize}
\item  $\Ad(g)\cdot X = gXg^{-1}$, for $g \in \GLn(\C)$
\item $\Ad(\s)\cdot X = -J_n{}^tX J_n^{-1}$.
\end{itemize} 
We also denote by $\Ad \otimes \chi_{E/F}$ the adjoint action twisted by the quadratic character $\chi_{E/F}$ attached to $E/F$, through which $\s$ acts via $ X \mapsto J_n{}^tX J_n^{-1}$. The adjoint $L$-function $L(\pi,\Ad,s)$ of $\pi$ is defined as the Euler product, over the finite places $v$ of $F$, of the local $L$-factors $L(\pi_v,\Ad,s)$, which is defined to be the $L$-factor associated with $\Ad \circ \rho_v$, where $\rho_v$ is the Weil-Deligne representation corresponding to $\pi_v$ through the local Langlands correspondence for $U_{E/F}$. The twisted adjoint $L$-function $L(\pi,\Ad \otimes \chi_{E/F},s)$ of $\pi$ is defined similarly. \\

The adjoint $L$-functions of $\pi$ and the Asai $L$-functions of its stable base change $\Pi$ are related by the following two relations:
\begin{equation}
\label{residue_formula_+}
 L(\Pi,\As^+,s) = \zeta_F(s) \cdot L(\pi,\Ad\otimes \chi_{E/F} ,s)
\end{equation}
\begin{equation}
\label{residue_formula_-}
L(\Pi,\As^-,s) = L(\chi_{E/F},s) \cdot L(\pi,\Ad ,s)
\end{equation}
where $\zeta_F$ is the Dedekind zeta function of $F$ and $L(\chi_{E/F},\cdot)$ is the Dirichlet $L$-function of $\chi_{E/F}$. Finally, the adjoint $L$-function of $\Pi$ (see \S\ref{adjoint_L_formula}) admits the following decomposition:
$$
L(\Pi,\Ad,s) = L(\pi,\Ad,s) \cdot L(\pi,\Ad \otimes \chi_{E/F},s) 
$$
which is also valid for the completed $\Lambda$-functions, which are obtained by adding the $\Gamma$-factors at archimedean places. \\

We now define the \textit{imprimitive} $L$-functions, which differ from the geniuine $L$-functions introduced above by a finite number of factors, and whose factors at the ramified places are simpler.

\subsubsection{Imprimitive Asai $L$-function} We first define the local $L$-factors. Let $v$ be a finite place of $F$ and let $q_v$ be the cardinality of the residue field of $F_v$. Assume first that $v$ splits into $w$ and $w'$ in $E$. The local \textit{imprimitive} Asai $L$-factor of $\Pi$ at $v$ is defined to be the local imprimitive Rankin-Selberg $L$-factor of $\Pi_w$ and $\Pi_{w'}$ at $v$:
$$
L^{imp}(\Pi_v,\As,s) = L^{imp}(\Pi_w \x \Pi_{w'},s)
$$ 
where $\Pi_w$ and $\Pi_{w'}$ are both viewed as representations of $\GLn(F_v)$. The imprimitive Rankin-Selberg is defined in \cite[\S3.1.1]{TR_BC}. Assume now that $v$ is non-split in $E$ and let $w$ be the place of $E$ upon $v$, and let $q_w$ be the cardinality of the residue field of $E_w$. Let us write the standard local $L$-factors associated with $\Pi_w$ as (see \cite{ZFSA}): 
\begin{equation}
\label{standard_L_factor}
L(\Pi_w,s) = \prod_{i=1}^{r} (1 - \a_i q_w^{-s})^{-1}
\end{equation}
with $\a_i \in \C^*$. One has that $r \leq n$ and that $\Pi$ is ramified if and only if $r<n$ (see \cite[Section 3]{Jacquet79}). Then, the local \textit{imprimitive} Asai $L$-factor of $\Pi$ at $v$ is defined by:
$$
L^{imp}(\Pi_v,\As,s) =  \prod_{i=1}^r(1-\a_i q_v^{-s})^{-1} \prod_{j<k}^r(1-\a_j\a_k q_v^{-2s})^{-1} 
$$
when $v$ is inert, and that:
$$
L^{imp}(\Pi_v,\As,s) =  \prod_{i=1}^r(1-\a_i^2 q_v^{-s})^{-1} \prod_{j<k}^r(1-\a_j\a_k q_v^{-s})^{-1} 
$$
when $v$ is ramified. In particular, $L^{imp}(\Pi_v,\As,s)$ is nowhere vanishing. \\

Of course, when $\Pi$ is unramified at $v$, the imprimitive local $L$-factors at $v$ is equal to the Asai local $L$-factor $L(\Pi_v,\As,s)$. In general, we have the following lemma, which follows from  \cite[Corollaries 3.3 \& 4.3]{Jo22} and \cite[Theorem 5.3]{Mat09}:
\begin{lemma}
There exists a polynomial $P_v \in \C[X]$ satisfying $P(0) = 1$ such that:
$$
L^{imp}(\Pi_v,\As,s) = P_v(q_v^{-s})L(\Pi_v,\As,s)
$$
Moreover, $P=1$ when $\Pi$ is unramified at $v$.
\end{lemma}
In particular, since $L(\Pi_v,\As,s)$ has no pole at $s=1$ (see \cite[Proposition p.305]{Flicker88}), the above lemma implies that $L^{imp}(\Pi_v, \As,s)$ has no pole at $s=1$ either. \\

The global \textit{imprimitive} Asai $L$-function of $\Pi$ is defined to be the product of the local Asai $L$-factors over the finite places of $F$:
$$
L^{imp}(\Pi, \As,s) = \prod_{v} L^{imp}(\Pi_v,\As,s)
$$
By the foregoing, $L^{imp}(\Pi, \As,s)$ differs from $L(\Pi,\As,s)$ only by a finite number of local factors (at places where $\Pi$ is ramified) which have no pole nor vanish at $s=1$. Consequently is a meromorphic function on some right half plane containing $s=1$, and has a simple pole at $s=1$ (we recall that $\Pi$ is a stable base change). Similarly, one can also define the \textit{imprimitive} twisted Asai $L$-function $L^{imp}(\Pi,\As^-,s)$ by replacing the primitive local factors $L(\Pi_v,\As^-,s)$ by the imprimitive local factors $L^{imp}(\Pi_v,\As^-,s)$ at the places $v$ where $\Pi$ is ramified. It is meromorphic on some right half plane containing $s=1$, and is holomorphic and nonzero at $s=1$, since $\Pi$ is a stable base change.  \\

\subsubsection{Imprimitive unitary ajdoint $L$-functions}
\label{imp_unitary_adjoint_L_funs}

The \textit{imprimitive} adjoint $L$-function $L^{imp}(\pi,\Ad,s)$ of $\pi$ and the twisted adjoint $L$-function $L^{imp}(\pi,\Ad \otimes \chi_{E/F},s)$ of $\pi$ are defined by the following formulas:
$$
L^{imp}(\Pi, \As^-,s) = L(\chi_{E/F},s) \cdot L^{imp}(\pi,\Ad,s)
$$
$$
L^{imp}(\Pi, \As^{+},s) = \zeta_{F}(s) \cdot L^{imp}(\pi,\Ad \otimes \chi_{E/F},s)
$$
Then $L^{imp}(\pi,\Ad,s)$ and $L^{imp}(\pi,\Ad \otimes \chi_{E/F},s)$ are meromorphic functions on some right half plane containing $s=1$, and are holomorphic and nonzero at $s=1$. Moreover, they differ from the primitive adjoint $L$-functions $L(\pi,\Ad,s)$ and $L(\pi,\Ad \otimes \chi_{E/F},s)$ only by a finite number of local factors at the places where $\pi$ is ramified.

Finally, assume that $\pi$ is only ramified at places of $F$ which are split in $E$. Then, as for primitive $L$-functions, the imprimitive adjoint $L$-function of $\Pi$ (see \S\ref{adjoint_L_formula}) admits the following decomposition:
\begin{equation}
\label{decomposition_ad_imp}
L^{imp}(\Pi,\Ad,s) = L^{imp}(\pi,\Ad,s) \cdot L^{imp}(\pi,\Ad \otimes \chi_{E/F},s)
\end{equation}
which is also valid for the completed imprimitive $\Lambda^{imp}$-functions. \\

\subsection{The Flicker-Rallis period and the Flicker-Rallis conjecture} 
\label{sss_flicker_rallis_conj}
In this paragraph, we remain in the general framework of an integer $n \geq 1$ and an arbitrary quadratic extension $E/F$ of number fields. We denote by $\chi_{E/F}$ the character of $\A_F^\x/F^\x$ associated with $E/F$ by class field theory.

Let $\Pi$ be a cuspidal automorphic  representation of $\GLn(\A_E)$. We say that $\Pi$ is $1$-\textbf{distinguished} (or simply \textbf{distinguished}) with respect to $\GLn(\A_F)$, if its central character is trivial on $\A_F^\x$ and if there exists a form $\phi \in \Pi$ such that the so-called \textit{Flicker-Rallis period} defined by:
\begin{equation}
\label{flicker-rallis_period}
\P^+(\phi):= \int_{Z_{n}(\mathbb{A}_F) \GLn(F) \backslash \GLn(\mathbb{A}_F)} \phi(g)  d g
\end{equation}
is non-zero. Similarly we say that $\Pi$ is $\chi_{E/F}$-\textbf{distinguished} with respect to $\GLn(\A_F)$, if its central character is equal to $\chi_{E/F}$ on $\A_F^\x$ and if there exists a form $\phi \in \Pi$ such that the twisted Flicker-Rallis period:
\begin{equation}
\label{twisted_flicker-rallis_period}
\P^-(\phi):= \int_{Z_{n}(\mathbb{A}_F) \GLn(F) \backslash \GLn(\mathbb{A}_F)} \phi(g) \chi_{E/F}(\det(g)) d g
\end{equation}
is non-zero. Flicker and Rallis conjectured (see \cite[Conjecture]{Flicker91}) that a cuspidal automorphic representation $\Pi$ of $\GLn(\A_E)$ is distinguished with respect to $\GLn(\A_F)$ if and only if it is a stable base change (when $n$ is odd) or a twisted base change (when $n$ is even) from the unitary group $U_{E/F}$. This conjecture has been proven by Rogawski \cite{ARU3} in the case $n=3$ and by Mok \cite{Mok} for general $n$.

\begin{theorem}
\label{th_mok}
Let $\Pi$ be a cuspidal automorphic representation of $\GLn(\A_E)$ and let $\w_\Pi$ denote its central character. Then the following conditions are equivalent:
\begin{enumerate}[(i)]
\item $\Pi$ is conjugate self-dual and the restriction of $\w_\Pi$ to $\A_F^\x$ is equal to $\chi_{E/F}^{(-1)^{n+1}}$ ; 
\item $\Pi$ is $\chi_{E/F}^{(-1)^{n+1}}$-distinguished with respect to $\GLn(\A_F)$;
\item $L(\Pi,\As^{(-1)^{n+1}},s)$ admits a pole at $s=1$ ;
\item $\Pi$ is a stable base change from $U_{E/F}$.
\end{enumerate}
where we set $\chi_{E/F}^+ = 1$ and $\chi_{E/F}^- = \chi_{E/F}$.
\end{theorem}

Likewise, there is a similar theorem for the twisted base change, replacing $(-1)^{n+1}$ by $(-1)^{n}$.

\subsection{The Flicker's formula}

The equivalence $(ii) \Leftrightarrow (iii)$ in \ref{th_mok} is proven by expressing the Flicker-Rallis period in terms of the residue of the Asai $L$-function. This relation, which is called the Flicker's formula, can be obtained from the integral representation of the Asai $L$-function, which was established by Asai \cite{Asai77} in the case of classical modular forms and generalized to automorphic representations of $\GLn$ by Flicker \cite{Flicker88}. In both cases, this relation is used to establish the meromorphic extension of the Asai $L$-function and its functional equation. We refer to  \cite[\S3.2]{Zh14} for a presentation and proof of the Flicker's formula in the general case.

Since we will later need to compute explicitly the Flicker-Rallis period at some explicit form $\phi$ in $\Pi$, we now give a refinement of the Flicker-Rallis formula given in \cite[Proposition 3.2]{Zh14} by computing all the local factors, including the local factors at primes where the additive character $\psi_E$ is ramified (i.e. at primes ramified in $E$) and where $\Pi$ is ramified. These ramified local factors are computed for essential vector $\phi_f = \phi_\Pi^\circ$ of $\Pi$. In order to simplify the presentation, we stick to the special case where $n$ is odd and $E$ is a quadratic field (i.e. $F=\Q$) even if a similar formula should hold for a general quadratic extension $E/F$. The only difference is that in the general case, $\psi_E$ is not only ramified above places of $F$ which are ramified in $E$, but also above the different of $F$. Moreover, a similar formula should also be true when $n$ is even by replacing $\P^+$ by $\P^-$, $L(\Pi, \As^+,s)$ by $L(\Pi, \As^-,s)$, etc... \\

\subsubsection{Settings}
\label{settings_SBC}

Let $E$ be a real quadratic field and write $E = \Q[\sqrt{d}]$, where $d>0$ is a square-free integer. The discriminant $D_E$ of $E$ is equal to $d$ if $d \equiv 1 \mod 4$ and equal to $4d$ if $d \equiv 2,3 \mod 4$. We recall that we have fixed the additive character $\psi_E =  \bigotimes_w \psi_{w} : E \bs \A_E \to \C^\x$ to be $\psi_E(x) = \psi_\Q(\mathrm{Tr}_{\A_E/\A}(\sqrt{d}x))$, and that $\psi_E$ is trivial on $\A$. Let $\pi$ be a non-endoscopic cohomological cuspidal automorphic representation of $U_E(\A)$ which is self-dual and $\Pi$ its stable base change to $\GLn(\A_E)$. The cohomological weight of $\Pi$ is denoted $\mu_E = (\mu,\mu^\vee) \in X^+(T_E)$. Let $\n \subset \O_E$ be the mirahoric level of $\Pi$. For each prime $q$, let $c_q = v_{\mathfrak{q}}(\n)$ where $\mathfrak{q} \mid q$. Note that $c_q$ is well defined even when $q$ is split in $E$, as $\Pi$ is conjugate self-dual and thus $v_\mathfrak{q}(\n) = v_{\s(\mathfrak{q})}(\n)$ for $\mathfrak{q}\mid q$. The measure $dg_q$ on $\GLn(\Q_q)$ is normalized so that the volume of $K_1(q^{c_q}) \subset \GLn(\Z_q)$ is $1$. The Haar measure $dg_q$ on $\mathrm{GL}_{n-1}(\Q_q)$ is normalized so that $\mathrm{vol}(\mathrm{GL}_{n-1}(\Z_q), dg_q) = 1$. The Haar measure $dg_\inf$ on $\GLn(\R)$ is normalized as specified in \S\ref{general_haar_measures}. The global Haar measure $dg$ on $\GLn(\A_\Q)$ is $dg = \prod_v dg_v$. The measure $dn = \prod dn_v$ on $N_n(\A_Q)$ is the Tamagawa measure. In particular $\vol(N_n(\Z_q),dn_q) = 1$ for all prime $q$. \\

\subsubsection{The completed Flicker's formula}
 
Let $\W(\Pi, \psi_E) =\otimes_v \W(\Pi_v, \psi_{v})$ be the Whittaker model of $\Pi$, 
where $\Pi = \otimes_v \Pi_v$ is decomposed as a tensor product of local representations $\Pi_v$ of $\GLn(E_v)$, where $v$ is a place of $\Q$ and $E_v = E \otimes_\Q \Q_v$. For such a place $v$ of $F$, we define the linear form $\P_v^+: \W(\Pi_v, \psi_{E,v}) \to \C$ by: 

\begin{equation}
\label{lf_whittaker}
\P^+_{v}\left(W_{v}\right)=\int_{N_{n-1}\left(F_{v}\right) \backslash \mathrm{GL}_{n-1}\left(F_{v}\right)} W_{v}\left(\begin{array}{ll}
 h & \\
& 1
\end{array}\right) d g .
\end{equation}
This integral is convergent since $\Pi_v$ is a generic irreducible admissible representation of $\GLn(E_v)$ which is distinguished with respect to $\GLn(F_v)$ \cite[Remark 2]{AM17}. Moreover, the linear form $\P_v^+$ is known to be $\GLn(F_v)$-invariant (see \cite[Proposition 5.3]{AM17} and below). We have the following proposition:

\begin{prop}
\label{full_flicker-rallis_formula}
Assume that $\Pi$ is a conjugate self-dual cuspidal automorphic representation of $\GLn(\A_E)$ which is self-dual and let $\phi_f = \phi_\Pi^\circ \in \Pi_f$ be the essential vector of $\Pi$. Let $\phi = \phi_f \otimes \phi_\inf$, for some decomposable archimedean form $\phi_\inf \in \Pi_\inf$. Then:

$$
\P^{+}(\phi)= c_{n,\inf} \cdot n \cdot (D_E)^{-n(n-1)/2} \cdot N_\pi^n \cdot \mathrm{Res}_{s=1} L^{imp}(\Pi,\As,s) \times {\P_\inf^+(W_{\phi,\inf})}
$$
where $N_\pi \in \Z_{\geq 0}$ is some integer such that $\mathrm{supp}(N_\pi) = \mathrm{supp}(N_{E/\Q}(\n))$,
$$
c = \left\{
    \begin{array}{ll}
       1 & \mbox{if } d \equiv 1 \mod 4 \\
       2^{n(n-1)/2} & \mbox{if } d \equiv 2,3 \mod 4
    \end{array}
\right.
$$
and $c_{n,\inf} \in \R^\x$ is some computable constant. In particular $c_{3,\inf} = (4\pi)^{-1}$.
\end{prop}

The rest of this subsection is devoted to the demonstration of this proposition. Let $S = S_\pi \sqcup \{ \inf \}$ be the finite set of places of $\Q$ consisting in the archimedean place and the places where $\pi$ is ramified. Let $N_\pi = \prod_{q \in S_\pi} q^{c_q}$ where $c_q = v_\mathfrak{q}(\n)$ for $\mathfrak{q}\mid q$. Note that $c_q$ is well defined even when $q$ is split in $E$, as $\Pi$ is conjugate self-dual and thus $v_\mathfrak{q}(\n) = v_{\s(\mathfrak{q})}(\n)$ for $\mathfrak{q}\mid q$. Let $\Phi \in \mathcal{S}(\A^n)$ be a Schwartz–Bruhat function which is a tensor product $\Phi = \bigotimes_v \Phi_v$ of local Schwartz–Bruhat functions $\Phi_v \in \mathcal{S}(\Q_v^n)$ with:
\begin{itemize}
\item $\Phi_q$ is the characteristic function $\mathbf{1}_{\Z_q^n}$ of $\Z_q^n$ if $q \notin S$;
\item $\Phi_q =\Phi_{q,c}$ is the characteristic function of $q^c\Z_q \x \dots \x q^c \Z_q \x (1 + q^c\Z_q)$ in $\Q_q^n$ if $q \in S_\pi$, where $c =v_q(N_\pi)$;
\item $\Phi_\inf$ is such that $\widehat{\Phi_\inf}(0) =1$.

\end{itemize}
Consequently, for any rational prime $q$ we have:
$$
\widehat{\Phi_q}(0) = \left\{
    \begin{array}{ll}
       \vol(\Z_q^n) = 1 & \mbox{if } q \notin S_\pi \\
         q^{-nc} \cdot \vol(\Z_q)^n =  q^{-nc} & \mbox{if } q \in S_\pi
    \end{array}
\right. 
$$
for our choice of additive Haar measure on $\Q_v$. Thus $\widehat{\Phi}(0) = N_\pi^{-n}$. Flicker \cite[Section 4]{J-S81-I} have shown, using the Rankin-Selberg method, that (see \cite[Proposition 3.2]{Zh14} for details):

\begin{equation}
\label{Flicker_residue}
\P^+(\phi)  =  \frac{n}{\mathrm{vol}(\Q^\x \bs \A_\Q^1)\cdot  \widehat{\Phi}(0)} \cdot \mathrm{Res}_{s=1} \left( \prod_{v \notin S}\Psi_v(s,\Phi_v,W_{\phi,v}) \right) \cdot  \prod_{v \in S } \Psi_v(1,\Phi_v,W_{\phi,v})
\end{equation}
where $\mathrm{vol}(\Q^\x \bs \A_\Q^1) = 1$ for our choice of Haar measures (see \cite[Theorem 4.11.3]{Leahy}) and the local zeta integrals are defined by:
$$
\Psi_v(s,\Phi_v,W_{\phi,v}) = \int_{N_n(\Q_v) \bs \GLn(\Q_v)} \Phi_v((0,0,\dots,1)g) W_{\phi,v}(g) |\det(g)|_v^s dg
$$
for each place $v$ of $\Q$. It thus remains to compute the local factors $\Psi_v(s,\Phi_v,W_{\phi,v})$ in (\ref{Flicker_residue}). \\

\subsubsection{Spherical places not dividing $D_E$} Let $v$ be a finite place of $\Q$ such that $\Pi_v$ is unramified (i.e. $v\notin S$). Suppose first that $v$ does not divide the discriminant $D_E$ of $E$, so that the local component $\psi_{E,v}$ of our fixed additive character $\psi_E$ is unramified, i.e. $\psi_{E,v}(\O_{w}) = 1$ and $\psi_{E,v}(\varpi_{w}^{-1}) \neq 1$, for $w \mid v$. Then, $W_{\phi,v} = W_{\Pi_v}^\circ$ is the spherical vector of $\Pi_v$ with respect to $\psi_{E,v}$. If $v$ is split in $w'$ and $w''$ in $E$, then we have $W_{\phi,v} = W_{\Pi_{w'}}^\circ \otimes W_{\Pi_{w''}}^\circ$. If $v$ is not split and $w$ is the place of $E$ above $v$, then we just have $W_{\phi,v}= W_{\Pi_w}^\circ$. Since $v$ is not ramified in $E$, we know from \cite[Proposition p.305]{Flicker88} that:
$$
\Psi_v(s,\Phi_v,W_{\phi,v}) = L(\Pi_v, \As,s)
$$

\subsubsection{Spherical places dividing $D_E$.} Suppose now that $v$ divides $D_E$ (i.e. that $v$ is ramified in $E$), but still that $\Pi_v$ is unramified. Let $w$ be the place of $E$ upon $v$. Let $\varpi_w$ (resp. $\varpi_v$) be an uniformizer of $E_w$ (resp. of $\Q_v$). If $v$ is not the even place, then $\psi_w$ has conductor $\varpi_w^{-2}$. Thus, we can write $\psi_w = \varpi_v^{-1} \cdot \psi^\circ_w$, where $\psi^\circ_w$ is an unramified additive character of $E_w$, and $(a \cdot \psi^\circ_w)(x) := \psi^\circ_w(a x)$ for $a \in E_w$. Note that $\psi^\circ_w$ is also trivial on $\Q_v$ (since $\varpi_v \in \Q_v$). We have the following lemma which computes the local zeta integral when $E_w/\Q_v$ is ramified:
\begin{lemma}
\label{spherical_ramified_factor} Assume that $E_w/\Q_v$ is ramified and let $\psi_w^\circ$ be a non-trivial unramified additive character of $E_w$ which is trivial on $\Q_v$. Then:
$$
\Psi_v(s,\Phi_v,W_{\Pi_w}^\circ) = L(\Pi_v, \As,s)
$$
where $W_{\Pi_w}^\circ \in \W(\Pi_w, \psi_w^\circ)$ is the spherical Whittaker function of the Whittaker model of $\Pi_w$ with respect to $\psi_w^\circ$.
\end{lemma}

\begin{proof}[Proof of \ref{spherical_ramified_factor}] We give a proof which is similar to the proof of Jacquet-Shalika (when $v$ is split) and Flicker (when $v$ is inert), using Shintani's formula for the diagonal values of the spherical Whittaker functions. To simplify notation, we denote by $L/K$ the extension of local fields $E_w/F_v$, and by $W$ the spherical Whittaker function $W_{\phi,v}:= W_{\Pi_w}^\circ$. Let $\varpi_K$ (resp. $\varpi_L$) be an uniformizer of $K$ (resp. $L$), and let $\d_{B_n(K)}$ (resp. $\d_{B_n(L)}$) be the modulus character of $B_n(K)$ (resp. ${B_n(L)}$). If $\l = (\l_1,\dots,\l_n) \in \Z_n$, $\varpi_K^\l \in Z_n(K)$ is defined to be $\mathrm{diag}(\varpi_K^{\l_1},\dots,\varpi_K^{\l_n})$, and $\varpi_L^\l \in Z_n(L)$ is defined similarly. Moreover if $\l$ is such that $\l_1 \geq \cdots \geq \l_n$, $s_\l(X_1,\dots,X_n)$ denotes the Schur polynomial associated to $\l$. Finally, we define $\a = (\a_i)_{i=1}^n$ to be the $n$-tuple of complex numbers $\a_i \in \C^\x$ such that the standard $L$-function of $\Pi_w$ writes:
$$
L(\Pi_w,s) = \prod_{i=1}^n (1-\a_i q^{-s})^{-1}
$$
where $q = \# (\O_L/(\varpi_L))$. Then Shintani's formula \cite[Theorem]{Shintani76} is:
$$
W_{\Pi_w}^\circ(\varpi^\l) = \left\{
    \begin{array}{ll}
        \d_{B_n(L)}^{1/2}(\varpi^\l) s_\l(\a) & \mbox{if } \l_1 \geq \cdots \geq \l_n \\
        0 & \mbox{otherwise.}
    \end{array}
\right.
$$

Recall that the Haar measures on $\GLn(K)$ and $N_n(K)$ are normalized so that $\vol(\GLn(\O_K)) = \vol(N_n(\O_K)) = 1$ (since $\Pi$ is unramified at $w$). Similarly to the proof of \cite[Proposition p.305]{Flicker88}, we have that:
$$
\begin{aligned}
\Psi_v(W, \Phi_v,s)  &=  \sum_{\l \in \Z^n} W(\varpi_K^\l) \d_{B_n(K)}^{-1}(\varpi_K^\l) \Phi_v(0,\dots,0,\varpi_K^{\l_n}) |\varpi_K^\l|_K^s \\
& =  \sum_{\l_n \geq 0} W(\varpi_L^{2\l}) \d_{B_n(K)}(\varpi_K^\l) q^{-s\Tr\l} \\
& =  \sum_{\l} \d_{B_n(L)}^{1/2}(\varpi_L^{2\l}) s_{2\l}(\a) \d_{B_n(K)}(\varpi_K^\l) q^{-s\Tr\l} \\
\end{aligned}
$$
where the last sum is indexed by the $n$-uples $\l \in \Z^n$ such that $\l_1 \geq \dots \geq \l_n \geq 0$. One then checks:
$$
\d_{B_n(L)}^{1/2}(\varpi_L^{2\l}) = q^{-\sum_{i=1}^n (n+1-2i)\l_i} = \d_{B_n(K)}(\varpi_K^{\l})
$$
Thus, using \cite[Formula 5.(a) p.77]{Macdonald}:
$$
\Psi_v(W, \Phi_v,s)  =  \sum_{\l} s_{2\l}(q^{-s/2}\a) = \prod_{i=1}^n(1-q^{-s}\a_i^2)^{-1} \prod_{j<k}^n(1-q^{-s}\a_j\a_k)^{-1} = L(\Pi_v,\As,s)
$$
\end{proof}

By definition of $\phi_f$, we have that:
$$
W_{\phi,w}(g) = W_{\Pi_w}^\circ(\mathrm{diag}(\varpi_v^{-(n-1)},\dots,\varpi_v^{-1},1)g) \in \W(\Pi_w,\psi_w)
$$
 where $W_{\Pi_w}^\circ$ is the spherical vector in $\W(\Pi_w,\psi_w^\circ)$. Thus, making change of variable, we have:
$$
\Psi_v(s,\Phi_v,W_{\phi,v})  = |\varpi_{v}|_v^{n(n-1)s/2}\Psi(s,\Phi_v,W_{\phi,v})  =  |\varpi_{v}|_v^{n(n-1)s/2} L(\Pi_v, \As,s)
$$

If $v$ is the even place, then $\psi_w$ has conductor $\varpi_w^{-4}$ or $\varpi_w^{-2}$ depending whether $d \equiv 2 \mod 4$ or $d \equiv 3 \mod 4$. Set $k=2$ in the first case and $k=1$ in the second. Then, one checks that:
$$
\Psi_v(s,\Phi_v,W_{\phi,v}) = (2^{-k})^{n(n-1)s/2} L(\Pi_v, \As,s)
$$

\subsubsection{Archimedean and ramification places} Let $v \in S$. There exists a constant $c_{n,v} \in \R$, depending on our choice of Haar measures, such that (see the formula just below (3.19) in \cite{Zh14}):
$$
\Psi_v(1,\Phi_v,W_{\phi,v})= c_{n,v} \cdot  \widehat{\Phi}_v(0) \cdot \P_v^+(W_{\phi,v})
$$
for all $W_{\phi,v} \in \W(\Pi_v,\psi_v)$ and $\Phi_v \in \mathcal{S}(\Q_v^n)$. This constant is equal to $1$ for  the choice of Haar measures made in \cite[\S 2.1]{Zh14}. The constant $c_{n,\inf}$ is computed in \cite[\S 14.11, Corollary]{AGBI}. When $n=3$, one has that $c_{3,\inf} = (4\pi)^{-1}$ (see \cite[Remark 5.2]{Che22}). Let $q \in S_\pi$ and let $c>0$ be the mirahoric conductor of $\Pi_q$. Then, by comparing \cite[Theorem 3.2]{Jo22} and \cite[Theorem 6.3]{AM17} when $v$ is non split (resp. \cite[Theorem 3.2]{Jo22} and \cite[Theorem 3.7]{Jo22} when $v$ is split) one sees that:
$$
c_{n,v} = \left\{
    \begin{array}{ll}
        \widehat{\Phi_{q,c}}(0)^{-1} & \mbox{if } q  \mbox{ does not divide } D_E \\
        \widehat{\Phi_{q,c}}(0)^{-1} q^{-n(n-1)/2} & \mbox{if } q  \mbox{ divides } D_E  \mbox{ and } q\neq 2 \\
        \widehat{\Phi_{q,c}}(0)^{-1} 2^{-n(n-1)/2} & \mbox{if } q  \mbox{ divides } D_E, \, q = 2 \mbox{ and } q  \mbox{ does not divide } d\\
        \widehat{\Phi_{q,c}}(0)^{-1} 4^{-n(n-1)/2} & \mbox{if } q  \mbox{ divides } D_E, \, q = 2 \mbox{ and } q  \mbox{ divides } d  
    \end{array}
\right. 
$$
where $\Phi_{q,c}$ is the characteristic function of $q^c\Z_q \x \dots \x q^c \Z_q \x (1 + q^c\Z_q)$ in $\Q_q^n$. Comparing of the residue of $\Psi$ and $I$ at $s=1$ and gathering the above formulas, we get:
$$
\P^{+}(\phi)=  c \cdot c_{n,\inf} \cdot n \cdot (D_E)^{-n(n-1)/2} \cdot N_\pi^n\cdot \mathrm{Res}_{s=1} L^{S}(\Pi,\As,s)\times \prod_{v \in S} \P_v^+(W_v).
$$
where $c = 1$ if $d \equiv 1 \mod 4$ and $c = 2^{n(n-1)/2}$ if $d \equiv 2,3 \mod 4$. \\

\subsubsection{Local factor at ramification places.} We now compute the local factors $\P_v^+(W_{\phi,v})$, for $v \in S_\pi$, i.e. when $\Pi_v$ is ramified. We recall that $\phi_f$ is the essential vector of $\Pi_f$, i.e if $\phi_f = \otimes_w \phi_w$, with $\phi_w = \phi_{\Pi_w}^\circ$ the essential vector of $\Pi_w$. When $v$ is not ramified in $E$, then $\psi_{E,v}$ is unramified, and:
$$
\P_v^+(W_{\phi,v}) = \P_v^+(W_{\Pi_v}^\circ)
$$
When $v$ is ramified and $w$ is the place of $E$ above $v$, we can write as above that $\psi_w = d_v \cdot \psi_w^\circ$ for some $d_v \in \Q_v^\x$ and some unramified additive character $\psi_w^\circ$ of $E_w$ which is trivial on $\Q_v$. Then, making a change of variable, we also get:
$$
\P_v^+(W_{\phi,v}) = \P_v^+(W_{\Pi_v}^\circ)
$$
where $W_{\phi,v} \in \W(\Pi_v,\psi_v)$ and $W_{\Pi_v}^\circ$ is the essential vector in $\W(\Pi_v,\psi_v^\circ)$. Consequently, we now compute the factor $\P_v^+(W_{\Pi_v}^\circ)$.  When $v$ splits in $w$ and $w'$, since $\Pi^\s = \Pi^\vee$, we have that $\Pi_{w'}= \Pi_{w}^\vee$, and we know from \cite[Lemma 3.2]{TR_BC} that:
$$
\P_v^+(W_{\Pi_v}^\circ) = \P_v^+(W_{\Pi_w}^\circ \otimes W_{\Pi_{w'}}^\circ) = \langle W^\circ_{\Pi_w}, W^\circ_{\Pi^\vee_w} \rangle_v = L^{{imp}}(\Pi_w \x \Pi^\vee_w,1) = L^{{imp}}(\Pi_v,\As,1)
$$

When there is only one place $w$ of $E$ upon $v$, then we know (see \cite[Theorem 6.1]{AM17}, see also \cite[Theorem 4.5]{Jo22}) that:
$$
\P_v^+(W_{\Pi_v}^\circ) = \P_v^+(W_{\Pi_w}^\circ) = L^{{imp}}(\Pi_v,\As,1). 
$$

\subsection{A cohomological interpretation of the Flicker-Rallis formula}
\label{proof_SBC}

We retain the notation introduced in \S\ref{settings_SBC}, and henceforth assume $n=3$.

\subsubsection{Some linear forms}
\label{linear_forms_SBC}
In order to give a cohomological interpretation of the Flicker-Rallis period, we define in this subsection some linear forms on the different cohomological objects under consideration. \\

\textit{A linear form on the cuspidal cohomology.} We first define the linear form, of main interest, on the cuspidal cohomolgy. Let:
$$
Y_\Q(K_f) = \GL(\Q) \bs \GL(\A)/K_{f,\Q}K_3 
$$
be the adelic variety of level $K_{f,\Q} = K_f \cap  \GL(\A_f)$ associated with $G_\Q$. That's a $5$-dimensionnal subvariety of $Y_E(K_f)$. Let $\Ld$ be the $\O$-linear form obtained by composing the following maps:
$$
\begin{aligned}
\Ld: H^5_{cusp}(Y_E(K_f),\L_{\mu_E}(\O)) & \inj H^5_{c}(Y_E(K_f),\L_{\mu_E}(\O)) \\
&  \to H^5_{c}(Y_\Q(K_f),\L_{\mu_E}(\O)|_{\Q}) \\
&\to H^5_{c}(Y_\Q(K_f),\L_{\mu_\tau}(\O)_{\Q} \otimes \L_{\mu_\tau}(\O)_{\Q}^\vee) \\
&\to H^5_{c}(Y_\Q(K_f),\O) \\
&\to \O
\end{aligned}
$$
The fourth map is obtained by functoriality from the $\GL(\O)$-equivariant map onto the trivial representation:
$$
L_{\mu_\tau}(\O) \otimes L_{\mu_\tau}(\O)^\vee \to \O, \quad P\otimes Q^\vee \mapsto \langle P,Q^\vee\rangle_{\mu_\tau}
$$
Here we need to assume that the weight $\mu_\tau = (n,n,v)$ is $p$-small (i.e. that $n < p$), so that the above projection preserves integrality. Note that $\Ld$ is invariant by the action of the Galois involution $\s$ on the cuspidal cohomology described in paragraph~\ref{s_involution}.  \\

\textit{A linear form on the $(\g,K_\inf)$-cohomology.} We now define a linear form:
$$
\mathfrak{L}: H^5(\g,K_\inf ;  \W(\Pi_\inf,\psi_\inf) \otimes L_{\mu_E}(\C)) \to \C
$$
on the $(\g,K_\inf)$-cohomology. Recall that we have the following expression of the relative Lie algebra cohomology groups (see \cite[II, Proposition 3.2]{BW00}):
$$
H^5(\g,K_\inf ;  \W(\Pi_\inf,\psi_\inf) \otimes L_{\mu_E}(\C))= \left(  \W(\Pi_\inf,\psi_\inf) \otimes L_{\mu_E}(\C) \otimes \bigwedge^5 \p_{\C}^* \right)^{K_\inf}
$$
We then construct the linear form as a tensor product on the right-hand side:

\begin{itemize}
\item We consider the linear form $\mathfrak{L}_{\mu_E}: L_{\mu_E}(\C) \to \C$ defined by $\mathfrak{L}_{\mu_E}(P_\tau \otimes P_{\s\tau}) = \langle P_\tau, P_{\s\tau} \rangle_{\mu_\tau}$ where $\langle \cdot,\cdot \rangle_{\mu_\tau}$ has been defined in (\ref{pairing_coefficients}) (recall that $\mu_{\tau}^\vee=\mu_{\s\tau}$)
\item We consider the linear form $\ell: \bigwedge^5 \p_{\C}^* \to \C$ constructed by functoriality of the exterior product from the map:
$$
X_\tau + Y_{\s\tau} \in \p_{\C}^*  \mapsto X_\tau + Y_{\s\tau} \in \p_{3,\C}^*
$$ 
where $\p_{\C}^* = \p_{\tau,\C}^* \oplus \p_{\s\tau,\C}^* = \p_{3,\C}^* \oplus \p_{3,\C}^*$ and identifying $\bigwedge^5\p_{3,\C}^*$ with $\C$.
\item $\P_\inf^+:  \W(\Pi_\inf,\psi_\inf) \to \C$ is the linear form defined in~(\ref{lf_whittaker}) by: 
$$
\P_\inf^+\left(W_{\inf}\right)=\int_{N_{2}\left(\R\right) \backslash \mathrm{GL}_2\left(\R\right)} W_{\inf}\left(\begin{array}{ll}
h & \\
& 1
\end{array}\right) d h .
$$
where $\inf$ designates the archimedean place of $\Q$.
\end{itemize}

Then $\mathfrak{L}$ is just defined by tensorisation of these three linear forms, i.e. it's defined on pure tensors by:
$$
\mathfrak{L}(W_\inf \otimes P \otimes \w) = \P_\inf^+(W_\inf) \times  \mathfrak{L}_{\mu_E}(P) \times \ell(\w)
$$
It is easy to check that $\mathfrak{L}$ is invariant by the action of $\s$ on the $(\g,K_\inf)$-cohomology.

\subsubsection{A cohomological interpretation of Flicker's formula}
\label{coho_SBC}
In this paragraph only, $\Pi$ is \textit{any} representation in $\mathrm{Coh}(G_E,\mu_E,K_f)$. Let $\w_{\Pi} = \w_{\Pi_f} \otimes \w_{\Pi_\inf}$ its central character. Suppose that $\w_{\Pi_f}$ is trivial on $\A_f^\x$. We give a cohomological interpretation of the Flicker-Rallis period and its link with the residue of the Asai $L$-function, using the Eichler-Shimura isomorphisms defined in paragraph~\ref{eichler-shimura_maps}. Let $\d$ denote the following comparison isomorphism:
$$
\delta: \Pi_f^{K_f} \otimes H^5(\g,K_\inf ;  \W(\Pi_\inf,\psi_\inf) \otimes L_{\mu_E}(\C)) \toeq H^5_{cusp}(Y_E(K_f),\L_{\mu_E}(\C))[\Pi_f]
$$

Let $\mathfrak{X} = \sum_{i\in I} \w_i \otimes W_i \otimes P_i$ be some element in $H^5(\g,K_\inf ;  \W(\Pi_\inf,\psi_\inf) \otimes L_{\mu_E}(\C))$ and let $\phi_f$ a $K_f$-fixed vector in $\Pi_f$. Then, similarly to \cite[paragraph 3.3.3]{BR17}, one has that:
$$
\Ld(\d(\phi_f \otimes \mathfrak{X})) = 2 h(K_{f,\Q})(1 + \w_{\Pi_\inf}(-1,-1))\sum_{i \in I} \ell(\w_i) \times  \mathfrak{L}_{\mu_E}(P_i) \times \P(\phi_i)
$$
where $\phi_i \in \Pi$ is the form $\phi_f \otimes \phi_{\inf,i}$, for $\phi_{\inf,i} \in \Pi_\inf$ corresponding to the Whittaker function $W_i$. Moreover:
$$
h(K_{f,\Q}) := \mathrm{vol}(Z(\Q) \bs Z(\A_f) / Z(\A_f) \cap K_{f,\Q}) = \mathrm{vol}(\Q^\x \bs \A_f^\x / U(N))
$$
where $N = \mathfrak{n} \cap \Z$ and  $U(N) := \prod_{q} U_q(N)$ with $U_q(N) := 1 + q^{v_q(N)} \Z_q$ for any rational prime $q$. Then $h(K_{f,\Q})$ is equal to $h_\Q(N)$ the (wide) ray class number of level $N$, which is equals to $\ph(N)/2$ if $N\neq 2$ and $\ph(2) = 1$ if $N=2$. We see from the above formula that the linear form $\Ld$ vanishes on representations $\Pi$ such that $\w_\Pi(-1,-1) = -1$. Of course if $\Pi$ is a stable base change from $U_E$, we have that $\w_{\Pi_\inf}(-1,-1) = 1$. This phenomenon is similar to what happens in the $\mathrm{GL}_2$ case, where the Hida's linear form vanishes on the $\nu$-part of the cuspidal cohomology of the Hilbert surface, when the character $\nu$ of the Weyl group $\mathbf{K}_\inf/K_\inf$ is unbalanced (see \cite[Definition 4.24]{TU22} and \cite[Section 4]{Hi99} for details, there the Weyl character is denoted by $\e$). When $n$ is odd, the action of the Weyl group $\mathbf{K}_\inf/K_\inf$ is determined by the central character of $\Pi$ (see \S \ref{eichler-shimura_maps}), hence,  in this situation, the condition on the Weyl character becomes a condition on $\w_{\Pi_\inf}$.

Assume that $\w_{\Pi_\inf}(-1,-1) = 1$. Suppose that $\phi_f$ and each $W_i$ are pure tensors. The above formula, combined with the Flicker-Rallis formula \cite[Proposition 3.2]{Zh14}, gives:
\begin{equation}
\label{linear_coho_delta}
\Ld(\d(W_f \otimes [\Pi_\inf])) =  \frac{12 h_\Q(N) \cdot \mathrm{Res}_{s=1} L^{S}(s,\Pi,\As^{+})}{\mathrm{vol}(\Q^\times \bs \A^1) \cdot\widehat{\Phi}^S(0)} \times \prod_{v \in S} \P_v^+(W_v) \times \mathfrak{L}([\Pi_\inf]) 
\end{equation}
where $S$ is some sufficiently large finite set of non-archimedean places. Using \ref{th_mok} and \ref{full_flicker-rallis_formula}, a direct consequence of the above formula is the following proposition:
\begin{prop}
\label{lf_vanishing_SBC}
Let $\Pi \in \mathrm{Coh}(G_E,\mu_E,K_f)$ such that $\w_{\Pi_f}$ is trivial on $\A_f^\x$.
\begin{itemize}
\item If $\Pi$ is not a stable base change from $U_E$, then $\Ld$ vanishes on $H^5_{cusp}(Y_E(K_f),\L_{\mu_E}(\C))[\Pi]$
\item If $\Pi$ is the stable base change of some self-dual cuspidal automorphic representation $\pi$ of $U_E(\A)$ then for the essential vector $\phi_f = \phi_\Pi^\circ \in \Pi_f$:
\begin{equation}
\label{linear_coho}
\Ld(\d_\s^\pm(\phi_f)) = c \cdot \frac{12  h_\Q(N) N_\pi^3 D_E^{-3}}{\pi} \times  L^{imp}(\pi,\Ad\otimes \chi_E,1) \times \mathfrak{L}([\Pi_\inf]_\s^\pm) 
\end{equation}
where the constants $c$ and $N_\pi$ are explicited in \ref{full_flicker-rallis_formula}. 
\end{itemize}
\end{prop}

\subsubsection{Archimedean computations}
\label{archimedean_SBC}

Let $\pi$ and $\Pi = \mathrm{SBC}(\pi)$ be as in \S\ref{settings_SBC} and we recall that $n=3$. We now compute the archimedean factor $\mathfrak{L}([\Pi_\inf]^\pm_\s)$ in (\ref{linear_coho}). We recall that $[\Pi_\inf]^\pm_\s = [\Pi_\inf]_{\{\tau\}} \pm \s([\Pi_\inf]_{\{\tau\}})$ is the explicit generator defined in \S\ref{chen_generators} and used to construct the Eichler-Shimura maps $\d^\pm_\s$. Let $\Pi_\inf = \Pi_\tau \otimes \Pi_{\s\tau}$ be the archimedean part of $\Pi$ and let $\ell$ be the minimal $\SO$-type of $\Pi_\tau$ (and $\Pi_{\s\tau}$). Let:
$$
B : H^2(\g_3,K_3, \W(\pi_\inf,\psi_\R) \otimes L_\mu(\C)) \times H^3(\g_3,K_3, \W(\pi_\inf^\vee,\psi_\R^{-1}) \otimes L_{\mu^\vee}(\C)) \to \C
$$
be the pairing defined in \S\ref{Lapid_Mao_coho}, for $F=\Q$. Since $\Pi_{\s\tau} = \Pi_\tau^\vee$, one checks that:
$$
\mathfrak{L}([\Pi_\inf]^+_\s) = 2 \mathfrak{L}([\Pi_\inf]_{\{\tau\}}) =  2 B([\Pi_\tau]_2,[\Pi_{\tau}^\vee]_3) = - 2^{\ell+6} \pi \times L(\Pi_\inf, \As^+,1)
$$
where the first equality follows from the fact that $\mathfrak{L}$ is $\s$-invariant, and the last equality is the computation in \cite[Lemmas 5.13 \& 5.14]{Che22}. Finally, the above formula combined with (\ref{linear_coho}), gives the following proposition:
\begin{prop}
\label{explicit_linear}
Let $\phi_f \in \Pi_f$ be the essential vector of $\Pi$. Assume that $p \nmid 6N_{E/\Q}(\mathfrak{n}) h_E(\mathfrak{n})D_E$. Then:

$$
\Ld(\d^+_\s(\phi_f)) \sim \Lambda^{imp}(\pi,\Ad \otimes \chi_{E},1)
$$
\end{prop}

\subsection{The main divisibility}

\subsubsection{Congruence modules with an involution}
\label{congruence_modules_with_an_involution}

We retain the notation from paragraph \S\ref{congruence_modules}. Suppose now that $\TT$ is endowed with an involution $\i : \TT \to\TT$ of $\O$-algebra. If $T\in \TT$ we will write $T^\i$ for $\i(T)$. Let $M$ be a $\TT$-module as above. We will say that $M$ is endowed with a semi-linear involution if there exists a non-trivial $\O$-linear involution, also denoted $\i$, which is semi-linear with respect to the $\TT$-action, i.e. such that $\forall T\in \TT$:
$$
\i \circ T = T^\i \circ \i 
$$
in $\End_\O(M)$. Let $\l: \TT \to \O$ be an augmentation, i.e. a surjective $\O$-algebra morphism. Suppose that $\l$ is invariant with respect to the $\i$-action on $\TT$, i.e. that $\l(T^\i) = \l(T)$, for all $T\in \TT$. In this case, the congruence module $C_\l(M)$ (see \S\ref{congruence_modules} for its definition) inherits an  $\O$-linear involutive action of $\i$. Let $C_\l(M)[\pm]$ denote the eigenspace associated to the eigenvalue $\pm 1$ of $\i$. Then, let $\eta_\l(M)[\pm]$ be the Fitting ideal of the $\O$-module $C_\l(M)[\pm]$. Since $C_\l(M) = C_\l(M)[+] \oplus C_\l(M)[-]$, we have that:
$$
\eta_\l(M) = \eta_\l(M)[+]\x \eta_\l(M)[-]
$$
Suppose now that the $\l$-rank of $M$ is $2$ and that the action of $\i$ on $M[\lambda]$ is not trivial, so that it is non-trivial on $C_\l(M)$ etiher. In this case, one checks that:
$$
C_\l(M)[\pm] = \O/\eta_\l(M)[\pm] \quad \mbox{ and } \quad \eta_\l(M)[\pm] \,\, | \,\,\eta_\l.
$$

\subsubsection{Relative congruence modules}
\label{sss_relative}
 We now switch to the relative setting and consider some surjective $\O$-algebra morphism $\theta: \TT' \to \TT$. Let $\l: \TT \to \O$ a Hecke eigensystem, and let $\l':= \l \circ \theta$ be its transfer to $\TT'$. Let $M$ be a $\TT'$-module which is finite flat over $\O$. We define $M_\TT \subset M$ to be the submodule of elements canceled out by $\Ker(\theta)$. Then, $M_\TT$ is endowed with a $\TT$-module structure. One checks that the restriction induces a map $M^*[\l'] \to M_\TT^*[\l]$, where $M^*=\mathrm{Hom}_\O(M,\O)$ is the dual module of $M$, endowed with its natural $\TT$-module structure. The \textit{relative congruence module} of $\l$ on $M$, denoted $C_\l^\#(M)$ is defined to be the cokernel of the map: 
$$
\Hom_\O(M_\TT^*[\l],\O) \to \Hom_\O(M^*[\l'],\O)
$$
Its Fitting ideal $\eta_\l^\#(M) := \mathrm{Fitt}_\O(C_\l^\#(M))$, is called the \textit{relative congruence number} of $\l$ on $M$. An equivalent definition is given in \cite[\S4.6.1]{TR_BC}. \\

Assume moreover that $\TT'$ is given with an involution denoted $\i$, and that $\l'$ is $\i$-invariant. Let $M$ be a $\TT'$-module of rank $2$ such that $M$ is given an action of $\i$ which is semi-linear. Then the action of $\i$ on $M$ induces an involutive action on the three following congruences modules:
$$
C_{\l'}(M), \quad C_\l(M_\TT) \quad \mbox{ and } \quad C_{\l}^{\#}(M)
$$
We then assume that the action of $\i$ is non-trivial on $M[\l']$, so that it is not trivial on the three congruence modules either. Then we denote by $\eta_{\l'}(M)[\pm]$, $\eta_{\l}(M_\TT)[\pm]$ and $\eta_{\l}^{\#}(M)[\pm]$ the Fitting ideals of respectively $C_\l(M)[\pm]$, $C_{\l}(M_\TT)[\pm]$ and $C_{\l}^{\#}(M)[\pm]$. We have the following lemma:
\begin{lemma}
\label{cn_decomposition}
 Let $M$ be a $\TT'$-module, finite flat over $\O$, given with a semi-linear involution $\i$. Then:
$$
\eta_{\l'}(M)[\pm] = \eta_{\l}(M_\TT)[\pm] \cdot \eta_{\l}^{\#}(M)[\pm]
$$
\end{lemma}

Moreover we have the following lemma: 

\begin{lemma}
\label{lf_lemma}
Let $M$ be $\TT'$-module, finite flat over $\O$, given with a semi-linear involution $\i$. Assume that the $\l$-rank of $M$ is $2$. Let $L \in M^*_\TT[\pm]$ be a linear form which is $\i$-$\pm$-invariant (i.e $L(\i^{-1}(m)) = \pm L(m)$ for all $m \in M$) and which is canceled out by $\Ker(\theta)$. Then for every $\d$ in $M[\lambda'][\pm]$ we have:

$$
L(\d) \in \eta^{\#}_{\lambda'}(M^*)[\pm]
$$
\end{lemma}

The proof is easily adapted to the case of congruence modules with an involution from the proof of \cite[Proposition 2.9]{TU22} (see \cite[Lemma 3.8]{thesis}). \\

\subsubsection{The main theorem}

Let $E$ be a real quadratic field and let $U_E$ be the quasi-split unitary group in $3$ variables associated to $E$. Let $\pi$ be a non-endoscopic cohomological cuspidal automorphic representation of $U_E(\A)$ of cohomological weight $\mu_U \in X^+(T_3)$ which is self-dual. Let $\Pi = \mathrm{SBC}(\pi)$ be its strong stable base change to $\GL(\A_E)$ (see \ref{sss_sbc} for details). Then $\Pi$ is a cuspidal automorphic representation of $\GL(\A_E)$ which is conjugate self-dual, and since $\pi$ is self-dual, $\Pi$ is also self-conjugate. Thus $\Pi$ is associated with two middle-degree $\s$-periods $\Om_5(\Pi,\s,\pm)$, as explained in \S\ref{s_periods}. Let $\mu_E = (\mu_U,\mu_U^\vee)$ be the cohomological weight of $\Pi$. Let $\mathfrak{n} =\mathfrak{n}(\Pi)$ be the mirahoric level of $\Pi$, and let $K_f$ denote the mirahoric subgroup $K_f = K_1(\mathfrak{n})$ of level $\mathfrak{n}$. \\

Let $\l_\Pi : \TT_E \to \O$ be the Hecke-eigensystem associated with $\Pi$. Let $M$ denote the cuspidal cohomology of $G_E$ localized at the maximal ideal $\m_\Pi$ of $h(K_f;\O)$ corresponding to $\Pi$: 
$$
M = H_{cusp}^5(Y_E(K_f), \L_{\mu_E}(\O))_{\m_\Pi}
$$
It is a $\TT_E$-module. Since $\Pi$ is self-conjugate, $M$ is equipped with a semi-linear action of the Galois involution $\s$. Let $\eta_{\l_\Pi}^\#(M^*)$ be the relative congruence numbers of $\Pi$ on the $\TT_E$-module $M^* := \Hom_\O(M,\O)$, for the stable base change map $\theta_{{SBC}}: \TT_E \to \TT_U$ described in paragraph~\ref{hecke_SBC}.

The main result of this section is the following theorem, establishing a divisibility between this congruence number and the imprimitive completed twisted adjoint $L$-function of $\pi$ (see paragraph~\ref{imp_unitary_adjoint_L_funs} for a precise definition), normalized by the base change period associated with $\Pi$:

\begin{theorem}
\label{SBC_divisibility}
Let $\pi$ be a non-endoscopic cohomological cuspidal automorphic representation of $U_E(\A)$ which self-dual, and denote by $\Pi = \mathrm{SBC}(\pi)$ its strong stable base change to $\GL(\A_E)$. Suppose that the Galois representation associated with $\Pi$ is residually absolutely irreducible. Suppose that the cohomological weight $\mu$ of $\pi$ is $p$-small and that $p \nmid 6N_{E/\Q}(\mathfrak{n}) h_E(\mathfrak{n})D_E$. Then, we have the following divisibility:
$$
\eta_{\l_\Pi}^\#(M^*)[+] \quad | \quad \frac{\Lambda^{imp}(\pi,\Ad \otimes \chi_{E},1)}{\Om_5(\Pi,\s,+)}
$$
where $\eta_{\l_\Pi}^\#(M^*)[+]$ is the $+$-part for the action of $\s$.
\end{theorem}

\begin{proof}[Proof of \ref{SBC_divisibility}]We now prove \ref{SBC_divisibility}, using the congruence number formalism introduced in the last paragraphs. We consider the $\TT_E$-module $M = H^5_{cusp}(Y_E(K_f),\L_{\mu_E}(\O))_{\m_\Pi}$, and the linear form $\Ld: M \to \O$ obtained by restricting to $M$ the linear form $\Ld$ defined in paragraph~\ref{linear_forms_SBC}. We begin by situation (I). Since $\Pi$ is self-conjugate, we have that $t \in \m_\Pi \iff t^\s \in \m_\Pi$, and the action of $\s$ on $h_{G_E}(K_f,\O)$ induces an action on $\TT_{G_E}= h_{G_E}(K_f,\O)_{\m_\Pi}$ and $\TT_E$. Then, $M$ also inherits from an action of $\s$ and this action is semi-linear, in the sense of \S\ref{congruence_modules_with_an_involution}. We have seen that $\Ld$ is naturally $\s$-invariant. \ref{lf_vanishing_SBC} implies that $\Ld$ is cancelled out by $\Ker(\theta_{SBC})$ where $\theta_{{SBC}}: \TT_E \to \TT_U$ is the stable base change transfer defined in \S\ref{hecke_SBC}. Then \ref{lf_lemma} implies:
$$
\eta_{\l_\Pi}^\#(M^*)[+] \quad | \quad \Ld \left(\frac{\d^+_\s({\phi_f})}{\Om_5(\Pi,\s,+)}\right) \sim \frac{\Lambda^{imp}(\pi,\Ad \otimes \chi_{E},1)}{\Om_5(\Pi,\s,+)}
$$
because by definition of the $\s$-periods, the cohomology class $\d^+_\s(W_{\phi_f})/\Om_5(\Pi,\s,+)$ is an $\O$-base of $M_{\l_\Pi}[+]$, and the last equality follows from \ref{explicit_linear}.  \\
\end{proof}

\subsection{A divisibility of automorphic periods for the stable base change}

We keep the same notation as in the last paragraph. The cuspidal cohomology of $U_E$ is concentrated in degrees $q=2,3$. As explained in \S\ref{unitary_automorphic_periods} of last section, one can attach two automorphic periods $\Om_2(\pi)$ and $\Om_3(\pi)$ to $\pi$, which are non-zero complex numbers defined up to an element in $\O^\x$. Let $N$ be the mirahoric level of $\pi$ and let $k_f = K_1(N) \subset U_{E}(\A_{f})$ be the mirahoric subgroup of level $N$ for $U_E$ (see \S\ref{global_unitary_newforms}). Let $H^3_\m$ denote the (torsion-free part of the) localized top-degree inner cohomology group of level $k_f$ for $U_E$ (defined in \S\ref{unitary_cusp_coho}):
$$
H^3_\m := H_{!}^3(Y_U(k_f), \L_{\mu}(\O))_{\m_\pi},
$$
endowed with its structure of $\TT_U$-module. We consider the following assumption: \\

\begin{assumption}[$H = \TT_U$]
The module $H^3_\m$ is free of rank $1$ over $\TT_U$. \\
\end{assumption}

Under this assumtion, we prove a divisibility between the automorphic periods of $\pi$ and the automorphic periods of its stable base change $\Pi$:

\begin{theorem} 
\label{divisibilite_unitaire}
Let $\pi$ be a self-dual non-endoscopic cohomological cuspidal automorphic representation of $U_E(\A)$ which is globally generic. Assume that $\pi$ is only ramified at split primes. Let $\Pi = \mathrm{SBC}(\pi)$ be its strong stable base change to $\GL(\A_E)$. Assume that the Galois representation associated with $\Pi$ is residually absolutely irreducible. Assume that the cohomological weight $\mu$ of $\pi$ is $p$-small, that $p \nmid 6N_{E/\Q}(\mathfrak{n}) h_E(\mathfrak{n})D_E$ and that $p\notin S_\partial$. Then, under the assumption $(H = \TT_U)$, we have the following divisibility:
$$
\nu_\pi \cdot \Om_2(\pi) \cdot \Om_3(\pi) \;  \mid \; \Om_5(\Pi,\s,-)
$$
where $\nu_\pi := \eta_\pi \cdot \eta_\pi(M_{\TT_\Q})[+]^{-1} \in \O$.
\end{theorem}

\begin{proof} First, from \ref{cn_decomposition}, we have a factorization of congruence numbers:
$$
\eta_{\Pi}(M^*)[\pm] \sim \eta_{\pi}(M^*_T)[\pm] \cdot \eta_{\Pi}^{\#}(M^*)[\pm] 
$$
where $M = H_{cusp}^5(Y_E(K_f), \L_{\mu_E}(\O))_{\m_\Pi}$. One has that $\eta_{\Pi}(M^*)[\pm] = \eta_{\Pi}(M)[\pm]$ and that $\eta_{\pi}(M_T^*)[\pm]$ divides the Hecke congruence number $\eta_{\pi}$ of $\pi$. Thus we have the following relation of congruence numbers:
\begin{equation}
\label{CN_equality}
\eta_{\Pi}(M)[+] \, \sim \, \nu_\pi^{-1} \cdot \eta_{\pi} \cdot \eta_{\Pi}^\#(M^*)[+]
\end{equation}
where $\nu_\pi := \eta_\pi \cdot \eta_\pi(M_{\TT_\Q})[+]^{-1}$ and $M_{\TT_\Q}$ is defined in \S\ref{sss_relative}. Then, assumption $(H = \TT_U)$ implies that $\eta_{\pi}(H^3_\m) \sim \eta_{\pi}$, where $\eta_\pi(H^3_\m)$ is the congruence number of $\pi$ on $H^3_\m$. Thus, \ref{unitary_adjoint_L_value} (for $F=\Q$) implies that, under our hypothesis on $p$, we have that:
$$
\eta_\pi \sim \frac{\Lambda^{imp}(\pi, \Ad,1)}{\Om_2(\pi) \cdot \Om_3(\pi)}
$$
Next, recall from \ref{adjoint_L_value} that we have:
$$
\eta_{\Pi}(M)[+] \sim \frac{\Lambda^{imp}(\Pi,\Ad,1)}{\Om_5(\Pi,\s,+)\cdot\Om_5(\Pi,\s,-)}
$$
From these two adjoint $L$-value formulas, from \ref{SBC_divisibility} and from (\ref{CN_equality}), we get the following divisibility:
$$
\frac{\Lambda^{imp}(\Pi,\Ad,s)}{\Om_5(\Pi,\s,+)\cdot\Om_5(\Pi,\s,-)} \, \mid \, \nu_\pi^{-1} \cdot \frac{\Lambda^{imp}(\pi, \Ad,1)}{\Om_2(\pi) \cdot \Om_3(\pi)} \x \frac{\Lambda^{imp}(\pi,\Ad \otimes \chi_E,s)}{\Om_5(\Pi,\s,+)}
$$
\ref{divisibilite_unitaire} then follows from the following factorization of completed imprimitive $L$-functions (see (\ref{decomposition_ad_imp})):
$$
\Lambda^{imp}(\Pi,\Ad,s) = \Lambda^{imp}(\pi,\Ad,s) \x \Lambda^{imp}(\pi,\Ad \otimes \chi_E,s)
$$

\end{proof}

Let us quickly discuss the plausibility of the assumption $(H = \TT_U)$. Let $\pi_{\GL}$ be a cohomological cuspidal automorphic representation of $\GL(\A_\Q)$ of mirahoric level $N$ and cohomological weight $\mu \in X^+(T_3)$. Assume that the residual Galois representation associated with $\pi_{\GL}$ is absolutely irreducible and has enormous image, that $\pi_{\GL}$ is Fontaine-Laffaille at $p$ and $N$-minimal. Then, the Calegari-Geraghty theory \cite{CG18} shows that the top degree $t=3$ cohomology group for $\GL(\Q)$ of weight $\mu$, localized at $\pi_{\GL}$, is free of rank one over the (localized) Hecke algebra $\TT_{\GL}$ of $\GL(\Q)$. We refer to \cite[\S4.7]{TR_BC} for a detailed discussion. It is plausible that Calegari-Geraghty method also applies to the quasi-split group $U_E$, and thus that $(H = \TT_U)$ holds under similar hypothesis for $\pi$. We have not had the opportunity to investigate this direction further. \\

\subsection{Results for non self-dual representations} 
\label{non_self_dual}

We keep the notation from the previous paragraph but we no longer assume that the representation $\pi$ is self-dual. Consequently its stable base change $\Pi = \mathrm{SBC}(\pi)$ is no longer self-conjugate and the $\s$-periods of $\Pi$ are not even defined. However, as a stable base change, $\Pi$ is still is conjugate self-dual. Assume that $\pi$ is only ramified at primes which are split in $E$. In this case, following \cite[\S2.4.5]{TR_BC} we can attached two middle-degree periods:
$$
\Om_5(\Pi,\e,\pm) \in \C^\x /\O^\x
$$
to the cohomological cuspidal automorphic representation $\Pi$. These automorphic periods are defined within the cuspidal cohomology group of middle degree $q=5$ and will be called the $\e$-periods of $\Pi$. \\

However, unlike what happens for the $\s$-periods, the newform $\phi_f \in \Pi_f$ used to define the $\e$-periods of $\Pi$ is not always a good test vector for the Flicker-Rallis integral period. More precisely, it annihilates the Flicker-Rallis period, unless $\Pi$ satisfies the following condition at ramification places: \\

\begin{quote}
For each prime $w$ where $\Pi$ is ramified, the standard $L$-function $L(\Pi_w,s)$ of $\Pi_w$ has degree $2$. \\
\end{quote}

For instance, this condition is satisfied if the mirahoric level $\n$ of $\Pi$ is squarefree and the representation  $\Pi$ at $\n$ is either in the principal series (induced from two unramified characters and one ramified character) or the partial Steinberg associated with the parabolic of type $(2,1)$.  \\

Nethertheless, under this quite restrictive assumption on $\Pi$, the newform $\phi_f \in \Pi_f$ becomes a good test vector for the Flicker-Rallis period, and we prove the following version of \ref{SBC_divisibility} when $\pi$ is non-necessarily self-dual:

\begin{theorem}
\label{SBC_divisibility_NSD}
Let $\pi$ and $\Pi= \mathrm{SBC}(\pi)$ be as above. Assume that the Galois representation associated with $\Pi$ is residually absolutely irreducible. Suppose that the cohomological weight $\mu$ of $\pi$ is $p$-small and that $p \nmid 6N_{E/\Q}(\mathfrak{n}) h_E(\mathfrak{n})D_E$. Then, there exists a sign $\d \in \{ \pm\} $, such that we have the following divisibility:
$$
\eta_{\l_\Pi}^\#(M^*)[\d] \quad | \quad \frac{\Lambda^{imp}(\pi,\Ad \otimes \chi_{E},1)}{\Om_5(\Pi,\e,\d)}
$$
where $\eta_{\l_\Pi}^\#(M^*)[\d]$ is the $\d$-part for the action of $\e$,
\end{theorem}

The proof of this theorem is similar to that of \ref{SBC_divisibility}. Consequently, to avoid overloading the presentation, we only state the results here and refer to \cite[Section 6]{thesis} for a complete proof. The situation where $\pi$ is not self-dual is referred to there as Situation (II), and the ramification condition on $\Pi$ is denoted \textbf{(Ram)}. The sign $\d \in \{ \pm\} $ is the sign for which the Flicker-Rallis linear form $L$ of \S\ref{linear_forms_SBC} is $\d$-invariant for the action of the conjugate self-dual involution $\e$ on $M = H^5_{cusp}(Y_E(K_f),\L_{\mu_E}(\C))_{\m_\Pi}$. It is equal to the constant $-c_\inf$ in \cite[Lemma 6.3]{thesis}.

\begin{corollaire}
Let $\pi$ and $\Pi = \mathrm{SBC}(\pi)$ be as above. Assume that the Galois representation associated with $\Pi$ is residually absolutely irreducible. Assume that the cohomological weight $\mu$ of $\pi$ is $p$-small, that $p \nmid 6N_{E/\Q}(\mathfrak{n}) h_E(\mathfrak{n})D_E$ and that $p\notin S_\partial$. Then, under the assumption $(H = \TT_U)$, we have the following divisibility:
$$
\nu_\pi \cdot \Om_2(\pi) \cdot \Om_3(\pi^\vee) \mid  \Om_5(\Pi^\vee,\e,- \d)
$$
where $\nu_\pi := \eta_\pi \cdot \eta_\pi(M_{\TT_\Q})[+]^{-1} \in \O$, and $\d \in \{\pm1\}$ is the sign of \ref{SBC_divisibility_NSD}.
\end{corollaire}

\bibliographystyle{alpha}
\bibliography{../../Notes/Bibliographie/biblio_report}

\end{document}